\documentclass{amsart}
\usepackage{amssymb}
\usepackage{amsmath}
\usepackage{url}
\usepackage{hyperref}
\usepackage{accents}
\usepackage{amscd}
\usepackage{amsmath}
\usepackage{scalefnt}
\usepackage{geometry}

 \hypersetup{colorlinks=true}
\newtheorem{theorem}{Theorem}[section]
\newtheorem{lemma}[theorem]{Lemma}

\numberwithin{equation}{section}
\newtheorem*{acknowledgements}{Acknowledgements}

\newtheorem{corollary}[theorem]{Corollary}

\newtheorem{proposition}[theorem]{Proposition}

\theoremstyle{condition}

\theoremstyle{remark}
\newtheorem{remark}[theorem]{\bf{Remark}}
\newtheorem{example}[theorem]{\bf{Example}}
\newtheorem{definition}[theorem]{\bf{Definition}}
\newtheorem{notation}[theorem]{\bf{Notation}}

\numberwithin{equation}{section}

\input{tcilatex}

\begin{document}
\title{On reductions of families of crystalline Galois representations}
\author{Gerasimos Dousmanis}
\address{Makis Dousmanis, Universit\'{e} de Lyon, UMPA ENS Lyon, 46 all\'{e}%
e d'Italie, 69007 Lyon, France}
\email{\href{mailto:makis.dousmanis@umpa.ens-lyon.fr}{%
makis.dousmanis@umpa.ens-lyon.fr}}
\maketitle

\begin{abstract}
Let $K_{f}$ be the finite unramified extension of $\mathbb{Q}_{p}$ of degree 
$f$ and $E$ any finite large enough coefficient field containing $K_{f}.$ We
construct analytic families of \'{e}tale $(\varphi ,\Gamma )$-modules which
give rise to families of crystalline $E$-representations of the absolute
Galois group $\mathrm{G}_{K_{f}}$ of $K_{f}.$ For any irreducible effective
two-dimensional crystalline $E$-representation of $G_{K_{f}}$ with labeled
Hodge-Tate weights $\{0,-k_{i}\}_{\tau _{i}}$ induced from a crystalline
character of $G_{K_{2f}},$ we construct an infinite family of crystalline $E$%
-representations of $G_{K_{f}}$ of the same Hodge-Tate type which contains
it. As an application, we compute the semisimplified $\func{mod}$ $p$
reductions of the members of each such family.
\end{abstract}

\tableofcontents

\section{Introduction}

Let $p\ $be a prime number and $\bar{\mathbb{Q}}_{p}$ a fixed algebraic
closure of $%
\mathbb{Q}
_{p}.$ Let $N$ be a positive integer and $g=\tsum\limits_{n\geq 1}a_{n}q^{n}$
a newform of weight $k\geq 2$ over $\Gamma _{1}(N)$ with character $\psi .$
The complex coefficients $a_{n}$ are algebraic over $%
\mathbb{Q}
$ and may be viewed as elements of $\bar{\mathbb{Q}}_{p}$ after fixing
embeddings\ $\bar{\mathbb{Q}}\rightarrow 
\mathbb{C}
\ $and $\bar{\mathbb{Q}}\rightarrow \bar{\mathbb{Q}}_{p}.$ By work of
Eichler-Shimura when $k=2$ and Deligne when $k>2,$ there exists a continuous
irreducible two-dimensional $p$-adic representation $\rho _{g}:G_{%
\mathbb{Q}
}\longrightarrow \mathrm{GL}_{2}(\bar{\mathbb{Q}}_{p})$ attached to $g.$ If$%
\ l\nmid pN,$ then $\rho _{g}$ is unramified at $l$ and $\det (X-\rho _{g}(%
\mathrm{Frob}_{l}))=X^{2}-a_{l}X+\psi \left( l\right) l^{k-1},$ where $%
\mathrm{Frob}_{l}$ is any choice of an arithmetic Frobenius at $l.$ The
contraction of the maximal ideal of the ring of integers of $\bar{\mathbb{Q}}%
_{p}$ via an embedding $\bar{\mathbb{Q}}\rightarrow \bar{\mathbb{Q}}_{p}$
gives rise to the choice of a place of $\bar{\mathbb{Q}}$ above $p,$ and the
decomposition group $D_{p}$ at this place is isomorphic to the local Galois
group $G_{%
\mathbb{Q}
_{p}}\ $via the same embedding. The local representation%
\begin{equation*}
\rho _{g,p}:G_{%
\mathbb{Q}
_{p}}\longrightarrow \mathrm{GL}_{2}(\bar{\mathbb{Q}}_{p}),
\end{equation*}%
obtained by restricting $\rho _{g}$ to $D_{p},$ is de Rham with Hodge-Tate
weights $\{0,$ $k-1\}$ (\cite{TS99}). If $p\nmid N$ the representation $\rho
_{g,p}$\ is crystalline and the characteristic polynomial of Frobenius of
the weakly admissible filtered $\varphi $-module$\ \mathbb{D}_{k,a_{p}}:=%
\mathbb{D}_{\mathrm{cris}}\left( \rho _{g,p}\right) $ attached to $\rho
_{g,p}$ by Fontaine is $X^{2}-a_{p}X+\psi \left( p\right) p^{k-1}\ $(\cite%
{FA89} and \cite{SC90}). The roots of Frobenius are distinct if $k=2$ and
conjecturally distinct if $k\geq 3$ (see \cite{CE98}). In this case, weak
admissibility imposes a unique up to isomorphism choice of the filtration of 
$\mathbb{D}_{k,a_{p}},$ and the isomorphism class of the crystalline
representation $\rho _{g,p}$ is completely determined by the characteristic
polynomial of Frobenius of $\mathbb{D}_{k,a_{p}}.$ The mod $p$ reduction $%
\bar{\rho}_{g,p}:G_{%
\mathbb{Q}
_{p}}\longrightarrow \mathrm{GL}_{2}(\bar{\mathbb{F}}_{p})$ of the local
representation $\rho _{g,p}\ $is well defined up to semisimplification and
plays a role in the proof of Serre's modularity conjecture, now a theorem of
Khare and Wintenberger \cite{KW09a},$\ $\cite{KW09b} which states that any
irreducible continuous odd Galois representation $\rho :G_{%
\mathbb{Q}
}\longrightarrow \mathrm{GL}_{2}(\bar{\mathbb{F}}_{p})$ is similar to a
representation of the form $\bar{\rho}_{g}$ for a certain newform $g\ $which
should occur in level $N(\rho ),$ an integer prime-to-$p$, and weight $%
\kappa (\rho )\geq 2,$ which Serre explicitly defined in \cite{SE 87}. If $%
\rho _{g,p}$ is crystalline, the semisimplified mod $p$ reduction $\bar{\rho}%
_{g,p}$ has been given concrete descriptions in certain cases by work of
Berger-Li-Zhu \cite{BLZ04} combined with work of Breuil\ \cite{BRE03}, which
extended previous results of Deligne, Fontaine, Serre and Edixhoven, and
more recently by Buzzard-Gee \cite{BG09} using the $p$-adic Langlands
correspondence for $\mathrm{GL}_{2}\left( 
\mathbb{Q}
_{p}\right) .$ For a more detailed account and the shape of these
reductions, the reader can see \cite[\S 5.2]{BE10}.

Recall that (up to unramified twist) all irreducible two-dimensional
crystalline representations of $G_{%
\mathbb{Q}
_{p}}$ with fixed Hodge-Tate weights in the range $[0;\ p]$ have the same
irreducible $\func{mod}$ $p$ reduction. Reductions of crystalline
representations of $G_{%
\mathbb{Q}
_{p^{f}}}$ with $f\neq 1,\ $where $%
\mathbb{Q}
_{p^{f}}$ is the unramified extension of $%
\mathbb{Q}
_{p}$ of degree $f,$ are more complicated. For example, in the simpler case
where $f=2,\ $there exist irreducible two-dimensional crystalline
representation of $G_{%
\mathbb{Q}
_{p^{2}}}$ with Hodge-Tate weights in the range $[0;\ p-1],$ sharing the
same characteristic polynomial and filtration, with distinct irreducible or
reducible reductions (cf. Proposition \ref{protash paradeigma}).

The purpose of this article is to extend the constructions of \cite{BLZ04}
to two-dimensional crystalline representations of $G_{%
\mathbb{Q}
_{p^{f}}}:=\mathrm{Gal}(\bar{\mathbb{Q}}_{p}/%
\mathbb{Q}
_{p^{f}}),$ and to compute the semisimplified $\func{mod}$ $p$ reductions of
the crystalline representations constructed.\ The strategy for computing
reductions is to fit irreducible representations of $G_{K_{f}}$ which are
not induced from crystalline characters of $G_{K_{2f}}$ into families of
representations of the same Hodge-Tate type and with the same mod $p$
reduction, which contain some member which is either reducible or
irreducible induced.

Serre's conjecture has been recently generalized by Buzzard, Diamond and
Jarvis \cite{BDJ} for irreducible totally odd two-dimensional $\mathbb{\bar{F%
}}_{p}$-representations of the absolute Galois group of any totally real
field unramified at $p,\ $and has subsequently been extended by Schein \cite%
{Sch08} to cases where $p$ is odd and tamely ramified in $F.$ Crystalline
representations of the absolute Galois group of finite unramified extensions
of $%
\mathbb{Q}
_{p}$ arise naturally in this context of the conjecture of Buzzard, Diamond
and Jarvis, and their modulo $p$ reductions are crucial for the weight part
of this conjecture (see \cite[\S 3]{BDJ}).

Let $F$ be a totally real number field of degree $d>1,$ and let $I=\{\tau
_{1},...,\tau _{d}\}$ be the set of real embeddings of $F.$ Let $\mathbf{k}%
=\left( k_{\tau _{1}},k_{\tau _{2}},...,k_{\tau _{d}},w\right) \in 
\mathbb{N}
_{\geq 1}^{d+1}$ with $k_{\tau _{i}}\equiv w\func{mod}\ 2.$ We denote by $%
\mathcal{O}$ the ring of integers of $F$ and we let $\mathfrak{n}\neq 0$ be
an ideal of $\mathcal{O}.$ The space $\mathrm{S}_{\mathbf{k}}(\mathrm{U}_{1}(%
\mathfrak{n}))$ of Hilbert modular cusp forms of level $\mathfrak{n}$ and
weight $\mathbf{k}$ is a finite dimensional complex vector space endowed
with actions of Hecke operators $\mathrm{T}_{\mathfrak{q}}$ indexed by
nonzero ideals $\mathfrak{q}$ of $\mathcal{O},$ and Hecke operators $\mathrm{%
S}_{\mathfrak{a}}$ indexed by ideals of $\mathcal{O}$ prime to $\mathfrak{n\ 
}$(for the precise definitions see \cite{TA89}).$\ $Let $0\neq g\in \mathrm{S%
}_{\mathbf{k}}(\mathrm{U}_{1}(\mathfrak{n}))$ be an eigenform for all the $%
\mathrm{T}_{\mathfrak{q}},$ and fix embeddings $\bar{%
\mathbb{Q}%
}\rightarrow 
\mathbb{C}
$ and $\bar{%
\mathbb{Q}%
}\rightarrow \bar{%
\mathbb{Q}%
}_{p}.$ By constructions of Rogawski-Tunnell \cite{RT83}, Ohta \cite{Oh84},
Carayol \cite{CA86}, Blasius-Rogawski \cite{BR89},\ Taylor \cite{TA89}, and
Jarvis \cite{Ja97}, one can attach to $g$ a continuous Galois representation 
$\rho _{g}:G_{F}\rightarrow \mathrm{GL}_{2}(\bar{\mathbb{Q}}_{p}),$ where $%
G_{F}$ is the absolute Galois group of the totally real field $F.$ Fixing an
isomorphism between the residue field of $\bar{\mathbb{Q}}_{p}$ with $%
\mathbb{\bar{F}}_{p},$ the $\func{mod}$ $p$ reduction $\bar{\rho}%
_{g}:G_{F}\rightarrow \mathrm{GL}_{2}(\mathbb{\bar{F}}_{p})$ is well defined
up to semisimplification. A continuous representation $\rho
:G_{F}\longrightarrow \mathrm{GL}_{2}(\mathbb{\bar{F}}_{p})$ is called
modular if $\rho \sim \bar{\rho}_{g}$ for some Hilbert modular eigenform $g.$
Conjecturally, every irreducible totally odd continuous Galois
representation $\rho :G_{F}\longrightarrow \mathrm{GL}_{2}(\mathbb{\bar{F}}%
_{p})$ is modular (\cite{BDJ}). We now assume that $k_{\tau _{i}}\geq 2$ for
all $i.$ We fix an isomorphism $\bar{%
\mathbb{Q}%
}_{p}\overset{i}{\simeq }%
\mathbb{C}
\ $and an algebraic closure $\bar{F}$ of $F.$ For each prime ideal $%
\mathfrak{p}$ of $\mathcal{O}$ lying above $p\ $we denote by $F_{\mathfrak{p}%
}$ the completion of $F$ at $\mathfrak{p},$ and we fix an algebraic closure $%
\bar{F}_{\mathfrak{p}}$ of $F_{\mathfrak{p}}$ and an $F$-embedding $\bar{F}%
\hookrightarrow \bar{F}_{\mathfrak{p}}.$ These determine a choice of a
decomposition group $D_{\mathfrak{p}}\subset G_{F}$ an isomorphism $D_{%
\mathfrak{p}}\simeq G_{F_{\mathfrak{p}}}.$ For each embedding $\tau :F_{%
\mathfrak{p}}\rightarrow \bar{\mathbb{Q}}_{p},$ let $k_{\tau }$ be the
weight of $g$ corresponding to the embedding $\tau _{\mid F}:F\rightarrow 
\bar{\mathbb{Q}}_{p}\overset{i}{\simeq }%
\mathbb{C}
.$ By works of Blasius-Rogawski \cite{BR93}, Saito \cite{SAI09}, Skinner 
\cite{SK09}, and T. Liu \cite{Liu09}, the local representation%
\begin{equation*}
\rho _{g,F_{\mathfrak{p}}}:G_{F_{\mathfrak{p}}}\longrightarrow \mathrm{GL}%
_{2}(\mathbb{\bar{Q}}_{p}),
\end{equation*}%
obtained by restricting $\rho _{g}$ to the decomposition subgroup $G_{F_{%
\mathfrak{p}}},$ is de Rham with labeled Hodge-Tate weights $\{\frac{%
k-k_{\tau }}{2},\frac{k+k_{\tau }-2}{2}\}_{\tau :F_{\mathfrak{p}}\rightarrow 
\bar{\mathbb{Q}}_{p}},$ where $k=\max \{k_{\tau _{i}}\}.$ This has also been
proved by Kisin \cite[Theorem 4.3]{Kis08}, under the assumption that $\rho
_{g,F_{\mathfrak{p}}}$ is residually irreducible. If $p$ is odd unramified
in $F$ and prime to $\mathfrak{n},$ then $\rho _{g,F_{\mathfrak{p}}}$ is
crystalline by works of Breuil \cite[Th\'{e}or\`{e}me 1(1)]{BRE99} and
Berger \cite[Th\'{e}or\`{e}me IV.2.1]{BE03}.

In the newform case, assuming that $\rho _{g,p}$ is crystalline, the weight
of $g$ and the eigenvalue of the Hecke operator $\mathrm{T}_{p}$ on $g$
completely determine the structure of the filtered $\varphi $-module $%
\mathbb{D}_{\mathrm{cris}}(\rho _{g,p}).$ In the Hilbert modular newform
case, assuming that $\rho _{g,F_{\mathfrak{p}}}$ is crystalline, the
structure of $\mathbb{D}_{\mathrm{cris}}(\rho _{g,F_{\mathfrak{p}}})$ is
more complicated and the characteristic polynomial of Frobenius and the
labeled Hodge-Tate weights do not suffice to completely determine its
structure. The filtration of $\mathbb{D}_{\mathrm{cris}}(\rho _{g,F_{%
\mathfrak{p}}})$ is generally unknown, and, even worse, the characteristic
polynomial of Frobenius and the filtration are not enough to determine the
structure of the filtered $\varphi $-module $\mathbb{D}_{\mathrm{cris}}(\rho
_{g,F_{\mathfrak{p}}}).$ In this case, the isomorphism class is (roughly)
determined by an extra parameter in $\left( \bar{%
\mathbb{Q}%
}_{p}^{\times }\right) ^{f_{\mathfrak{p}}-1}\ $(for a precise statement see
\ \cite[\S \S 6, 7]{DO10}). As a consequence, if $f_{\mathfrak{p}}\geq 2\ $%
there exist infinite families of non-isomorphic, irreducible two-dimensional
crystalline representations of $G_{%
\mathbb{Q}
_{p^{f_{\mathfrak{p}}}}}$ sharing the same characteristic polynomial and
filtration.

\begin{acknowledgements}
I thank Fred Diamond for suggesting this problem and for his feedback,
Laurent Berger for useful suggestions, and Seunghwan Chang for detailed
comments on drafts. The last parts of the paper were written during visits
at the I.H.P. and the I.H.\'{E}.S. in Spring 2010. I thank both institutions
for their hospitality and the C.N.R.S. and the S.F.B. 478 \textquotedblleft
Geometrische Strukturen in der Mathematik\textquotedblright\ of M\"{u}nster
University for financial support.
\end{acknowledgements}

\subsection{Preliminaries and statement of results\label{product ring}}

\noindent Throughout this paper $p$ will be a fixed prime number, $K_{f}=%
\mathcal{\mathbb{Q}}_{p^{f}}$ the finite unramified extension of $\mathcal{%
\mathbb{Q}}_{p}$ of degree $f,$ and $E$ a finite large enough extension of $%
K_{f}$ with maximal ideal $\mathfrak{m}_{E}\ $and residue field $k_{E}.$ We
simply write $K$ whenever the degree over $%
\mathbb{Q}
_{p}\ $plays no role. We denote by $\sigma _{K}$ the absolute Frobenius of $%
K.$ We fix once and for all an embedding $K{\overset{\tau _{0}}{%
\hookrightarrow }{E}}$ and we let $\tau _{j}=\tau _{0}\circ \sigma _{K}^{j}$
for all $j=0,1,...,f-1.$ We fix the $f$-tuple of embeddings $\mid \tau \mid
:=(\tau _{0},\tau _{1},...,\tau _{f-1})\ $and we denote $E^{\mid \tau \mid
}:=\prod\limits_{\tau :K\hookrightarrow E}E.$ The map $\xi :E\otimes
K\rightarrow E^{\mid \tau \mid }$ with $\xi _{K}(x\otimes y)=(x\tau
(y))_{\tau }$ and the embeddings ordered as above is a ring isomorphism. The
ring automorphism $1_{E}\otimes \sigma _{K}:E\otimes K\rightarrow E\otimes K$
transforms via $\xi $ to the automorphism $\mathcal{\varphi }:E^{\mid \tau
\mid }\rightarrow E^{\mid \tau \mid }$ with $\mathcal{\varphi }%
(x_{0},x_{1},...,x_{f-1})=(x_{1},...,x_{f-1},x_{0}).$ We denote by $%
e_{j}=(0,...,1,...,0)$ the idempotent of $E^{\mid \tau \mid }$ where the $1$
occurs in the $\tau _{j}$-th coordinate for each $j\in \{0,1,...,f-1\}.$
\smallskip

It is well-known (see for instance \cite[Lemme 2.2.1.1]{BM02}) that every
continuous representation $\rho :G_{K}\rightarrow \mathrm{GL}_{n}(\mathbb{%
\bar{Q}}_{p})$ is defined over some finite extension of $%
\mathbb{Q}
_{p}.$ Let $\rho :G_{K}\rightarrow \mathrm{GL}_{E}(V)$ be a continuous $E$%
-linear representation.$\ $\noindent Recall that $\mathbb{D}_{\text{\textrm{%
cris}}}(V)=(\mathbb{B}_{\mathrm{cris}}\otimes _{%
\mathbb{Q}
_{p}}V)^{G_{K}},$ where $\mathbb{B}_{\mathrm{cris}}$ is the ring constructed
by Fontaine in \cite{FO88}, is a filtered $\varphi $-module over $K$ with $E$%
-coefficients and $V$ is crystalline if and only if $\mathbb{D}_{\mathrm{cris%
}}(V)$ is free over $E\otimes K$ of rank $\dim _{E}V.$ One can easily prove
that $V$ is crystalline as an $E$-linear representation of $G_{K}$ if and
only if it is crystalline as a $%
\mathbb{Q}
_{p}$-linear representation of $G_{K}\ $(cf. \cite{CDT99}$\ $appendix$\ $B).
We may therefore extend $E$ whenever appropriate without affecting
crystallinity. By a variant of the fundamental theorem of Colmez and
Fontaine (\cite{CF00}, Th\'{e}or\`{e}me A) for nontrivial coefficients, the
functor $V\mapsto \mathbb{D}_{\mathrm{cris}}(V)$ is an equivalence of
categories from the category of crystalline $E$-linear representations of $%
G_{K}$ to the category of weakly admissible filtered $\varphi $-modules $(%
\mathbb{D},\varphi )$ over $K\ $with $E$-coefficients (see \cite{BM02}, \S %
3). Such a filtered module $\mathbb{D}$ is a module over $E\otimes K$ and
may be viewed as a module over $E^{\mid \tau \mid }$ via the ring
isomorphism $\xi \ $defined above. Its Frobenius endomorphism is bijective
and semilinear with respect to the automorphism $\varphi $ of $E^{\mid \tau
\mid }.$ For each embedding $\tau _{i}$ of $K$ into $E$ we define $\mathbb{D}%
_{i}:=e_{i}\mathbb{D}.$ We have the decomposition $\mathbb{D}%
=\bigoplus\limits_{i=0}^{f-1}\mathbb{D}_{i},$ and we filter each component $%
\mathbb{D}_{i}$ by setting $\mathrm{Fil}^{\mathrm{j}}\mathbb{D}_{i}:=e_{i}%
\mathrm{Fil}^{\mathrm{j}}\mathbb{D}.$ An integer $j$ is called a labeled
Hodge-Tate weight with respect to the embedding $\tau _{i}$ of $K\ $in $E$
if and only if $e_{i}\mathrm{Fil}^{\mathrm{-j}}\mathbb{D}\neq e_{i}\mathrm{%
Fil}^{\mathrm{-j+1}}\mathbb{D},$ and is counted with multiplicity $\dim
_{E}\left( e_{i}\mathrm{Fil}^{\mathrm{-j}}\mathbb{D}/e_{i}\mathrm{Fil}^{%
\mathrm{-j+1}}\mathbb{D}\right) .$ Since the Frobenius endomorphism of $%
\mathbb{D}$ restricts to an $E$-linear isomorphism from $\mathbb{D}_{i}$ to $%
\mathbb{D}_{i-1}$ for all $i,$ the components $\mathbb{D}_{i}$ are
equidimensional over $E.$ As a consequence, there are $n=$ $\mathrm{rank}%
_{E\otimes K}(\mathbb{D})$ labeled Hodge-Tate weights for each embedding,
counting multiplicities. The labeled Hodge-Tate weights of $\mathbb{D}$ are
by definition the $f$-tuple of multiset $(W_{i})_{\tau _{i}},$ where each
such multiset $W_{i}$ contains $n$ integers, the opposites of the jumps of
the filtration of $\mathbb{D}_{i}.$ The characteristic polynomial of a
crystalline $E$-linear representation of $G_{K}$ is the characteristic
polynomial of the $E^{\mid \tau \mid }$-linear map $\varphi ^{f},$ where $(%
\mathbb{D},\varphi )$ is the weakly admissible filtered $\varphi $-module
corresponding to it$\ $by Fontaine's functor.

\begin{definition}
A filtered $\varphi $-module $(\mathbb{D},\varphi )$\ is called
F-semisimple, non-F-semisimple, or F-scalar if the $E^{\mid \tau \mid }$%
-linear map $\varphi ^{f}$\ has the corresponding property.
\end{definition}

\noindent We may twist $\mathbb{D}$ by some appropriate rank one weakly
admissible filtered $\varphi $-module (see Proposition \ref{rank 1 isom})
and assume that $W_{i}=\{-w_{in-1}\leq ...\leq -w_{i2}\leq -w_{i1}\leq 0\}$
for all $i=0,1,...,f-1\ $for some non-negative integers $w_{ij}.$ The
Hodge-Tate weights of a crystalline representation $V$ are the opposites of
the jumps of the filtration of $\mathbb{D}_{\mathrm{cris}}(V).$ If they are
all non-positive, the crystalline representation is called effective or
positive. To avoid trivialities, throughout the paper we assume that at
least one labeled Hodge-Tate weight is strictly negative.

\begin{notation}
\label{the notation}Let $k_{i}$\ be nonnegative integers which we call
weights. Assume that after ordering them\ and omitting possibly repeated
weights we get $w_{0}<w_{1}<...<w_{t-1},$\ where $w_{0}\ $is the smallest
weight, $w_{1}$\ the second smallest weight,..., and $w_{t-1}\ $for some $%
1\leq t\leq f$\ is the largest weight. The largest weight $w_{t-1}$\ will be
usually denoted by $k.$\ For convenience we define $w_{-1}=0.\ $Let $%
I_{0}=\{0,1,...,f-1\}\ $and $I_{0}^{+}=\{i\in I_{0}:k_{i}>0\}.$\ For $%
j=1,2,...,t-1$\ we let $I_{j}=\{i\in I_{0}:k_{i}>w_{j-1}\}$\ and for $j=t$\
we define $I_{t}=\varnothing .$\ Let $f^{+}=\left\vert I_{0}^{+}\right\vert
\ $be the number of strictly positive weights.$\ $For each subset $J\subset
I_{0}$\ we write $f_{J}=\sum\limits_{i\in J}e_{_{i}}$\ and $E^{\mid \tau
_{J}\mid }=f_{J}\cdot E^{\mid \tau \mid }.\ $

\noindent We may visualize the sets $E^{\mid \tau _{I_{j}}\mid }$\ as
follows: $E^{\mid \tau _{I_{0}}\mid }$\ is the Cartesian product $E^{f}.\ $%
Starting with $E^{\mid \tau _{I_{0}}\mid },$\ we obtain $E^{\mid \tau
_{I_{1}}\mid }$\ by killing\ the coordinates where the smallest weight
occurs i.e. by killing\ the $i$-th coordinate for all $i\ $with $%
k_{i}=w_{0}. $\ We obtain $E^{\mid \tau _{I_{2}}\mid }$\ by further killing\
the coordinates where the second smallest weight $w_{1}$\ occurs and so on.

\noindent \noindent For any vector $\vec{x}\in E^{\mid \tau \mid }$\ we
denote by $x_{i}\ $its $i$-th coordinate and by $J_{\vec{x}}\ $its support $%
\{i\in I_{0}:x_{i}\neq 0\}.$\ We define as norm of $\vec{x}$\ with respect
to $\varphi \ $the vector $\mathrm{Nm}_{\varphi }(\vec{x}):=\prod%
\limits_{i=0}^{f-1}\varphi ^{i}(\vec{x})$\ and$\ $we$\ $write\ $\mathrm{v}_{%
\text{\textrm{p}}}(\mathrm{Nm}_{\varphi }(\vec{x})):=\mathrm{v}_{\text{%
\textrm{p}}}\left( \prod\limits_{i=0}^{f-1}x_{i}\right) ,$\ where $\mathrm{v}%
_{\text{\textrm{p}}}$\ is the normalized $p$-adic\ valuation of $\bar{%
\mathbb{Q}%
}_{p}.$\ If $\ell $\ is an integer we write $\vec{\ell}=(\ell ,\ell
,...,\ell )\ $and $\mathrm{v}_{\text{\textrm{p}}}(\vec{x})>\vec{\ell}$\ \
(resp. if $\mathrm{v}_{\text{\textrm{p}}}(\vec{x})\geq \vec{\ell})$\ if and
only if $\mathrm{v}_{\text{\textrm{p}}}(x_{i})>\ell $\ (resp. $v_{\text{%
\textrm{p}}}(x_{i})\geq \ell )$\ for all $i.$\ Finally, for any matrix $A\in 
$ $M_{n}(E^{\mid \tau \mid })$\ we define as its $\varphi $-norm the matrix $%
\mathrm{Nm}_{\varphi }(A):=A\varphi (A)\cdot \cdot \cdot \varphi ^{f-1}(A),$%
\ with $\varphi $\ acting on each entry of $A.$
\end{notation}

In \S \ref{the crystalline characters} we construct the effective
crystalline characters of $G_{K_{f}}.$ More precisely, for $i=0,1,...,f-1$
we construct $E$-characters $\chi _{i}$ of $G_{K_{f}}\ $with labeled
Hodge-Tate weights $-e_{i+1}=(0,...,-1,...0),\ $

\noindent with the $-1$ appearing in the $\left( i+1\right) $-place for all $%
i,\ $and we show that any crystalline $E$-character of $G_{K_{f}}$ with
labeled Hodge-Tate weights $\{-k_{i}\}_{\tau _{i}}$ can be written uniquely
in the form $\chi =\eta \cdot \chi _{0}^{k_{1}}\cdot \chi _{1}^{k_{2}}\cdot
\cdots \cdot \chi _{f-2}^{k_{f-1}}\cdot \chi _{f-1}^{k_{0}}$ for some
unramified character $\eta $ of $G_{K_{f}}.$ In the same section we prove
the following.

\begin{theorem}
\label{about induced}Let $\{\ell _{i},\ell _{i+f}\}=\{0,k_{i}\},$ where the $%
k_{i},$ $i=0,1,...,f-1$ are nonnegative integers. Let $f^{+}$ be the number
of strictly positive $k_{i}$ and assume that $f^{+}\geq 1.$

\begin{enumerate}
\item[(i)] The crystalline character$\ \noindent \chi _{\vec{\ell}}=\chi
_{0}^{\ell _{1}}\cdot \chi _{1}^{\ell _{2}}\cdot \cdots \cdot \chi
_{2f-2}^{\ell _{2f-1}}\cdot \chi _{2f-1}^{\ell _{0}}$ of $G_{K_{2f}}$ has
labeled Hodge-Tate weights $\left( -\ell _{0},-\ell _{1},...,-\ell
_{2f-1}\right) $ and does not extend to $G_{K_{f}}.$ The induced
representation $\mathrm{Ind}_{K_{2f}}^{K_{f}}\left( \chi _{\vec{\ell}%
}\right) $ is irreducible and crystalline with labeled Hodge-Tate weights $%
\{0,-k_{i}\}_{\tau _{i}}.$

\item[(ii)] Let $V$ be an irreducible two-dimensional crystalline $E$%
-representation of $G_{K_{f}}$ with labeled Hodge-Tate weights $%
\{0,-k_{i}\}_{\tau _{i}},$ whose restriction to $G_{K_{2f}}$ is reducible.
There exist an unramified character $\eta $ of $G_{K_{f}}$ and nonnegative
integers $m_{i},\ i=0,1,...,2f-1,$ with $\{m_{i},m_{i+f}\}=\{0,k_{i}\}$ for
all $i=0,1,...,f-1,\ $such that 
\begin{equation*}
V\simeq \eta \otimes \mathrm{Ind}_{K_{2f}}^{K_{f}}\left( \chi
_{0}^{m_{1}}\cdot \chi _{1}^{m_{2}}\cdot \cdots \cdot \chi
_{2f-2}^{m_{2f-1}}\cdot \chi _{2f-1}^{m_{0}}\right) .
\end{equation*}

\item[(iii)] $\mathrm{Ind}_{K_{2f}}^{K_{f}}\left( \chi _{\vec{\ell}}\right)
\simeq \mathrm{Ind}_{K_{2f}}^{K_{f}}\left( \chi _{\vec{m}}\right) $ if and
only if $\chi _{\vec{\ell}}=\chi _{\vec{m}}$ or $\chi _{\vec{\ell}}^{\sigma
}=\chi _{\vec{m}},$ where $\chi _{\vec{\ell}}^{\sigma }=\chi _{0}^{\ell
_{1}^{\prime }}\cdot \chi _{1}^{\ell _{2}^{\prime }}\cdot \cdots \cdot \chi
_{2f-2}^{\ell _{2f-1}^{\prime }}\cdot \chi _{2f-1}^{\ell _{f}^{\prime }},$
with $\ell _{i}^{\prime }=\ell _{i+f}$ and indices viewed modulo $2f.$

\item[(iv)] Up to twist by some unramified character, there exist precisely $%
2^{f^{^{+}}-1}$ distinct isomorphism classes of irreducible two-dimensional
crystalline $E$-representations of $G_{K_{f}}\ $with labeled Hodge-Tate
weights $\{0,-k_{i}\}_{\tau _{i}},$ induced from crystalline characters of $%
G_{K_{2f}}.$
\end{enumerate}
\end{theorem}

\noindent Next, we turn our attention to generically irreducible families of
two-dimensional crystalline $E$-representations of $G_{K_{f}}.$ For any
irreducible effective two-dimensional crystalline $E$-representation of $%
G_{K_{f}}$ with labeled Hodge-Tate weights $\{0,-k_{i}\}_{\tau _{i}}$ which
is induced from a crystalline character of $G_{K_{2f}},$ we construct an
infinite family of crystalline $E$-representations of $G_{K_{f}}$ of the
same Hodge-Tate type which contains it. The members of each of these
families have the same semisimplified $\func{mod}$ $p$ reductions which we
explicitly compute. \noindent

\noindent \noindent Consider the representation$\ V_{\vec{\ell}}=\mathrm{Ind}%
_{K_{2f}}^{K_{f}}\left( \chi _{0}^{\ell _{1}}\cdot \chi _{1}^{\ell
_{2}}\cdot \cdots \cdot \chi _{2f-2}^{\ell _{2f-1}}\cdot \chi _{2f-1}^{\ell
_{0}}\right) ,$ where $\{\ell _{i},\ell _{i+f}\}=\{0,k_{i}\}$ for all $%
i=0,1,...,f-1,$ and assume that at least one $k_{i}$ is strictly positive.
Theorem \ref{about induced} asserts that $V_{\vec{\ell}}$ is irreducible and
crystalline with labeled Hodge-Tate weights $\{0,-k_{i}\}_{\tau _{i}}.$ We
describe the members of the family containing $V_{\vec{\ell}}$ in terms of
their corresponding by the Colmez-Fontaine theorem weakly admissible
filtered $\varphi $-modules.

\begin{definition}
\label{o prwtos orismos}We define the following four types of matrices:\ 
\begin{equation*}
\text{t}_{1}\mathbf{:}\left( 
\begin{array}{cc}
p^{k_{i}} & 0 \\ 
X_{i} & 1%
\end{array}%
\right) ,\ \text{t}_{2}\mathbf{:}\left( 
\begin{array}{cc}
X_{i} & 1 \\ 
p^{k_{i}} & 0%
\end{array}%
\right) ,\ \text{t}_{3}\mathbf{:}\left( 
\begin{array}{cc}
1 & X_{i} \\ 
0 & p^{k_{i}}%
\end{array}%
\right) ,\ \text{t}_{4}:\left( 
\begin{array}{cc}
0 & p^{k_{i}} \\ 
1 & X_{i}%
\end{array}%
\right) ,
\end{equation*}%
where the $X_{i}\ $are indeterminates. Let $k=\max \{k_{i},\ i=0,1,...,f-1\}$
and let 
\begin{equation*}
m:=\left\{ 
\begin{array}{l}
\ \lfloor \frac{k-1}{p-1}\rfloor \ \ \ \ \text{if }k\geq p\ \text{and }%
k_{i}\neq p\ \text{for some }i, \\ 
\ \ \ \ 0\ \ \ \ \ \ \ \text{if }k\leq p-1\ \text{or }k_{i}=p\ \text{for all 
}i.%
\end{array}%
\right.
\end{equation*}%
Let $P(\overrightarrow{X})=\left( P_{1}\left( X_{1}\right) ,P_{2}\left(
X_{2}\right) ,...,P_{f}\left( X_{f}\right) \right) $ be a matrix whose
coordinates $P_{j}\left( X_{j}\right) $ are matrices of type $1,2,3$ or $4.$
To each such $f$-tuple we attach a type-vector $\vec{i}\in \{1,2,3,4\}^{f},$
where for any $j=1,2,...,f,$ the $j$-th coordinate of $\vec{i}$ is defined
to be the type of the matrix $P_{j}.\ $We write $P(\overrightarrow{X})=P^{%
\vec{i}}(\overrightarrow{X}).$ The set of all $f$-tuples of matrices of type 
$1,2,3,4$ will be denoted by $\mathcal{P}.$ There is no loss to assume that
the first $f-1$ coordinates of $P(\overrightarrow{X})$ are of type $1$ or $2$
(see Remark\ $\ref{remark about types}$) and unless otherwise stated we
always assume so. Matrices of type $t_{1}$ or $t_{3}$ are called of odd
type, while matrices of type $t_{2}$ or $t_{4}$ are called of even type.
\end{definition}

\noindent \indent For any $\vec{a}=\left( \alpha _{1},\alpha _{2},...,\alpha
_{f}\right) \in \left( p^{m}\mathfrak{m}_{E}\right) ^{f},$ let $P^{\vec{i}%
}\left( \vec{\alpha}\right) =\left( P_{1}\left( \alpha _{1}\right)
,P_{2}\left( \alpha _{2}\right) ,...,P_{f}\left( \alpha _{f}\right) \right)
\ $be the matrix obtained by evaluating each indeterminate $X_{i}$ at $%
\alpha _{i}.$ We view indices of $f$-tuples $\func{mod}$ $f,$ so $%
P_{f}=P_{0}.$ To construct the family containing $V_{\vec{\ell}},$ we choose
the types of the matrices $P_{i}$ as follows:

(1) If $\ell _{1}=0,$ $P_{1}=t_{2};$

(\noindent 2) If $\ell _{1}=k_{1}>0,$ $P_{1}=t_{1}.$ \noindent

\noindent For $i=2,3,...,f-1,\ $we choose the type of the matrix $P_{i}$ as
follows:

\noindent

(1) If $\ell _{i}=0,$ then: \noindent

\begin{itemize}
\item If an even number of coordinates of $(P_{1},P_{2},...,P_{i-1})$ is of
even type, $P_{i}=t_{2};$ \noindent

\item If an odd number of coordinates of $(P_{1},P_{2},...,P_{i-1})$ is of
even type, $P_{i}=t_{1}.$
\end{itemize}

\noindent

(\noindent 2) If $\ell _{i}=k_{i}>0,$ then: \noindent

\begin{itemize}
\item If an even number of coordinates of $(P_{1},P_{2},...,P_{i-1})$ is of
even type, $P_{i}=t_{1};$ \noindent

\item If an odd number of coordinates of $(P_{1},P_{2},...,P_{i-1})$ is of
even type, $P_{i}=t_{2}.$
\end{itemize}

\noindent Finally, we choose the type of the matrix $P_{0}$ as follows:

\noindent

(\noindent 1) If $\ell _{0}=0,$ then:

\begin{itemize}
\item \noindent If an even number of coordinates of $%
(P_{1},P_{2},...,P_{f-1})$ is of even type, $P_{0}=t_{4};$

\item If an odd number of coordinates of $(P_{1},P_{2},...,P_{f-1})$ is of
even type, $P_{0}=t_{3}.$
\end{itemize}

(\noindent 2) If $\ell _{0}=k_{i}>0,$ then:

\begin{itemize}
\item If an even number of coordinates of $(P_{1},P_{2},...,P_{f-1})$ is of
even type, $P_{0}=t_{2};$

\item If an odd number of coordinates of $(P_{1},P_{2},...,P_{f-1})$ is of
even type, $P_{0}=t_{1}.$
\end{itemize}

\noindent We define families of rank two filtered $\varphi $-modules $\left( 
\mathbb{D}_{\vec{k}}^{\vec{i}}\left( \vec{\alpha}\right) ,\varphi \right) \ $%
over $E^{\mid \tau \mid }$ by equipping%
\begin{equation*}
\mathbb{D}_{\vec{k}}^{\vec{i}}\left( \vec{\alpha}\right) =E^{\mid \tau \mid
}\eta _{1}\tbigoplus E^{\mid \tau \mid }\eta _{2}
\end{equation*}%
with the Frobenius endomorphism defined by $\left( \varphi \left( \eta
_{1}\right) ,\varphi \left( \eta _{2}\right) \right) =\left( \eta _{1},\eta
_{2}\right) P^{\vec{i}}\left( \vec{\alpha}\right) $ and the filtration 
\begin{equation}
\ \ \ \ \ \ \ \ \ \ \ \ \mathrm{Fil}^{\mathrm{j}}(\mathbb{D}_{\vec{k}}^{\vec{%
i}}\left( \vec{\alpha}\right) )=\left\{ 
\begin{array}{l}
E^{\mid \tau \mid }\eta _{1}\oplus E^{\mid \tau \mid }\eta _{2}\ \ \ \ \ \ 
\text{if\ \ \ }j\leq 0, \\ 
E^{\mid \tau _{I_{0}}\mid }\left( \vec{x}\eta _{1}+\vec{y}\eta _{2}\right) \
\ \ \text{if \ }1\leq j\leq w_{0}, \\ 
E^{\mid \tau _{I_{1}}\mid }\left( \vec{x}\eta _{1}+\vec{y}\eta _{2}\right) \
\ \ \text{if \ }1+w_{0}\leq j\leq w_{1}, \\ 
\ \ \ \ \ \ \ \ \ \ \ \ \ \ \ \ \ \ \ \ \cdots \cdots \\ 
E^{\mid \tau _{I_{t-1}}\mid }\left( \vec{x}\eta _{1}+\vec{y}\eta _{2}\right)
\ \text{if \ }1+w_{t-2}\leq j\leq w_{t-1}, \\ 
\ \ \ \ \ \ \ \ \ \ \ \ \ \ \ \ \ 0\ \ \ \ \ \ \ \ \ \text{if \ }j\geq
1+w_{t-1},%
\end{array}%
\right.  \label{ftrt}
\end{equation}%
where $\vec{x}=(x_{0},x_{1},...,x_{f-1})$ and $\vec{y}%
=(y_{0},y_{1},...,y_{f-1}),$ with 
\begin{equation}
(x_{i},y_{i})=\left\{ 
\begin{array}{l}
(1,-\alpha _{i})\text{\ if }P_{i}\ \text{has type }1\ \text{or}\ 2, \\ 
(-\alpha _{i},1)\ \text{if }P_{i}\ \text{has type }3\ \text{or }4.%
\end{array}%
\right.  \label{ftrt1}
\end{equation}

\begin{theorem}
\label{irre}Let $\vec{i}$ be the type-vector attached to the $f$-tuple $%
\left( P_{1},P_{2},...,P_{f}\right) $ defined above. For any $\vec{\alpha}%
\in \left( p^{m}\mathfrak{m}_{E}\right) ^{f},$

\begin{enumerate}
\item[(i)] The filtered $\varphi $-module $\mathbb{D}_{\vec{k}}^{\vec{i}}(%
\vec{\alpha})$ is weakly admissible and corresponds to a two-dimensional
crystalline $E$-representations $V_{\vec{k}}^{\vec{i}}(\vec{\alpha})$ of $%
G_{K_{f}}$ with labeled Hodge-Tate weights $\{0,-k_{i}\}_{\tau _{i}};$

\item[(ii)] $V_{\vec{k}}^{\vec{i}}(\vec{0})=\mathrm{Ind}_{K_{2f}}^{K_{f}}%
\left( \chi _{0}^{\ell _{1}}\cdot \chi _{1}^{\ell _{2}}\cdot \cdots \cdot
\chi _{2f-2}^{\ell _{2f-1}}\cdot \chi _{2f-1}^{\ell _{0}}\right) ;$

\item[(iii)] $\overline{V}_{\vec{k}}^{\vec{i}}\left( \vec{\alpha}\right) =%
\overline{V}_{\vec{k}}^{\vec{i}}(\vec{0});$

\item[(iv)] $\left( \overline{V}_{\vec{k}}^{\vec{i}}\left( \vec{\alpha}%
\right) _{\mid I_{K_{f}}}\right) ^{s.s.}=\omega _{2f,\bar{\tau}_{0}}^{\beta
}\oplus \omega _{2f,\bar{\tau}_{0}}^{p^{f}\beta },$ where $\beta
=-\tsum\limits_{i=0}^{2f-1}p^{i}\ell _{i};$

\item[(v)] The residual representation $\overline{V}_{\vec{k}}^{\vec{i}%
}\left( \vec{\alpha}\right) $ is irreducible if and only if $1+p^{f}\nmid
\beta ;$

\item[(vi)] Any irreducible member of the family $\left\{ V_{\vec{k}}^{\vec{i%
}}\left( \vec{\alpha}\right) ,\ \vec{\alpha}\in \left( p^{m}\mathfrak{m}%
_{E}\right) ^{f}\right\} ,$ other than $V_{\vec{k}}^{\vec{i}}(\vec{0}),$ is
non-induced.
\end{enumerate}
\end{theorem}

\noindent Notice that if $1+p^{f}\nmid \tsum\limits_{i=0}^{2f-1}p^{i}\ell
_{i},$ all the members of the family $\left\{ V_{\vec{k}}^{\vec{i}}\left( 
\vec{\alpha}\right) ,\ \vec{\alpha}\in \left( p^{m}\mathfrak{m}_{E}\right)
^{f}\right\} $ are forced to be irreducible. Next, we compute the
semisimplified reduction of any reducible two-dimensional crystalline $E$%
-representation of $G_{K_{f}}.$ After enlarging $E$ if necessary, any
reducible rank two weakly admissible filtered $\varphi $-module $\mathbb{D}$
over $E^{\mid \tau \mid }$ with labeled Hodge-Tate weights $%
\{0,-k_{i}\}_{\tau _{i}}\ $contains an ordered basis $\underline{\eta }%
=(\eta _{1},\eta _{2})$ in which the matrix of Frobenius takes the form $%
\mathrm{Mat}_{\underline{\eta }}(\varphi )=\left( 
\begin{array}{cc}
\vec{\alpha} & \vec{0} \\ 
\vec{\ast} & \vec{\delta}%
\end{array}%
\right) ,$ such that $\mathbb{D}_{2}=\left( E^{\mid \tau \mid }\right) \eta
_{2}$ is a $\varphi $-stable weakly admissible submodule (see Proposition %
\ref{reducible section prop}). The filtration of $\mathbb{D}$ in such a
basis $\underline{\eta }$ has the form 
\begin{equation*}
\mathrm{Fil}^{\mathrm{j}}(\mathbb{D})=\left\{ 
\begin{array}{l}
E^{\mid \tau \mid }\eta _{1}\oplus E^{\mid \tau \mid }\eta _{2}\ \ \ \ \ \ 
\text{if }j\leq 0, \\ 
E^{\mid \tau _{I_{0}}\mid }\left( \vec{x}\eta _{1}+\vec{y}\eta _{2}\right) \
\ \text{if }1\leq j\leq w_{0}, \\ 
E^{\mid \tau _{I_{1}}\mid }\left( \vec{x}\eta _{1}+\vec{y}\eta _{2}\right) \
\ \text{if }1+w_{0}\leq j\leq w_{1}, \\ 
\ \ \ \ \ \ \ \ \ \ \ \ \ \ \ \ \ \ \ \ \cdots \cdots \\ 
E^{\mid \tau _{I_{t-1}}\mid }\left( \vec{x}\eta _{1}+\vec{y}\eta _{2}\right)
\ \text{if }1+w_{t-2}\leq j\leq w_{t-1}, \\ 
\ \ \ \ \ \ \ \ \ \ 0\ \ \ \ \ \ \ \ \ \ \ \ \ \ \ \ \text{if }j\geq
1+w_{t-1},%
\end{array}%
\right.
\end{equation*}%
for some vectors $\vec{x},\vec{y}\in E^{\mid \tau \mid }$ with $\left(
x_{i},y_{i}\right) \neq \left( 0,0\right) $ for all $i.$ For each $i\in
I_{0},$ let 
\begin{equation*}
m_{i}=\left\{ 
\begin{array}{l}
0\ \ \text{if }x_{i}\neq 0~, \\ 
k_{i}\ \text{if }x_{i}=0.%
\end{array}%
\right.
\end{equation*}

\begin{theorem}
Let $V$ be any reducible two-dimensional crystalline $E$-representation of $%
G_{K_{f}}$ with labeled Hodge-Tate weights $\{0,-k_{i}\}_{\tau _{i}}$
corresponding to the weakly admissible filtered $\varphi $-module $\mathbb{D}
$ as above.

\begin{enumerate}
\item[(i)] There exist unramified characters $\eta _{i}$ of $G_{K_{f}}$ such
that%
\begin{equation*}
\ V\simeq \left( 
\begin{array}{cc}
\psi _{1} & \boldsymbol{\ast } \\ 
0 & \psi _{2}%
\end{array}%
\right) ,\ 
\end{equation*}%
where $\psi _{1}=\eta _{1}\cdot \chi _{0}^{m_{1}}\cdot \cdots \cdot \chi
_{f-2}^{m_{f-1}}\cdot \chi _{f-1}^{m_{0}}$ and $\psi _{2}=\eta _{2}\cdot
\chi _{0}^{k_{1}-m_{1}}\cdot \chi _{1}^{k_{2}-m_{2}}\cdot \cdots \cdot \chi
_{f-2}^{k_{f-1}-m_{f-1}}\cdot \chi _{f-1}^{k_{0}-m_{0}};$

\item[(ii)] $\left( \overline{V}_{\mid I_{K}}\right) ^{s.s.}=\omega _{f,\bar{%
\tau}_{0}}^{\beta _{1}}\oplus \omega _{f,\bar{\tau}_{0}}^{\beta _{2}},\ $%
where $\beta _{1}=-\tsum\limits_{i=0}^{f-1}m_{i}p^{i}$ and $\beta
_{2}=\tsum\limits_{i=0}^{f-1}\left( m_{i}-k_{i}\right) p^{i}.$
\end{enumerate}
\end{theorem}

\noindent The computation of the semisimplified $\func{mod}$ $p$ reduction
of a reducible two-dimensional crystalline representation is easy and does
not require the construction of the Wach module corresponding to some $%
G_{K_{f}}$-stable lattice contained in it. Computing the non-semisimplified $%
\func{mod}$ $p$ reduction of a two-dimensional crystalline representations
with reducible reduction is an interesting problem not pursued in this
paper. For results of this flavour for $K=%
\mathbb{Q}
_{p^{2}},$ the reader can see \cite{CD09}.

Up to twist by some unramified character, any split-reducible
two-dimensional crystalline $E$-representations of $G_{K_{f}}\ $with labeled
Hodge-Tate weights $\{0,-k_{i}\}_{\tau _{i}}$ is of the form 
\begin{equation*}
V_{\vec{\ell},\vec{\ell}^{\prime }}\left( \eta \right) =\eta \cdot \chi
_{0}^{\ell _{1}}\cdot \chi _{1}^{\ell _{2}}\cdot \cdots \cdot \chi
_{f-2}^{\ell _{f-1}}\cdot \chi _{f-1}^{\ell _{0}}\oplus \chi _{0}^{\ell
_{1}^{\prime }}\cdot \chi _{1}^{\ell _{2}^{\prime }}\cdot \cdots \cdot \chi
_{f-2}^{\ell _{f-1}^{\prime }}\cdot \chi _{f-1}^{\ell _{0}^{\prime }},
\end{equation*}%
for some unramified character $\eta $ and some nonnegative integers $\ell
_{i}$ and $\ell _{i}^{\prime }$ such that $\{\ell _{i},\ell _{i}^{\prime
}\}=\{0,k_{i}\}$ for all $i.$ In Theorem \ref{irre} we showed that each
irreducible representation of $G_{K_{f}}$ induced from some crystalline
character of $G_{K_{2f}}$ belongs to an infinite family of crystalline
representations of the same Hodge-Tate types with the same $\func{mod}$ $p$
reductions. In the next theorem we prove the same for any split-reducible,
non-ordinary two-dimensional crystalline $E$-representation of $G_{K_{f}}.$
We list the weakly admissible filtered $\varphi $-modules corresponding to
these families. In order to construct the infinite family containing $V_{%
\vec{\ell},\vec{\ell}^{\prime }}\left( \eta \right) ,$ we define a matrix $%
P^{\vec{i}}(\overrightarrow{X})\in \mathcal{P}$ by choosing the $\left(
f-1\right) $-tuple $\left( P_{1},P_{2},...,P_{f-1}\right) $ as in Theorem %
\ref{irre}. If $\eta =\eta _{c}$ is the unramified character which maps $%
\mathrm{Frob}_{K_{f}}$ (geometric Frobenius) to $c,$ we replace the entry $%
p^{k_{0}}$ in the definition of the matrix $P_{0}$ by $cp^{k_{0}}.$ The type
of the matrix $P_{0}$ is chosen as follows: \noindent

(1) If $\ell _{0}=0,$ then: \noindent

\begin{itemize}
\item If an even number of coordinates of $\left(
P_{1},P_{2},...,P_{f-1}\right) $ is of even type, $P_{0}=t_{3};$

\item If an odd number of coordinates of $\left(
P_{1},P_{2},...,P_{f-1}\right) $ is of even type, $P_{0}=t_{4}.$
\end{itemize}

\noindent

(2) If $\ell _{0}=k_{0}>0,$ then:

\begin{itemize}
\item If an even number of coordinates of $\left(
P_{1},P_{2},...,P_{f-1}\right) $ is of even type, $P_{0}=t_{1};$

\item If an odd number of coordinates of $\left(
P_{1},P_{2},...,P_{f-1}\right) $ is of even type, $P_{0}=t_{2}.$
\end{itemize}

\noindent We define families of two-dimensional crystalline $E$%
-representations $\left\{ V_{\vec{k}}^{\vec{i}}\left( \vec{\alpha}\right) ,\ 
\vec{\alpha}\in \left( p^{m}\mathfrak{m}_{E}\right) ^{f}\right\} $ of $%
G_{K_{f}}$ as in Theorem \ref{irre}. We prove the following.

\pagebreak

\begin{theorem}
\label{reducibles0}Let $\vec{i}$ be the type-vector attached to the $f$%
-tuple $\left( P_{1},P_{2},...,P_{f}\right) $ defined above.

\begin{enumerate}
\item[(i)] There exists some unramified character $\mu \ $such that $\mu
\otimes V_{\vec{k}}^{\vec{i}}(\vec{0})\simeq V_{\vec{\ell},\vec{\ell}%
^{\prime }}(\eta );$

\item[(ii)] Assume that $\vec{\ell}\neq \vec{0}\ $and $\vec{\ell}^{\prime
}\neq \vec{0}.$ For any $\vec{\alpha}\in \left( p^{m}\mathfrak{m}_{E}\right)
^{f},\ \overline{V}_{\vec{k}}^{\vec{i}}(\vec{\alpha})\simeq \overline{V}_{%
\vec{k}}^{\vec{i}}(\vec{0});$

\item[(iii)] $\overline{V}_{\vec{\ell},\vec{\ell}^{\prime }}(\eta )_{\mid
I_{K_{f}}}=\omega _{f,\bar{\tau}_{0}}^{\beta }\oplus \omega _{f,\bar{\tau}%
_{0}}^{\beta ^{\prime }},$ where $\beta =-\tsum\limits_{i=0}^{f-1}\ell
_{i}p^{i}$ and $\beta ^{\prime }=-\tsum\limits_{i=0}^{f-1}\ell _{i}^{\prime
}p^{i}.$
\end{enumerate}
\end{theorem}

\noindent A family as in Theorem \ref{reducibles0} can contain
simultaneously split and non-split reducible, as well as irreducible
crystalline representations. For example, in the family $\left\{ V_{\vec{k}%
}^{\left( 1,3\right) }(\vec{\alpha}),\ \vec{\alpha}\in \left( p^{m}\mathfrak{%
m}_{E}\right) ^{2}\right\} ,$ the representation $V_{\vec{k}}^{\left(
1,3\right) }(\vec{\alpha})$ is split-reducible if and only if $\vec{\alpha}=%
\vec{0},$ non-split-reducible if and only if precisely one of the
coordinates $\alpha _{i}$ of $\vec{\alpha}$ is zero, and irreducible if and
only if $\alpha _{0}\alpha _{1}\neq 0\ $(cf. Proposition \ref{an example}).
The families of Wach modules which give rise to $V_{\vec{k}}^{\left(
1,3\right) }(\vec{\alpha})$ contain infinite sub-families of non-split
reducible Wach modules which can be used to compute the non-semisimplified $%
\func{mod}$ $p$ reduction of the corresponding crystalline representations
with respect to $G_{K_{f}}$-stable $\mathcal{O}_{E}$-lattices. Some
reducible two-dimensional crystalline representations with labeled
Hodge-Tate weights $\{0,-k_{i}\}_{\tau _{i}}$ are easily recognized by
looking at their trace of Frobenius. More precisely, if $\mathrm{Tr}\left(
\varphi ^{f}\right) \in \mathcal{O}_{E}^{\times },$ then the representation
is reducible (cf. Proposition \ref{existance of inf families}), with the
converse being false.

\section{Overview of the theory \label{Fontaine copy(1)}}

\subsection{\'{E}tale $\left( \protect\varphi ,\Gamma \right) $-modules and
Wach modules}

The general theory of $\left( \varphi ,\Gamma \right) $-modules works for
arbitrary finite extensions $K$ of $%
\mathbb{Q}
_{p}.$ However, a theory of Wach modules currently exists only when $K$ is
unramified over $%
\mathbb{Q}
_{p}.$ Here we temporarily allow $K$ to be any finite extension of $%
\mathbb{Q}
_{p};$ we will go back to assume that $K$ is unramified after Theorem \ref%
{fontaine|s equivalence}. Let $K_{n}=K(\zeta _{p^{n}}),\ $where $\zeta
_{p^{n}}\ $is a primitive $\ p^{n}$-th$\ $root\ of\ unity$\ $inside$\ \bar{%
\mathbb{Q}%
}_{p},$ and let $K_{\infty }=\cup _{n\geq 1}K_{n}.$ Let $\chi
:G_{K}\rightarrow 
\mathbb{Z}
_{p}^{\times }$ be the cyclotomic character. We denote $H_{K}=\ker \chi =%
\mathrm{Gal}(\bar{%
\mathbb{Q}%
}_{p}/K_{\infty })$\ and $\Gamma _{K}=G_{K}/H_{K}=\mathrm{Gal}(K_{\infty
}/K).$ Fontaine (\cite{FO90}) has constructed topological rings $\mathbb{A}$
and $\mathbb{B}$ endowed with continuous commuting Frobenius $\varphi $ and $%
G_{%
\mathbb{Q}
_{p}}$-actions. Let $\mathbb{A}_{K}=\mathbb{A}^{H_{K}}$ and $\mathbb{B}_{K}=%
\mathbb{B}^{H_{K}}.\ $We define $\mathbb{A}_{K\,,E}:=\mathcal{O}_{E}\otimes
_{%
\mathbb{Z}
_{p}}\mathbb{A}_{K}$ and $\mathbb{B}_{K\,,E}:=E\otimes _{%
\mathbb{Q}
_{p}}\mathbb{B}_{K}.$ The actions of $\varphi $ and $\Gamma _{K}$ extend to $%
\mathbb{A}_{K,E}\ $and $\mathbb{B}_{K\,,E}$ by $\mathcal{O}_{E}$ (resp. $E,$ 
$k_{E}$)-linearity. One sees that $\mathbb{A}_{K,E}=\mathbb{A}_{E}^{H_{K}}$
and $\mathbb{B}_{K,E}=\mathbb{B}_{E}^{H_{K}}.$

\begin{definition}
\label{Wach}A $(\varphi ,\Gamma )$-module over $\mathbb{A}_{K,E}$\ (resp. $%
\mathbb{B}_{K,E}$) is an $\mathbb{A}_{K,E}$-module of finite type (resp. $%
\mathbb{B}_{K,E}$-module,\ free of finite type) with continuous (for the
weak topology) commuting semilinear actions of $\varphi $\ and $\Gamma
_{K}.\ $A $(\varphi ,\Gamma )$-module $M$\ over $\mathbb{A}_{K,E}$\ is
called \'{e}tale if it is free and $\varphi ^{\ast }(M)=M,$\ where $\varphi
^{\ast }(M)$\ is the $\mathbb{A}_{K,E}$-module generated by the set $\varphi
(M).$\ A $(\varphi ,\Gamma )$-module $M$\ over $\mathbb{B}_{K,E}$\ is called 
\'{e}tale if it contains a basis $(e_{1},...,e_{d})$\ over $\mathbb{B}_{K,E}$%
\ such that $\left( \varphi (e_{1}),...,\varphi (e_{d})\right)
=(e_{1},...,e_{d})A$\ for some matrix $A\in \mathrm{GL}_{d}\left( \mathbb{A}%
_{K,E}\right) .$
\end{definition}

\noindent If $V$ is an $E$-linear representation of $G_{K},$ let $\mathbb{D}%
(V):=\left( \mathbb{B}_{E}\otimes _{E}V\right) ^{H_{K}}.$ The $\mathbb{B}%
_{K,E}$-module $\mathbb{D}(V)$ is equipped with a Frobenius $\varphi $
defined by $\varphi (b\otimes v):=\varphi (b)\otimes v,$ where $\varphi $ in
the right hand side is the Frobenius of $\mathbb{B}_{E}$ and a commuting
with $\varphi $ action of $\Gamma _{K}$ given by$\ \bar{g}(b\otimes
v):=gb\otimes gv$ for any $g\in G_{K}.$ Fontaine has proved that $\mathbb{D}%
(V)$ is an \'{e}tale $\left( \varphi ,\Gamma \right) $-module over $\mathbb{B%
}_{K,E}.$ Conversely, if $D$ is an \'{e}tale $\left( \varphi ,\Gamma \right) 
$-module, let $\mathbb{V}(D):=\left( \mathbb{B}_{E}\otimes _{\mathbb{B}%
_{K,E}}D\right) ^{\varphi =1},$ where $\varphi (b\otimes d):=\varphi
(b)\otimes \varphi (d).$ The $E$-vector space $\mathbb{V}(D)$ is equipped
with a $G_{K}$-action given by $g(b\otimes d):=gb\otimes \bar{g}d.$ We have
the following fundamental theorem of Fontaine.

\begin{theorem}
\cite{FO90}\label{fontaine|s equivalence}

\begin{enumerate}
\item[(i)] There is an equivalence of categories between $E$-linear
representations of $G_{K}\ $and \'{e}tale $(\varphi ,\Gamma )$-modules over $%
\mathbb{B}_{K,E}$ given by%
\begin{equation*}
\mathbb{D}:\mathrm{Rep}_{E}\left( G_{K}\right) \rightarrow \mathcal{M}%
od_{(\varphi ,\Gamma )}^{\ \acute{e}t}\left( \mathbb{B}_{K,E}\right)
:V\longmapsto \mathbb{D}(V):=\left( \mathbb{B}_{E}\otimes _{E}V\right)
^{H_{K}},
\end{equation*}%
with quasi-inverse functor 
\begin{equation*}
\mathbb{V}:\mathcal{M}od_{(\varphi ,\Gamma )}^{\ \acute{e}t}\left( \mathbb{B}%
_{K,E}\right) \rightarrow \mathrm{Rep}_{E}\left( G_{K}\right) :D\longmapsto 
\mathbb{V}(D):=\left( \mathbb{B}_{E}\otimes _{\mathbb{B}_{K,E}}D\right)
^{\varphi =1}.
\end{equation*}

\item[(ii)] There is an equivalence of categories between $\mathcal{O}_{E}$%
-linear representations of $G_{K}\ $and \'{e}tale $(\varphi ,\Gamma )$%
-modules over $\mathbb{A}_{K,E}$ given by 
\begin{equation*}
\mathbb{D}:\mathrm{Rep}_{\mathcal{O}_{E}}\left( G_{K}\right) \rightarrow 
\mathcal{M}od_{(\varphi ,\Gamma )}^{\ \acute{e}t}\left( \mathbb{A}%
_{K,E}\right) :\mathrm{T}\longmapsto \mathbb{D}(T):=\left( \mathbb{A}%
_{E}\otimes _{\mathcal{O}_{E}}\mathrm{T}\right) ^{H_{K}},
\end{equation*}%
with quasi-inverse functor 
\begin{equation*}
\mathbb{T}:\mathcal{M}od_{(\varphi ,\Gamma )}^{\ \acute{e}t}\left( \mathbb{A}%
_{K,E}\right) \rightarrow \mathrm{Rep}_{\mathcal{O}_{E}}\left( G_{K}\right)
:D\longmapsto \mathbb{T}(D):=\left( \mathbb{A}_{E}\otimes _{\mathbb{A}%
_{K,E}}D\right) ^{\varphi =1}.\noindent
\end{equation*}
\end{enumerate}
\end{theorem}

\noindent We return to assume that $K$ is unramified over $%
\mathbb{Q}
_{p}.$ Now $\mathbb{A}_{K}$ has the form%
\begin{equation*}
\mathbb{A}_{K}=\{\tsum\limits_{n=-\infty }^{\infty }\alpha _{n}\pi
_{K}^{n}:\alpha _{n}\in \mathcal{O}_{K}\ \mathrm{and}\underset{n\rightarrow
-\infty }{\lim }\alpha _{n}=0\}
\end{equation*}%
for some element $\pi _{K}\ $which can be thought of as a formal variable.
The Frobenius endomorphism $\varphi $ extends the absolute Frobenius of $%
\mathcal{O}_{K}$ and is such that $\varphi (\pi _{K})=(1+\pi _{K})^{p}-1.$
The $\Gamma _{K}$-action is $\mathcal{O}_{K}$-linear, commutes with
Frobenius , and is such that $\gamma (\pi _{K})=(1+\pi _{K})^{\chi (\gamma
)}-1$ for all $\gamma \in \Gamma _{K}.\ $For simplicity we write $\pi $
instead of $\pi _{K}.$ The ring $\mathbb{A}_{K}$ is local with maximal ideal 
$(p),$ fraction field $\mathbb{B}_{K}=\mathbb{A}_{K}[\frac{1}{p}],$ and
residue field $\mathbb{E}_{K}:=k_{K}((\pi )),$ where $k_{K}$ is the residue
field of $K.$

\noindent The rings $\mathbb{A}_{K},\ \mathbb{A}_{K,E},\ \mathbb{B}_{K}$ and 
$\mathbb{B}_{K,E}$ contain the subrings $\mathbb{A}_{K}^{+}=\mathcal{O}%
_{K}[[\pi ]]$\noindent $,\ \mathbb{A}_{K,E}^{+}:=$ $\mathcal{O}_{E}\otimes _{%
\mathbb{Z}_{p}}\mathbb{A}_{K}^{+},\ \mathbb{B}_{K}^{+}=$ $\mathbb{A}_{K}^{+}[%
\frac{1}{p}]\ $and $\mathbb{B}_{K,E}^{+}:=$ $E\otimes _{%
\mathbb{Q}
_{p}}\mathbb{B}_{K}^{+}$ respectively which are equipped with the
restrictions of the $\varphi $ and the $\Gamma _{K}$-actions. There is a
ring isomorphism%
\begin{equation}
\xi :\mathbb{A}_{K,E}^{+}\rightarrow \tprod\limits_{\tau :K\hookrightarrow E}%
\mathcal{O}_{E}[[\pi ]]  \label{ksi}
\end{equation}%
given by%
\begin{equation*}
\xi \left( a\otimes b\right) =\left( a\tau _{0}\left( b\right) ,a\tau
_{1}\left( b\right) ,...,a\tau _{f-1}\left( b\right) \right) ,
\end{equation*}%
where%
\begin{equation*}
\tau _{i}\left( \tsum\limits_{n=0}^{\infty }\beta _{n}\pi ^{n}\right)
=\tsum\limits_{n=0}^{\infty }\tau _{i}\left( \beta _{n}\right) \pi ^{n}
\end{equation*}%
for all $b=\tsum\limits_{n=0}^{\infty }\beta _{n}\pi ^{n}\in \mathbb{A}%
_{K}^{+}.$ The ring\ $\mathcal{O}_{E}[[\pi ]]^{\mid \tau \mid
}:=\prod\limits_{\tau :K\hookrightarrow E}\mathcal{O}_{E}[[\pi ]]$ is
equipped with $\mathcal{O}_{E}$-linear actions of $\varphi $ and $\Gamma
_{K} $ given by$\ \ \ \ \ \ \ \ \ \ \ \ \ $%
\begin{align}
& \varphi (\alpha _{0}(\pi ),\alpha _{1}(\pi ),...,\alpha _{f-1}(\pi
))=(\alpha _{1}(\varphi (\pi )),...,\alpha _{f-1}(\varphi (\pi )),\alpha
_{0}(\varphi (\pi )))\noindent  \label{actions1} \\
& \mathrm{and}\text{\ }\gamma (\alpha _{0}(\pi ),\alpha _{1}(\pi
),...,\alpha _{f-1}(\pi ))=(\alpha _{0}(\gamma \pi ),\alpha _{1}(\gamma \pi
),...,\alpha _{f-1}(\gamma \pi ))  \label{actions2}
\end{align}%
for all $\gamma \in \Gamma _{K}.$

\begin{definition}
Suppose $k\geq 0.$\ A Wach module over $\mathbb{A}_{K,E}^{+}$\ (resp. $%
\mathbb{B}_{K,E}^{+}$) with weights in $[-k;\ 0]$\ is a free $\mathbb{A}%
_{K,E}^{+}$-module (resp. $\mathbb{B}_{K,E}^{+}$-module) $N$\ of finite
rank, endowed with an action of $\Gamma _{K}$\ which becomes trivial modulo $%
\pi $, and also with a Frobenius map $\varphi $\ which commutes with the
action of $\Gamma _{K}$\ and such that $\varphi (N)\subset N$\ and $%
N/\varphi ^{\ast }(N)$\ is killed by $q^{k},\ $where $q:=\varphi (\pi )/\pi
. $
\end{definition}

\noindent A natural question is to determine the types of \'{e}tale $%
(\varphi ,\Gamma )$-modules which correspond to crystalline representations
via Fontaine's functor. An answer is given by the following theorem of
Berger who built on previous work of Wach \cite{WA96}, \cite{WA97} and
Colmez \cite{CO99}.

\begin{theorem}
\cite{BE03} \label{berger thm}

\begin{enumerate}
\item[(i)] An $E$-linear representation $V$ of $G_{K}$ is crystalline with
Hodge-Tate weights in $[-k;\ 0]$ if and only if $\mathbb{D}(V)$ contains a
unique Wach module $\mathbb{N}(V)$ of rank $\dim _{E}V$ with weights in $%
[-k;\ 0].$ The functor $V\mapsto \mathbb{N}(V)$ defines an equivalence of
categories between crystalline representations of $G_{K}$ and Wach modules
over $\mathbb{B}_{K,E}^{+},$ compatible with tensor products, duality and
exact sequences. \noindent

\item[(ii)] For a given crystalline $E$-representation $V,$ the map $\mathrm{%
T}\mapsto \mathbb{N}(\mathrm{T}):=\mathbb{N}(V)\cap \mathbb{D}(\mathrm{T})$
induces a bijection between $G_{K}$-stable, $\mathcal{O}_{E}$-lattices of $V$
and Wach modules over $\mathbb{A}_{K,E}^{+}$ which are $\mathbb{A}_{K,E}^{+}$%
-lattices contained in $\mathbb{N}(V).$ Moreover $\mathbb{D}(\mathrm{T})=%
\mathbb{A}_{K,E}\otimes _{\mathbb{A}_{K,E}^{+}}\mathbb{N}(\mathrm{T})$.

\item[(iii)] If $V$ is a crystalline $E$-representation of $G_{K},$ and if
we endow $\mathbb{N}(V)$ with the filtration $\mathrm{Fil}^{\mathrm{i}}%
\mathbb{N}(V)=\{x\in \mathbb{N}(V)|\varphi (x)\in q^{i}\mathbb{N}(V)\}$,
then we have an isomorphism 
\begin{equation*}
\mathbb{D}_{\mathrm{cris}}(V)\rightarrow \mathbb{N}(V)/\pi \mathbb{N}(V)
\end{equation*}%
of filtered $\varphi $-modules over $E^{\mid \tau \mid }$ (with the induced
filtration on $\mathbb{N}(V)/\pi \mathbb{N}(V)$).
\end{enumerate}
\end{theorem}

\noindent In view of Theorems \ref{fontaine|s equivalence} and \ref{berger
thm}, constructing the Wach module $\mathbb{N}(T)\ $of a $G_{K}$-stable $%
\mathcal{O}_{E}$-lattice $\mathrm{T}$ in a crystalline representation $V$
amounts to explicitly constructing the crystalline representation. Indeed,
we have 
\begin{equation*}
V\simeq E\otimes _{\mathcal{O}_{E}}\left( \mathbb{A}_{K,E}\otimes _{\mathbb{A%
}_{K,E}^{+}}\mathbb{N}(\mathrm{T})\right) ^{\varphi =1}.
\end{equation*}%
An obvious advantage of using Wach modules is that instead of working with
the more complicated rings $\mathbb{A}_{K,E}$ and $\mathbb{B}_{K,E},$ one
works with the simpler ones $\mathbb{A}_{K,E}^{+}$ and $\mathbb{B}%
_{K,E}^{+}. $

\subsection{Wach modules of restricted representations}

In this section we relate the Wach module of an effective $n$-dimensional
effective crystalline $E$-representation $V_{K_{f}}$ of $G_{K_{f}},$ to the
Wach module of its restriction $V_{K_{df}}$ to $G_{K_{df}}.$

\begin{proposition}

\begin{enumerate}
\item[(i)] \label{restriction prop}The Wach module associated to the
representation $V_{K_{df}}$ is given by%
\begin{equation*}
\mathbb{N}(V_{K_{df}})=\mathbb{B}_{K_{df},E}^{+}\otimes _{\mathbb{B}%
_{K_{f},E}^{+}}\mathbb{N}(V_{K_{f}}),
\end{equation*}%
where $\mathbb{N}(V_{K_{f}})$ is the Wach module associated to $V_{K_{f}}.$

\item[(ii)] If $\mathrm{T}_{K_{f}}$ is a $G_{K_{f}}$-stable $\mathcal{O}_{E}$%
-lattice in $V_{f}$ associated to the Wach-module $\mathbb{N}(\mathrm{T}%
_{K_{f}}),$ then $V_{df}$ contains some $G_{K_{df}}$-stable $\mathcal{O}_{E}$%
-lattice $\mathrm{T}_{K_{df}}$ whose associated Wach module is 
\begin{equation*}
\mathbb{N}(\mathrm{T}_{K_{df}})=\mathbb{A}_{K_{df},E}^{+}\otimes _{\mathbb{A}%
_{K_{f},E}^{+}}\mathbb{N}(\mathrm{T}_{K_{f}}).
\end{equation*}
\end{enumerate}
\end{proposition}

\begin{proof}
(i) Since $\mathbb{N}(V_{K_{f}})$ is a free $\mathbb{B}_{K_{f},E}^{+}$%
-module of rank $\dim _{E}V$ which is contained in $\mathbb{D}(V_{K_{f}}),$ $%
N:=\mathbb{B}_{K_{df},E}^{+}\otimes _{\mathbb{B}_{K_{f},E}^{+}}\mathbb{N}%
(V_{K_{f}})$ is a free $\mathbb{B}_{K_{df},E}^{+}$-module of rank $\dim
_{E}V $ contained in $\mathbb{D}(V_{K_{df}})\supseteq \mathbb{D}%
(V_{K_{f}}).~ $Moreover, it is endowed with an action of $\Gamma _{K_{df}}$
which becomes trivial modulo $\pi ,$ and also with a Frobenius map $\varphi $
which commutes with the action of $\Gamma _{K_{df}}$ and such that $\varphi
(N)\subset N$ and $N/\varphi ^{\ast }(N)$ is killed by $q^{k}.$ Hence $N%
\mathbb{=N}(V_{K_{df}})$ by the uniqueness part of Theorem \ref{berger thm}%
(i). Part (ii) follows immediately from Theorem \ref{berger thm}(ii) since $%
\mathbb{A}_{K_{df},E}^{+}\otimes _{\mathbb{A}_{K_{f},E}^{+}}\mathbb{N}(%
\mathrm{T}_{K_{f}})$ is an $\mathbb{A}_{K_{df},E}^{+}$-lattice in $\mathbb{N}%
(V_{K_{df}}).$
\end{proof}

\noindent We fix once and for all an embedding $\tau _{K_{df}}^{0}:K_{df}{%
\hookrightarrow {E}}$ and we let $\tau _{K_{df}}^{j}=\tau _{K_{df}}^{0}\circ
\sigma _{K_{df}}^{j}$ for $j=0,1,...,df-1,\ $where $\sigma _{K_{df}}\ $is
the absolute Frobenius of $K_{df}.$ We fix the $df$-tuple of embeddings $%
\mid \tau _{K_{df}}\mid :=(\tau _{K_{df}}^{0},\tau _{K_{df}}^{1},...,\tau
_{K_{df}}^{df-1}).$ We adjust the notation of \S \ref{product ring} for the
embeddings of $K_{f}$ into $E\ $to the relative situation considered in this
section. Let $\iota \ $be the natural inclusion of $K_{f}\ $into $K_{df},$
in the sense that $\iota \circ \sigma _{K_{f}}=\sigma _{K_{df}}\circ \iota ,$
where $\sigma _{K_{f}}$ is the absolute Frobenius of $K_{f}.$ This induces a
natural inclusion of $\mathbb{A}_{K}^{+}$ to $\mathbb{A}_{K_{df}}^{+}$ which
we also denote by $\iota .$ Let\ $\tau _{K_{f}}^{j}:=\tau _{K_{df}}^{0}\circ
\iota \circ \sigma _{K_{f}}^{j}$ for $j=0,1,...,f-1.$ We fix the $f$-tuple
of embeddings $\mid \tau _{K_{f}}\mid :=(\tau _{K_{f}}^{0},\tau
_{K_{f}}^{1},...,\tau _{K_{f}}^{f-1}).$ Since the restriction of $\sigma
_{K_{df}}$ to $K_{f}$ is $\sigma _{K_{f}},\ $we obtain the following
commutative diagram%
\begin{equation*}
\begin{CD} \mathbb{A}_{K_{f},E}^{+} @>{\operatorname{\xi_{K_{f}}}}>>
\mathcal{O} _{E}^{\mid\tau_{K_{f}}\mid}[[\pi]]\\
@V{\operatorname{1_{\mathcal{O}_{E}}\otimes \iota}}VV
@VV{\operatorname{\theta} }V \\ {\mathbb{A}_{K_{df},E}^{+}}
@>{\operatorname{\xi_{K_{df}}}}>> {\mathcal{O}
_{E}^{\mid\tau_{K_{df}}\mid}[[\pi]]} \end{CD}
\end{equation*}%
where $\theta $ is the ring homomorphism defined by\noindent 
\begin{equation*}
\theta (\alpha _{0},\alpha _{1},...,\alpha _{f-1})=\underset{d\text{-times}}{%
\underbrace{\left( \alpha _{0},\alpha _{1},...,\alpha _{f-1},\alpha
_{0},\alpha _{1},...,\alpha _{f-1},...,\alpha _{0},\alpha _{1},...,\alpha
_{f-1}\right) }}=:(\alpha _{0},\alpha _{1},...,\alpha _{f-1})^{\otimes d}.
\end{equation*}

\noindent For any matrix $A\in M_{n}\left( \mathcal{O}_{E}^{\mid \tau
_{K_{f}}\mid }[[\pi ]]\right) $ we denote by $A^{\otimes d}$ the matrix
obtained by replacing each entry $\vec{\alpha}$ of $A$ by $\vec{\alpha}%
^{\otimes d}.$ A similar commutative diagram is obtained by replacing $%
\mathbb{A}_{K}^{+}$ by $\mathbb{B}_{K}^{+}$ and $\mathcal{O}_{E}^{\mid \tau
_{K}\mid }[[\pi ]]$ by $\mathcal{O}_{E}^{\mid \tau _{K}\mid }[[\pi ]][\frac{1%
}{p}].$ The following proposition follows easily from the discussion above.

\begin{proposition}
\label{comment} Let $V_{K_{f}},$ $V_{K_{df}},$ $\mathrm{T}_{K_{f}},$ and $%
\mathrm{T}_{K_{df}}$ be as in Proposition $\ref{restriction prop}.$

\begin{enumerate}
\item[(i)] If the Wach module $\mathbb{N}(V_{K_{f}})$ of $V_{K_{f}}$ is
defined by the actions of $\varphi $ and $\Gamma _{K_{f}}$ given by
\noindent $(\varphi (\eta _{1}),\varphi (\eta _{2}),...,\varphi (\eta _{n}))=%
\underline{\eta }\cdot \Pi _{K_{f}}\ $and$\ (\gamma (\eta _{1}),\gamma (\eta
_{2}),...,\gamma (\eta _{n}))=\underline{\eta }\cdot G_{K_{f}}^{\gamma }$
\noindent for all $\gamma \in \Gamma _{K_{f}}\ $for some ordered basis $%
\underline{\eta }=(\eta _{1},\eta _{2},...,\eta _{n}),$ then the Wach module 
$\mathbb{N}(V_{K_{df}})\ $of $V_{K_{df}}$ is defined by \noindent $(\varphi
(\eta _{1}^{\prime }),\varphi (\eta _{2}^{\prime }),...,\varphi (\eta
_{n}^{\prime }))=\underline{\eta }^{\prime }\cdot \Pi _{K_{df}}\ $and $%
(\gamma (\eta _{1}^{\prime }),\gamma (\eta _{2}^{\prime }),...,\gamma (\eta
_{n}^{\prime }))=\underline{\eta }^{\prime }\cdot G_{K_{df}}^{\gamma }$%
\noindent\ for all $\gamma \in \Gamma _{K_{df}},$ where $\Pi
_{K_{df}}=\left( \Pi _{K_{f}}\right) ^{\otimes d}$ and $G_{K_{df}}^{\gamma
}=\left( G_{K_{f}}^{\gamma }\right) ^{\otimes d}$ for all $\gamma \in \Gamma
_{K_{df}},$ for some ordered basis $\underline{\eta }^{\prime }$ of $\mathbb{%
N}(V_{K_{df}}).$

\item[(ii)] If the Wach module$\ \mathbb{N}(\mathrm{T}_{K_{f}})$ of $\mathrm{%
T}_{K_{f}}\ $is defined by the actions of $\varphi $ and $\Gamma _{K_{f}}$
given by \noindent $(\varphi (\eta _{1}),\varphi (\eta _{2}),...,\varphi
(\eta _{n}))=\underline{\eta }\cdot \Pi _{K_{f}}\ $and$\ (\gamma (\eta
_{1}),\gamma (\eta _{2}),...,\gamma (\eta _{n}))=\underline{\eta }\cdot
G_{K_{f}}^{\gamma }$ \noindent for all $\gamma \in \Gamma _{K_{f}}\ $for
some ordered basis $\underline{\eta }=(\eta _{1},\eta _{2},...,\eta _{n}),$
then the Wach module $\mathbb{N}(\mathrm{T}_{K_{df}})$ of $\mathrm{T}%
_{K_{df}}$ is defined by \noindent $(\varphi (\eta _{1}^{\prime }),\varphi
(\eta _{2}^{\prime }),...,\varphi (\eta _{n}^{\prime }))=\underline{\eta }%
^{\prime }\cdot \Pi _{K_{df}}\ $and $(\gamma (\eta _{1}^{\prime }),\gamma
(\eta _{2}^{\prime }),...,\gamma (\eta _{n}^{\prime }))=\underline{\eta }%
^{\prime }\cdot G_{K_{df}}^{\gamma }$\noindent\ for all $\gamma \in \Gamma
_{K_{df}},$ where $\Pi _{K_{df}}=\left( \Pi _{K_{f}}\right) ^{\otimes d}$
and $G_{K_{df}}^{\gamma }=\left( G_{K_{f}}^{\gamma }\right) ^{\otimes d}$
for all $\gamma \in \Gamma _{K_{df}},$ for some ordered basis $\underline{%
\eta }^{\prime }$ of $\mathbb{N}(V_{K_{df}}).$
\end{enumerate}
\end{proposition}

\begin{corollary}
\label{weights of restrictions}If $V_{K_{f}}$ is a two-dimensional effective
crystalline $E$-representation of $G_{K_{f}}$ with labeled Hodge-Tate
weights $\left( \{0,-k_{i}\}\right) _{\tau _{i}},$ $i=0,1,...,f-1,$ then $%
V_{K_{df}}$ is an effective crystalline $E$-representation of $G_{K_{df}}$
with labeled Hodge-Tate weights $\left( \{0,-k_{i}\}\right) _{\tau _{i}},$ $%
i=0,1,...,df-1,$ with $k_{j}=k_{j}$ for all $i,j=0,1,...,df-1\ $with $%
i\equiv j\func{mod}f.$
\end{corollary}

\begin{proof}
By Proposition \ref{comment} there exist ordered bases $\underline{\eta }$
and $\underline{\eta }^{\prime }$ of $\mathbb{N}(V_{K_{f}})$ and $\mathbb{N}%
(V_{K_{df}})$ respectively, such that $\varphi \left( \underline{\eta }%
\right) =\underline{\eta }\cdot \Pi _{K_{f}},$ $\gamma \left( \underline{%
\eta }\right) =\underline{\eta }\cdot G_{K_{f}}^{\gamma }$ for all $\gamma
\in \Gamma _{K_{f}}$ and $\varphi \left( \underline{\eta ^{\prime }}\right) =%
\underline{\eta }^{\prime }\cdot \left( \Pi _{K_{f}}\right) ^{\otimes d},$ $%
\gamma \left( \underline{\eta }^{\prime }\right) =\underline{\eta }^{\prime
}\cdot \left( G_{K_{f}}^{\gamma }\right) ^{\otimes d}$ for all $\gamma \in
\Gamma _{K_{df}}.$ By Theorem \ref{berger thm}, $x\in \mathrm{Fil}^{\mathrm{j%
}}\left( \mathbb{N}(V_{K_{f}})\right) $ if and only if $\varphi \left(
x\right) \in q^{j}\mathbb{N}(V_{K_{f}}),$ from which it follows that $%
\mathrm{Fil}^{\mathrm{j}}\left( \mathbb{N}(V_{K_{df}})\right) =\left( 
\mathrm{Fil}^{\mathrm{j}}\left( \mathbb{N}(V_{K_{f}})\right) \right)
^{\otimes d}$ for all $j.$ By Theorem \ref{berger thm}, $\mathbb{D}%
(V_{K_{f}})\simeq \mathbb{N}(V_{K_{f}})/\pi \mathbb{N}(V_{K_{f}})$ as
filtered $\varphi $-modules over $E^{\mid \tau _{K_{f}}\mid }.$ This implies
that $\mathrm{Fil}^{\mathrm{j}}\left( \mathbb{D}(V_{K_{df}})\right) =\left( 
\mathrm{Fil}^{\mathrm{j}}\left( \mathbb{D}(V_{K_{f}})\right) \right)
^{\otimes d}$ for all $j$ and the corollary follows.
\end{proof}

\section{Effective Wach modules of rank one\label{the crystalline characters}%
}

\noindent In this section we construct the rank one Wach modules over $%
\mathcal{O}_{E}[[\pi]]^{\mid\tau\mid}$ with labeled Hodge-Tate weights $%
\{-k_{i}\}_{\tau_{i}}.$

\begin{definition}
\textrm{\noindent \label{q_n definition}}\textit{Recall that }$q=\frac{%
\varphi \left( \pi \right) }{\pi }$\textit{\ where }$\varphi \left( \pi
\right) =(1+\pi )^{p}-1.$ \textit{We define }$q_{1}=q$\textit{\ and }$%
q_{n}=\varphi ^{n-1}\left( q\right) \ $for all $n\geq 1.$\textit{\ Let }$%
\lambda _{f}=\prod\limits_{n=0}^{\infty }\left( \frac{q_{nf+1}}{p}\right) .$%
\textit{\ \noindent For each }$\gamma \in \Gamma _{K},$\textit{\ we define }$%
\lambda _{f},_{\gamma }=\frac{\lambda _{f}}{\gamma \lambda _{f}}.$
\end{definition}

\begin{lemma}
\label{proto}For each $\gamma \in \Gamma _{K},$ the functions $\lambda
_{f},\ \lambda _{f,\gamma }$ $\in 
\mathbb{Q}
_{p}[[\pi ]]$ have the following properties:

\begin{enumerate}
\item[(i)] $\lambda _{f}(0)=1;$

\item[(ii)] $\lambda _{f,\gamma }\in 1+\pi 
\mathbb{Z}
_{p}\left[ \left[ \pi \right] \right] .$
\end{enumerate}
\end{lemma}

\begin{proof}
(i) This is clear since $\frac{q_{n}(0)}{p}=1$ for all $n\geq 1.$ (ii) One
can easily check that $\frac{q}{\gamma q}\in 1+\pi 
\mathbb{Z}
_{p}\left[ \left[ \pi \right] \right] .$ From this we deduce that $\lambda
_{f,\gamma }\in 1+\pi 
\mathbb{Z}
_{p}\left[ \left[ \pi \right] \right] .$\noindent
\end{proof}

\noindent Consider the rank one module $\mathbb{N}_{\vec{k},c}=\left( 
\mathcal{O}_{E}[[\pi ]]^{\mid \tau \mid }\right) \eta ,$ equipped with
semilinear actions of $\varphi $ and $\Gamma _{K}$ defined by $\varphi (\eta
)=(c\cdot q^{k_{1}},q^{k_{2}},...,q^{k_{f-1}},q^{k_{0}})\eta \ \noindent $and%
$\ \gamma (\eta )$\noindent $=(g_{1}^{\gamma }(\pi ),g_{2}^{\gamma }(\pi
),...g_{f-1}^{\gamma }(\pi ),g_{0}^{\gamma }(\pi ))\eta $ for all $\gamma
\in \Gamma _{K},$ where $c\in \mathcal{O}_{E}^{\times }.$ We need to define
the functions $g_{i}(\pi )=g_{i}^{\gamma }(\pi )\in \mathcal{O}_{E}[[\pi ]]$
appropriately to make $\mathbb{N}_{\vec{k},c}$ a Wach module over $\mathcal{O%
}_{E}[[\pi ]]^{\mid \tau \mid }.$ The actions of $\varphi $ and $\gamma $
should commute and a short computation shows that $g_{0}$ should satisfy the
equation 
\begin{equation}
\varphi ^{f}(g_{0})=g_{0}\left( \frac{\gamma q}{q}\right) ^{k_{0}}\varphi (%
\frac{\gamma q}{q})^{k_{1}}\cdots \varphi ^{f-1}(\frac{\gamma q}{q}%
)^{k_{f-1}}.  \label{equation solve}
\end{equation}

\begin{lemma}
\label{solve}Equation $\ref{equation solve}$ has a unique $\equiv 1\func{mod}%
\pi $ solution in $%
\mathbb{Z}
_{p}[[\pi ]]$ given by 
\begin{equation*}
g_{0}=\lambda _{f,\gamma }^{k_{0}}\varphi (\lambda _{f,\gamma
})^{k_{1}}\varphi ^{2}(\lambda _{f,\gamma })^{k_{2}}\cdots \varphi
^{f-1}(\lambda _{f,\gamma })^{k_{f-1}}.
\end{equation*}
\end{lemma}

\begin{proof}
Notice that $\varphi ^{f}(\lambda _{f})=\frac{\lambda _{f}}{(\frac{q}{p})}$
and $\varphi ^{f}(\gamma \lambda _{f})=\frac{\gamma \lambda _{f}}{(\frac{%
\gamma q}{p})},$ hence $\lambda _{f,\gamma }=\frac{\lambda _{f}}{\gamma
\lambda _{f}}$ solves the equation $\varphi ^{f}(u)=u\left( \frac{\gamma q}{q%
}\right) .\ $It is straightforward to check that 
\begin{equation*}
g_{0}=\lambda _{f,\gamma }^{k_{0}}\varphi (\lambda _{f,\gamma
})^{k_{1}}\varphi ^{2}(\lambda _{f,\gamma })^{k_{2}}\cdots \varphi
^{f-1}(\lambda _{f,\gamma })^{k_{f-1}}
\end{equation*}%
is a solution of equation \ref{equation solve}. By Lemma \ref{proto}, $%
g_{0}\equiv 1\func{mod}\pi .$ If $g_{0}\ $and $g_{0}^{\prime }\ $are two
congruent to $1\func{mod}\pi \ $solutions of equation \ref{equation solve},$%
\ $then $(\frac{g_{0}^{\prime }}{g_{0}})\in 
\mathbb{Z}
_{p}[[\pi ]]\ $is fixed by $\varphi ^{f}$ and $\smallskip $is congruent$\ $%
to $1\func{mod}\pi ,\ $ hence equals $1.\ $
\end{proof}

\noindent Let $g_{0}$ be as in Lemma \ref{solve}. Commutativity of $\varphi $
with the$\ \Gamma _{K}$-actions implies that%
\begin{gather*}
g_{1}=(\frac{q}{\gamma q})^{k_{1}}\varphi (\frac{q}{\gamma q})^{k_{2}}\cdots
\varphi ^{f-2}(\frac{q}{\gamma q})^{k_{f-1}}\varphi ^{f-1}(\lambda
_{f,\gamma })^{k_{0}}\varphi ^{f}(\lambda _{f,\gamma })^{k_{1}}\cdots
\varphi ^{2f-2}(\lambda _{f,\gamma })^{k_{f-1}}, \\
\cdots \cdots \\
g_{f-2}=(\frac{q}{\gamma q})^{k_{f-2}}\varphi (\frac{q}{\gamma q}%
)^{k_{f-1}}\varphi ^{2}(\lambda _{f,\gamma })^{k_{0}}\varphi ^{3}(\lambda
_{f,\gamma })^{k_{1}}\cdots \varphi ^{f+1}(\lambda _{f,\gamma })^{k_{f-1}},\ 
\\
g_{f-1}=(\frac{q}{\gamma q})^{k_{f-1}}\varphi (\lambda _{f,\gamma
})^{k_{0}}\varphi ^{2}(\lambda _{f,\gamma })^{k_{1}}\varphi ^{3}(\lambda
_{f,\gamma })^{k_{2}}\cdots \varphi ^{f}(\lambda _{f,\gamma
})^{k_{f-1}}.\smallskip \smallskip
\end{gather*}%
Thus, by Lemma \ref{proto} we have that$\ g_{i}$ $\equiv 1\func{mod}\pi $
for all $i.$

\begin{proposition}
\label{The crystalline characters}We equip $\mathbb{N}_{\vec{k},c}=\left( 
\mathcal{O}_{E}[[\pi ]]^{\mid \tau \mid }\right) \eta $ with semilinear $%
\varphi $ and $\Gamma _{K}$-actions defined by $\varphi (\eta )=(c\cdot
q^{k_{1}},q^{k_{2}},...,q^{k_{f-1}},q^{k_{0}})\eta \ $and$\ \noindent \gamma
(\eta )=(g_{1}^{\gamma }(\pi ),g_{2}^{\gamma }(\pi ),...g_{f-1}^{\gamma
}(\pi ),g_{0}^{\gamma }(\pi ))\eta $ for the $g_{i}(\pi )=g_{i}^{\gamma
}(\pi )$ defined above, where $c\in \mathcal{O}_{E}^{\times }.$ The module $%
\mathbb{N}_{\vec{k},c}$ is a Wach module over $\mathcal{O}_{E}[[\pi ]]^{\mid
\tau \mid }$ with labeled Hodge-Tate weights $\{-k_{i}\}_{\tau _{i}}.$
Moreover, $\mathbb{D}_{\vec{k},c}\simeq E^{\mid \tau \mid
}\bigotimes\limits_{\mathcal{O}_{E}^{\mid \tau \mid }}\left( \mathbb{N}_{%
\vec{k},c}/\pi \mathbb{N}_{\vec{k},c}\right) \ $as\ filtered$\ \varphi $%
-modules over $E^{\mid \tau \mid },$ where$\ \mathbb{D}_{\vec{k},c}=\left(
E^{\mid \tau \mid }\right) \eta $ is the filtered $\varphi $-module with
Frobenius endomorphism%
\begin{equation*}
\varphi (\eta )=(c\cdot p^{k_{1}},p^{k_{2}},...,p^{k_{f-1}},p^{k_{0}})\eta
\end{equation*}%
and filtration 
\begin{equation*}
\ \mathrm{Fil}^{\mathrm{j}}(\mathbb{D}_{\vec{k},c})=\left\{ 
\begin{array}{l}
E^{\mid \tau _{I_{0}}\mid }\eta \text{ \ \ \ }\mathrm{if}\text{\ ~}j\leq
w_{0}, \\ 
E^{\mid \tau _{I_{1}}\mid }\eta \ \ \text{ \ }\mathrm{if}\text{ \ }%
1+w_{0}\leq j\leq w_{1}, \\ 
\ \ \ \ \ \ \ \ \ \ \ \ \ \ \ \cdots \cdots \\ 
E^{\mid \tau _{I_{t-1}}\mid }\eta \text{ \ }\mathrm{if}\text{\ }%
1+w_{t-2}\leq j\leq w_{t-1}, \\ 
\ \ \ 0\ \ \ \text{\ \ \ \ \ \ }\mathrm{if}\text{\ }j\geq 1+w_{t-1}.%
\end{array}%
\right.
\end{equation*}
\end{proposition}

\begin{proof}
$(i)$ To prove that $\Gamma _{K}$ acts on $\mathbb{N}_{\vec{k},c},$ it
suffices to prove that $g_{i}^{\gamma _{1}\gamma _{2}}(\pi )=g_{i}^{\gamma
_{1}}\gamma _{1}(g_{i}^{\gamma _{2}})$ for all $\gamma _{1},\gamma _{2}\in
\Gamma _{K}$ and $i\in I_{0}.$ This follows immediately from the cocycle
relations 
\begin{equation*}
\frac{q}{\gamma _{1}\gamma _{2}(q)}=\frac{q}{\gamma _{1}(q)}\gamma
_{1}\left( \frac{q}{\gamma _{2}(q)}\right) \ \text{and }\lambda _{f,\gamma
_{1}\gamma _{2}}=\lambda _{f,\gamma _{1}}\gamma _{1}(\lambda _{f,\gamma
_{2}}),
\end{equation*}%
and the definition of the $g_{i}^{\gamma }(\pi ).$ Since $g_{i}^{\gamma
}(\pi )\equiv 1\func{mod}\pi $ for all $i\in I_{0},\ \Gamma _{K}$ acts
trivially on $\mathbb{N}_{\vec{k},c}/\pi \mathbb{N}_{\vec{k},c}.$ \noindent
\noindent $(ii)$ Let $k=\max \{k_{0},k_{1},...,k_{f-1}\}$ and $\varphi
^{\ast }(\mathbb{N}_{\vec{k},c})$ be the $\mathcal{O}_{E}[[\pi ]]^{\mid \tau
\mid }$-linear span of the set $\varphi (\mathbb{N}_{\vec{k},c}).$ Let $%
c_{1}=c^{-1}\ $and $c_{i}=1\ $if $i\neq 1.\ $Since $q^{k}\eta
=\sum\limits_{i=0}^{f-1}\left( q^{k-k_{i}}c_{i}e_{i}\right) \varphi (\eta
)\in $ $\varphi ^{\ast }(\mathbb{N}_{\vec{k},c}),\ $it follows that $q^{k}$
kills $\mathbb{N}_{\vec{k},c}/\varphi ^{\ast }(\mathbb{N}_{\vec{k},c}).$
\noindent \noindent $(iii)$ To compute the filtration of $\mathbb{N}_{\vec{k}%
,c},$ we use the fact that $q^{j}\mid $ $\varphi (x)$ if and only if $\pi
^{j}\mid x$ for any $x\in \mathcal{O}_{E}[[\pi ]].$ Let $%
x=(x_{0},x_{1},...,x_{f-1})\eta \in \mathbb{N}_{\vec{k},c}.$ By Theorem \ref%
{berger thm}, $x\in \mathrm{Fil}^{\mathrm{j}}\mathbb{N}_{\vec{k},c}$ if and
only if $\varphi (x)\in q^{j}\mathbb{N}_{\vec{k},c}$ or equivalently $%
q^{j}\mid $ $\varphi (x_{i})q^{k_{i}}$ for all $i\in I_{0}.$ If $j\leq k_{i}$
there are no restrictions on the $x_{i},$ whereas if $j>k_{i}$ this is
equivalent to $x_{i}\equiv 0\func{mod}\pi ^{j-k_{i}}.$ \noindent Therefore, 
\begin{equation*}
e_{i}\mathrm{Fil}^{\mathrm{j}}\mathbb{N}_{\vec{k},c}=\left\{ 
\begin{array}{l}
\ \ \ \ \ \ e_{i}\mathbb{N}_{\vec{k},c}\ \ \ \ \ \ \ \ \ \ \ \text{if }j\leq
k_{i}, \\ 
e_{i}\pi ^{j-k_{i}}\mathcal{O}_{E}[[\pi ]]\eta \ \ \ \ \text{if }j\geq
1+k_{i}.%
\end{array}%
\right.
\end{equation*}%
This implies that 
\begin{equation*}
\noindent E^{\mid \tau \mid }\bigotimes\limits_{\mathcal{O}_{E}^{\mid \tau
\mid }}e_{i}\mathrm{Fil}^{\mathrm{j}}\left( \mathbb{N}_{\vec{k},c}/\pi 
\mathbb{N}_{\vec{k},c}\right) =\left\{ 
\begin{array}{l}
e_{i}E^{\mid \tau \mid }\overline{\eta }\ \ \ \text{if }j\leq k_{i}, \\ 
\ \ \ \ \ 0\ \ \ \ \ \ \text{if }j\geq 1+k_{i}.%
\end{array}%
\right.
\end{equation*}%
For the filtration, we have 
\begin{equation*}
E^{\mid \tau \mid }\bigotimes\limits_{\mathcal{O}_{E}^{\mid \tau \mid }}%
\mathrm{Fil}^{\mathrm{j}}\left( \mathbb{N}_{\vec{k},c}/\pi \mathbb{N}_{\vec{k%
},c}\right) =\bigoplus\limits_{i=0}^{f-1}\left( E^{\mid \tau \mid
}\bigotimes\limits_{\mathcal{O}_{E}^{\mid \tau \mid }}e_{i}\mathrm{Fil}^{%
\mathrm{j}}(\mathbb{N}_{\vec{k},c}/\pi \mathbb{N}_{\vec{k},c})\right) .
\end{equation*}%
Recall (see Notation \ref{the notation}) that after ordering the weights $%
k_{i}\ $and omitting possibly repeated weights we get $%
w_{0}<w_{1}<...<w_{t-1}.$ By the formulas above,%
\begin{equation*}
\mathrm{Fil}^{\mathrm{j}}(\mathbb{D}_{\vec{k},c})=\left\{ 
\begin{array}{l}
E^{\mid \tau \mid }\left( \tsum\limits_{i\in I_{0}}e_{i}\right) \eta \text{
\ \ \ \ \ \ \ \ \ \ \ \ }\mathrm{if}\text{\ ~}j\leq w_{0}, \\ 
E^{\mid \tau \mid }\left( \tsum\limits_{\{i\in
I_{0}:k_{i}>w_{0}\}}e_{i}\right) \eta \ \ \text{ \ }\mathrm{if}\text{ \ }%
1+w_{0}\leq j\leq w_{1}, \\ 
E^{\mid \tau \mid }\left( \tsum\limits_{\{i\in
I_{0}:k_{i}>w_{1}\}}e_{i}\right) \eta \ \ \text{ \ }\mathrm{if}\text{ \ }%
1+w_{1}\leq j\leq w_{2} \\ 
\ \ \ \ \ \ \ \ \ \ \ \ \ \ \ \cdots \cdots \\ 
E^{\mid \tau \mid }\left( \tsum\limits_{\{i\in
I_{0}:k_{i}>w_{t-2}\}}e_{i}\right) \eta \ \ \text{ }\mathrm{if}\text{\ }%
1+w_{t-2}\leq j\leq w_{t-1}, \\ 
\ \ \ \ \ \ \ \ \ \ \ \ \ \ 0\ \ \ \ \ \ \ \ \ \ \ \ \text{\ \ \ \ \ \ \ }%
\mathrm{if}\text{\ }j\geq 1+w_{t-1}.%
\end{array}%
\right.
\end{equation*}%
The formula for the filtration follows immediately, recalling that $%
I_{j}=\{i\in I_{0}:k_{i}>w_{j-1}\}\ $for each $j=1,2,...,t-1,$\ and $E^{\mid
\tau _{I_{r}}\mid }:=E^{f}\left( \tsum\limits_{i\in I_{r}}e_{i}\right) $ for
each $r=0,1,...,t-1.$ The isomorphism of filtered $\varphi $-modules is
obvious.\ \ 
\end{proof}

\begin{proposition}
\label{rank 1 isom}Let $k_{0},k_{1},...,\noindent k_{f-1}$ be arbitrary
integers.

\begin{enumerate}
\item[(i)] The weakly admissible rank one filtered $\varphi $-modules over $%
E^{\mid \tau \mid }$ with labeled Hodge-Tate weights $\{-k_{i}\}_{\tau _{i}}$
are of the form $\mathbb{D}_{\vec{k},\vec{\alpha}}=\left( E^{\mid \tau \mid
}\right) \eta ,$ with $\varphi (e)=(\alpha _{0},\alpha _{1},...,\alpha
_{f-1})\eta $ for some $\vec{\alpha}=(\alpha _{0},\alpha _{1},...,\alpha
_{f-1})\in $ $(E^{\times })^{\mid \tau \mid }$ such that $\mathrm{v}_{%
\mathrm{p}}(\mathrm{Nm}_{\mathrm{\varphi }}(\vec{\alpha}))=\sum\limits_{i\in
I_{0}}k_{i}$ and 
\begin{equation*}
\mathrm{Fil}^{\mathrm{j}}(\mathbb{D}_{\vec{k},\vec{\alpha}})=\left\{ 
\begin{array}{l}
E^{\mid \tau _{I_{0}}\mid }\eta \text{ \ \ \ }\mathrm{if}\text{\ }j\leq
w_{0}, \\ 
E^{\mid \tau _{I_{1}}\mid }\eta \ \text{ \ \ }\mathrm{if}\text{ \ }%
1+w_{0}\leq j\leq w_{1}, \\ 
\ \ \ \ \ \ \ \ \ \ \ \ \ \cdots \cdots \\ 
E^{\mid \tau _{I_{t-1}}\mid }\eta \text{ \ }\mathrm{if}\text{\ }%
1+w_{t-2}\leq j\leq w_{t-1}, \\ 
\ \ \ 0\text{\ \ \ \ \ \ \ \ \ }\mathrm{if}\text{\ }j\geq 1+w_{t-1}.%
\end{array}%
\right.
\end{equation*}

\item[(ii)] The filtered $\varphi $-modules $\mathbb{D}_{\vec{k},\vec{\alpha}%
}$ and $\mathbb{D}_{\vec{v},\vec{\beta}}$ are isomorphic if and only if $%
\vec{k}=\vec{v}\ $and \noindent $\mathrm{Nm}_{\varphi }(\vec{\alpha})=%
\mathrm{Nm}_{\varphi }(\vec{\beta}).$
\end{enumerate}
\end{proposition}

\begin{proof}
Follows easily arguing as in\ \cite{DO10}, \S \S 4 and 6.
\end{proof}

\begin{corollary}
\label{cor3.6}All the effective crystalline $E$-characters of $G_{K}$ are
those constructed in Proposition $\ref{The crystalline characters}$%
.\noindent \noindent
\end{corollary}

\noindent Let $c\in \mathcal{O}_{E}^{\times },$ $\vec{k}%
=(-k_{1},-k_{2},...,-k_{f-1},-k_{0})\ $and $\chi _{c,\vec{k}}$ be the
crystalline character of $G_{K}$ corresponding to the Wach module $\mathbb{N}%
_{\vec{k},c}=\left( \mathcal{O}_{E}[[\pi ]]^{\mid \tau \mid }\right) \eta $
with $\varphi $ action defined by $\varphi (e)=(c\cdot
q^{k_{1}},q^{k_{2}},...,q^{k_{f-1}},q^{k_{0}})\eta $ and the unique
commuting with it $\Gamma _{K}$-action defined in Proposition \ref{The
crystalline characters}. When $c=1$ we simply write $\chi _{\vec{k}}.$ The
crystalline character $\chi _{e_{i}}$ has labeled Hodge-Tate weights $%
-e_{i+1}$ for all $i$ (see Proposition \ref{The crystalline characters}),
where $e_{i}=(0,...,1,...0),\ $with the first $i$ appearing in the $i$-th
place for all $i=0,1,...,f-1.$ By taking tensor products we see that $\chi
_{c,\vec{k}}=\chi _{c,\vec{0}}\cdot \chi _{0}^{k_{1}}\cdot \chi
_{1}^{k_{2}}\cdot \cdots \cdot \chi _{f-2}^{k_{f-1}}\cdot \chi
_{f-1}^{k_{0}}.$ Let $\mathrm{Frob}_{p}$ be the geometric Frobenius of $G_{%
\mathbb{Q}
_{p}}$ and $\mathrm{Frob}_{K}$ the geometric Frobenius of $G_{K}.$

\begin{lemma}

\begin{enumerate}
\item[(i)] \label{chic prop}Let $c\in \mathcal{O}_{E}^{\times }.$ The
unramified character of $G_{K_{f}}$ which maps $\mathrm{Frob}_{K_{f}}$ to $c$
equals $\chi _{c,\vec{0}};$

\item[(ii)] For any $i=0,1,...,f-1,$ $\left( \chi _{i}\right) _{\mid
G_{K_{2f}}}=\chi _{i}\cdot \chi _{i+f},$ where the character on the left
hand side is a character of $G_{K_{f}}$ and the characters on the right hand
side characters of $G_{K_{2f}};$

\item[(iii)] If $\chi $ is a crystalline character of $G_{K_{f}}$ with
labeled Hodge-Tate weights $\{-k_{i}\}_{\tau _{i}},$ $i=0,1,...,f-1,$ its
restriction to $G_{K_{2f}}$ has labeled weights $\{-k_{i}\}_{\tau _{i}},$ $%
i=0,1,...,2f-1,$ with $k_{i+f}=k_{i}$ for all $i=0,1,...,f-1;$

\item[(iv)] If $\chi $ and $\psi $ are crystalline characters of $G_{K_{f}},$
then $\chi _{\mid G_{K_{df}}}=\psi _{\mid G_{K_{df}}}$ if and only if $\chi
=\eta \cdot \psi ,$ where $\eta $ is an unramified character of $G_{K_{f}}$
which maps $\mathrm{Frob}_{K_{f}}$ to a $d$-th root of unity.
\end{enumerate}
\end{lemma}

\begin{proof}
(i) Let $\sqrt[f]{c}$ be any choice of an $f$-th root of $c$ in $E.$ The
filtered $\varphi $-module with trivial filtration and $\varphi (e)=\sqrt[f]{%
c}\cdot e$ corresponds to the unramified character $\eta $ of $G_{%
\mathbb{Q}
_{p}}$ which maps $\mathrm{Frob}_{p}$ to $\sqrt[f]{c}.$ Since the $\mathrm{%
Frob}_{K_{f}}=\left( \mathrm{Frob}_{p}\right) _{\mid K_{f}}^{f},\ $the
restriction of $\eta _{c}$ of $\eta $ to $K_{f}$ maps $\mathrm{Frob}_{K_{f}}$
to $c.$ By Proposition \ref{comment} the rank one filtered $\varphi $-module
corresponding to the unramified character $\eta _{c}$ has trivial filtration
and Frobenius $\varphi (e)=\left( \sqrt[f]{c},\sqrt[f]{c},...,\sqrt[f]{c}%
\right) e,$ and by Proposition \ref{rank 1 isom}(ii) the latter is
isomorphic to the rank one filtered $\varphi $-module with trivial
filtration and $\varphi (e)=\left( c,1,...1\right) e.$ Part (ii) follows
from the definition of the characters $\chi _{i}$ and Proposition \ref%
{comment}. Part (iii) follows immediately from part (ii). For part (iv) it
suffices to prove that any crystalline character $\eta $ of $G_{K_{f}}$ with
trivial restriction to $G_{K_{df}}$ is an unramified character of $G_{K_{f}}$
which maps $\mathrm{Frob}_{K_{f}}$ to a $d$-th root of unity. The
restriction of $\eta $ to $G_{K_{df}}$ has all its labeled Hodge-Tate
weights equal to zero, and by Corollary \ref{weights of restrictions} so
does $\eta .$ By part (i) $\eta $ is an unramified character of $G_{K_{f}}$
which maps $\mathrm{Frob}_{K_{f}}$ to some constant, say $c.$ The
restriction of $\eta $ to $G_{K_{df}}$ is trivial and maps $\mathrm{Frob}%
_{K_{df}}=\left( \mathrm{Frob}_{K_{f}}\right) _{\mid K_{df}}^{d}$ to $c^{d},$
therefore $c$ is a $d$-th root of unity and part (iv) follows.
\end{proof}

\noindent Let $\chi $ be any $E$-character of $G_{K},$ and let $h\in $ $G_{%
\mathbb{Q}
_{p}}.$ Since $K$ is unramified over $%
\mathbb{Q}
_{p},$ it is $h$-stable and the character $\chi ^{h}$ with $\chi ^{h}\left(
g\right) :=\chi \left( hgh^{-1}\right) $ is well defined. Let $h_{\mid
K}=:\sigma _{K}^{n\left( h\right) }$ for a unique integer $n\left( h\right)
~ $modulo $f.$ We denote by $\mathrm{T}\left( \chi \right) $ the rank one $%
\mathcal{O}_{E}$-representation of $G_{K}$ defined by $\gamma e=\chi \left(
\gamma \right) e$ for any basis element $e$ and any $\gamma \in G_{K}.$

\begin{lemma}
Let $\chi $ be the crystalline character corresponding to the Wach module
defined in Proposition $\ref{The crystalline characters}$, and let $h\in G_{%
\mathbb{Q}
_{p}}.$ Let $\eta _{1}=\left( \bar{h}_{\mid K}^{-1}\right) \cdot \eta .\ $%
The rank one module $\mathbb{N}^{h}:=\left( \mathcal{O}_{E}[[\pi ]]^{\mid
\tau \mid }\right) \eta _{1}$ endowed with semilinear Frobenius and $\Gamma
_{K}$-actions defined by%
\begin{eqnarray*}
\varphi \left( \eta _{1}\right) &=&\left( c\cdot q^{k_{f+1-n\left( h\right)
}},\ q^{k_{f+2-n\left( h\right) }},...,\ q^{k_{2f-n\left( h\right) }}\right)
\eta _{1}, \\
\gamma \left( \eta _{1}\right) &=&\left( g_{f+1+n\left( h^{-1}\right)
}^{h\gamma h^{-1}},\ g_{f+2-n\left( h\right) }^{h\gamma h^{-1}},...,\
g_{2f-1-n\left( h\right) }^{h\gamma h^{-1}},\ g_{2f-n\left( h\right)
}^{h\gamma h^{-1}}\right) \eta _{1},
\end{eqnarray*}%
where the indices are viewed modulo $f,\ $is a Wach module whose
corresponding crystalline character is $\chi ^{h}.$
\end{lemma}

\begin{proof}
It is trivial to check that $\mathbb{N}^{h}$ with the above defined actions
is a Wach module. By Theorems \ref{fontaine|s equivalence} and \ref{berger
thm}, $\mathrm{T}\left( \chi \right) \simeq \left( \mathbb{A}_{K,E}\otimes _{%
\mathbb{A}_{K,E}^{+}}\mathbb{N}\left( \mathrm{T}\left( \chi \right) \right)
\right) ^{\varphi =1},$ hence there exists some $\alpha \in \mathbb{A}_{K,E}$
such that $\varphi \left( \alpha \otimes \eta \right) =\alpha \otimes \eta $
and $\gamma \left( \alpha \otimes \eta \right) =\chi \left( \gamma \right)
\left( \alpha \otimes \eta \right) $ for all $\gamma \in G_{K}.$ This is
equivalent to%
\begin{align}
\varphi \left( \alpha \right) \cdot \xi ^{-1}\left( c\cdot
q^{k_{1}},q^{k_{2}},...,q^{k_{0}}\right) \otimes \eta & =\alpha \otimes \eta
\ \text{and}  \label{phi} \\
\gamma \left( \alpha \right) \cdot \xi ^{-1}\left( g_{1}^{\gamma
},g_{2}^{\gamma },...g_{f-1}^{\gamma },g_{0}^{\gamma }\right) \otimes \eta &
=\chi \left( \gamma \right) \alpha \otimes \eta  \label{gamma}
\end{align}%
for all $\gamma \in G_{K},$ where $\xi $ is the isomorphism defined in
formula \ref{ksi}. Let $\alpha _{1}:=h^{-1}\alpha \in \mathbb{A}_{K,E}.$ A
little computation shows that for any $\left( x_{0},x_{1},...,x_{f-1}\right)
\in \mathcal{O}_{E}[[\pi ]]^{\mid \tau \mid },$%
\begin{equation}
h^{-1}\left( \xi ^{-1}\left( x_{0},x_{1},...,x_{f-1}\right) \right) =\xi
^{-1}\left( x_{f-n\left( h\right) },x_{f+1-n\left( h\right)
},...,x_{2f-1-n\left( h\right) }\right) .  \label{n(sigma)}
\end{equation}%
We show that $\varphi \left( \alpha _{1}\otimes \eta _{1}\right) =\alpha
_{1}\otimes \eta _{1}$ and $\gamma \left( \alpha _{1}\otimes \eta
_{1}\right) =\chi ^{h}\left( \gamma \right) \left( \alpha _{1}\otimes \eta
_{1}\right) $ for all $\gamma \in G_{K}.$ Indeed, 
\begin{align*}
\varphi \left( \alpha _{1}\otimes \eta _{1}\right) & =\varphi \left(
h^{-1}\alpha \right) \otimes \varphi \left( \eta _{1}\right) \\
& =h^{-1}\varphi \left( \alpha \right) \cdot \xi ^{-1}\left( c\cdot
q^{k_{f+1-n\left( h\right) }},\ q^{k_{f+2-n\left( h\right) }}\ ,...,\
q^{k_{2f-n\left( h\right) }}\right) \otimes \eta _{1}\  \\
& \overset{\ref{n(sigma)}}{=}\noindent h^{-1}\varphi \left( \alpha \right)
\cdot h^{-1}\xi ^{-1}\left( c\cdot
q^{k_{1}},q^{k_{2}},...,q^{k_{f-1}},q^{k_{0}}\right) \otimes h^{-1}\eta \\
& \overset{\ref{phi}}{=}\ h^{-1}\left( \alpha \otimes \eta \right) =\alpha
_{1}\otimes \eta _{1}.\ 
\end{align*}%
Also,%
\begin{align*}
\gamma \left( \alpha _{1}\otimes \eta _{1}\right) & =\gamma \left(
h^{-1}\alpha \right) \cdot \xi ^{-1}\left( g_{f+1-n\left( h\right)
}^{h\gamma h^{-1}},\ g_{f+2-n\left( h\right) }^{h\gamma h^{-1}},...,\
g_{2f-1-n\left( h\right) }^{h\gamma h^{-1}},\ g_{2f-n\left( h\right)
}^{h\gamma h^{-1}}\right) \otimes \eta _{1} \\
& \overset{\ref{n(sigma)}}{=}\noindent h^{-1}\left( h\gamma h^{-1}\alpha
\cdot \xi ^{-1}\left( g_{1}^{h\gamma h^{-1}},\ g_{2}^{h\gamma h^{-1}},...,\
g_{f-1}^{h\gamma h^{-1}},\ g_{f}^{h\gamma h^{-1}}\right) \otimes \eta \right)
\\
& \overset{\ref{gamma}}{=}h^{-1}\left( \chi \left( h\gamma h^{-1}\right)
\alpha \otimes \eta \right) =\chi ^{h}\left( \gamma \right) \left( \alpha
_{1}\otimes \eta _{1}\right)
\end{align*}%
for all $\gamma \in G_{K}.$ By Theorems \ref{fontaine|s equivalence} and \ref%
{berger thm}, it follows that the crystalline character which corresponds to 
$\mathbb{N}^{h}$ is $\chi ^{h}.$
\end{proof}

\begin{corollary}
\label{perm of indices char}If $\chi $ is a crystalline $E^{\times }$-valued
characters of $G_{K}$ with labeled Hodge-Tate weights $\{-k_{i}\}_{\tau
_{i}},$ the character $\chi ^{h}$ is crystalline with labeled Hodge-Tate
weights $\{-\ell _{i}\}_{\tau _{i}},$ with $\ell _{i}=k_{f+i-n\left(
h\right) }$ for all $i,$ where the indices $f+i-n\left( h\right) $ are
viewed modulo $f.$
\end{corollary}

\begin{corollary}
\label{ti na to pw}The representation $V_{K_{f}}\simeq \mathrm{Ind}%
_{K_{2f}}^{K_{f}}\left( \chi _{0}^{k_{1}}\cdot \chi _{1}^{k_{2}}\cdot \cdots
\cdot \chi _{2f-2}^{k_{2f-1}}\cdot \chi _{2f-1}^{k_{0}}\right) $ is
crystalline. Moreover, $V_{K_{f}}$ is irreducible if and only if $k_{i}\neq
k_{i+f}$ for some $i\in \{0,1,...,f-1\}.$
\end{corollary}

\begin{proof}
Since $V_{K_{2f}}$ is crystalline, $V_{K_{f}}$ is crystalline. The corollary
follows from Mackey's irreducibility criterion and Corollary \ref{perm of
indices char}.
\end{proof}

\begin{proposition}
\label{induction in char 0}Let $V_{K}$ be an irreducible two-dimensional
crystalline $E$-representation of $G_{K_{f}}$ with labeled Hodge-Tate
weights $\{0,-k_{i}\}_{\tau _{i}},$ whose restriction to $G_{K_{2f}}$ is
reducible. There exist some unramified character $\eta $ of $G_{K_{f}}$ and
some nonnegative integers $\ell _{i},\ i=0,1,...,2f-1$ with $\{\ell
_{i},\ell _{i+f}\}=\{0,k_{i}\}$ for all $i=0,1,...,f-1$ and $\ell _{i}\neq
\ell _{i+f}$ for some $i\in \{0,1,...,f-1\},\ $such that 
\begin{equation*}
V_{K_{f}}\simeq \eta \otimes \mathrm{Ind}_{K_{2f}}^{K_{f}}\left( \chi
_{0}^{\ell _{1}}\cdot \chi _{1}^{\ell _{2}}\cdot \cdots \cdot \chi
_{2f-2}^{\ell _{2f-1}}\cdot \chi _{2f-1}^{\ell _{0}}\right) .
\end{equation*}
\end{proposition}

\begin{proof}
Let $\chi $ be a constituent of $V_{K_{2f}}.$ By Corollary \ref{cor3.6}, $%
\chi =\chi _{c}\cdot \chi _{0}^{\ell _{1}}\cdot \chi _{1}^{\ell _{2}}\cdot
\cdots \cdot \chi _{2f-2}^{\ell _{2f-1}}\cdot \chi _{2f-1}^{\ell _{0}}$ for
some $c\in \mathcal{O}_{E}^{\times }$ and some integers $\ell _{i}.$ Let $%
\eta $ be the unramified character of $G_{K_{f}}$ which maps $\mathrm{Frob}%
_{K_{f}}$ to $\sqrt[2]{c}.$ Arguing as in the proof of Lemma \ref{chic prop}%
(i) we see that the restriction of $\eta $ to $G_{K_{2f}}$ is $\chi _{c},$
hence $\chi _{0}^{\ell _{1}}\cdot \chi _{1}^{\ell _{2}}\cdot \cdots \cdot
\chi _{2f-2}^{\ell _{2f-1}}\cdot \chi _{2f-2}^{\ell _{0}}$ is a constituent
of $\left( \eta ^{-1}\otimes V_{K_{f}}\right) _{\mid K_{2f}}.$ Since $\eta
^{-1}\otimes V_{K_{f}}$ is irreducible, 
\begin{equation*}
\eta ^{-1}\otimes V_{K_{f}}\simeq \mathrm{Ind}_{K_{2f}}^{K_{f}}\left( \chi
_{0}^{\ell _{1}}\cdot \chi _{1}^{\ell _{2}}\cdot \cdots \cdot \chi
_{2f-2}^{\ell _{2f-1}}\cdot \chi _{2f-1}^{\ell _{0}}\right)
\end{equation*}%
by Frobenius reciprocity. By Mackey's formula and Corollary \ref{perm of
indices char}, 
\begin{equation}
V_{K_{2f}}\simeq \left( \chi _{c}\cdot \chi _{0}^{\ell _{1}}\cdot \chi
_{1}^{\ell _{2}}\cdot \cdots \cdot \chi _{2f-2}^{\ell _{2f-1}}\cdot \chi
_{2f-1}^{\ell _{0}}\right) \tbigoplus \left( \chi _{c}\cdot \chi _{0}^{\ell
_{1+f}}\cdot \chi _{1}^{\ell _{2+f}}\cdot \cdots \cdot \chi _{2f-2}^{\ell
_{3f-1}}\cdot \chi _{2f-1}^{\ell _{3f}}\right) ,  \label{combinatorial}
\end{equation}%
where the indices of the exponents of the second summand are viewed modulo $%
2f.$ By Corollary \ref{weights of restrictions}, the restricted
representation $V_{K_{2f}}$ has labeled Hodge-Tate weights $%
\{0,-k_{i}\}_{\tau _{i}},$ $i=0,1,2,...,2f-1,$ where $k_{i+f}=k_{i}$ for all 
$i=0,1,...,f-1.$ The labeled Hodge-Tate weights of the direct sum of
characters in formula \ref{combinatorial} with respect to the embedding $%
\tau _{i}$ of $K_{2f}$ to $E$ are $\{-\ell _{i},-\ell _{i+f}\}$ for all $%
i=0,1,2,...,2f-1,$ with the indices $i+f$ viewed modulo $2f.$ Therefore $%
\{\ell _{i},\ell _{i+f}\}=\{0,k_{i}\}$ for all $i=0,1,...,f-1.$ The rest of
the proposition follows from Corollary \ref{ti na to pw}.
\end{proof}

\begin{proposition}
\label{2to the f}Up to twist by some unramified character, there exist
precisely $2^{f^{^{+}}-1}$ distinct isomorphism classes of irreducible
crystalline two-dimensional $E$-representations of $G_{K_{f}}\ $with labeled
Hodge-Tate weights $\{0,-k_{i}\}_{\tau _{i}},$ whose restriction to $%
G_{K_{2f}}$ is reducible.
\end{proposition}

\begin{proof}
In Proposition \ref{induction in char 0}, notice that $\ell
_{i+f}=k_{i}-\ell _{l}$ for all $i=0,1,...,f-1.$ The corollary follows since 
$\mathrm{Ind}_{K_{2f}}^{K_{f}}\left( \chi \right) \simeq \mathrm{Ind}%
_{K_{2f}}^{K_{f}}\left( \psi \right) $ if and only if $\{\chi ,\chi
^{h}\}=\{\psi ,\psi ^{h}\},$ where $h$ is any element in $G_{%
\mathbb{Q}
_{p}}$ lifting a generator of $\mathrm{Gal}\left( K_{2f}/K_{f}\right) .$
\end{proof}

\section{Families of effective Wach modules of arbitrary weight and rank 
\label{general construction of Wach}}

\noindent We extend the method used by Berger-Li-Zhu in \cite{BLZ04} for
two-dimensional crystalline representations of $G_{%
\mathbb{Q}
_{p}},$ to construct families of Wach modules of effective crystalline
representations of $G_{K}$ of arbitrary rank. In order to construct the Wach
module of an effective crystalline representation, fixing a basis, we need
to exhibit matrices $\Pi $ and $G_{\gamma }$ such that $\Pi \varphi
(G_{\gamma })=G_{\gamma }\gamma (\Pi )$ for all $\gamma \in \Gamma _{K},$
with some additional properties imposed by Theorem \ref{berger thm}. In the
two-dimensional case for $G_{%
\mathbb{Q}
_{p}},$ and for a suitable basis, it is trivial to write down such a matrix $%
\Pi $ assuming that the valuation of the trace of Frobenius of the
corresponding filtered $\varphi $-module is suitably large and the main
difficulty is in constructing a $\Gamma _{K}$-action which commutes with $%
\Pi .$ When $K\neq 
\mathbb{Q}
_{p},$ finding a matrix $\Pi $ which gives rise to a prescribed weakly
admissible filtration seems to be already hard, even in the two-dimensional
case. Assuming that such a matrix $\Pi $ is available it is usually very
hard to explicitly write down the matrices $G_{\gamma }.$ There are
exceptions to this, for example some split-reducible two-dimensional
crystalline representations. In the general case, instead of explicitly
writing down the matrices $G_{\gamma }$ we prove that such matrices exist
using a successive approximation argument.

Let $\mathcal{S}=\{X_{i}\ ;\ i=0,1,...,m-1\}$ be a set of indeterminates,
were $m\geq 1$ is any integer. We extend the actions of $\varphi $ and $%
\Gamma _{K}$ defined in equations \ref{actions1} and \ref{actions2} on the
ring $\mathcal{O}_{E}[[\pi ]]^{\mid \tau \mid }:=\prod\limits_{\tau
:K\hookrightarrow E}$ $\mathcal{O}_{E}[[\pi ]]\ $to an action on $\mathcal{O}%
_{E}[[\pi ,\mathcal{S}]]^{\mid \tau \mid }:=\prod\limits_{\tau
:K\hookrightarrow E}$ $\mathcal{O}_{E}[[\pi ,\mathcal{S}]],$ by letting $%
\varphi $ and $\Gamma _{K}$ act trivially on each indeterminate $X_{i}.$ We
let $\varphi $ and $\Gamma _{K}$ act on the matrices of $M_{n}^{\mathcal{S}%
}:=M_{n}(\mathcal{O}_{E}[[\pi ,\mathcal{S}]]^{\mid \tau \mid })$ entry-wise
for any integer $n\geq 2.$ For any integer $s\geq 0,$ we\textit{\ }write $%
\vec{\pi}^{s}=\left( \pi ^{s},\pi ^{s},...,\pi ^{s}\right) ,$ and for any $%
\alpha \in \mathcal{O}_{E}\left[ \left[ \pi ,\mathcal{S}\right] \right] $
and any vector $\vec{r}=\left( r_{0},r_{1},...,r_{f-1}\right) $ with
nonnegative integer coordinates, we write $\alpha ^{\vec{r}}=\left( \alpha
^{r_{0}},\alpha ^{r_{1}},...,\alpha ^{r_{f-1}}\right) .$ As usual, we assume
that $k_{i}$ are nonnegative integers, and we write $k:=w_{t-1}=\max
\{k_{0},k_{1},$\noindent $...,k_{f-1}\}.$ Let $\ell \geq k\ $be any fixed
integer. We start with the following lemma.

\begin{lemma}
\label{main technical lemma}Let $\Pi _{i}=\Pi _{i}(\mathcal{S}),\
i=0,1,...,f-1$ be matrices in $M_{n}(\mathcal{O}_{E}[[\pi ,\mathcal{S}]])$
such that $\det (\Pi _{i})=c_{i}q^{k_{i}},$ with $c_{i}\in \mathcal{O}%
_{E}[[\pi ]]^{\times }.$ We denote by $\Pi (\mathcal{S})\ $the matrix $%
\left( \Pi _{1},\Pi _{2},...,\Pi _{f-1},\Pi _{0}\right) \ $and view it as an
element of $M_{n}^{\mathcal{S}}$ via the natural isomorphism $M_{n}^{%
\mathcal{S}}\simeq M_{n}(\mathcal{O}_{E}[[\pi ,\mathcal{S}]])^{\mid \tau
\mid }.$ We denote by $P_{i}=P_{i}(\mathcal{S})\ $the reduction of $\Pi _{i}%
\func{mod}\pi $ for all $i.$ Assume that for each $\gamma \in \Gamma _{K}$
there exists a matrix $G_{\gamma }^{(\ell )}=G_{\gamma }^{(\ell )}(\mathcal{S%
})\in M_{n}^{\mathcal{S}}$ such that \noindent

\begin{enumerate}
\item $G_{\gamma }^{(\ell )}(\mathcal{S})\equiv \overrightarrow{Id}\func{mod}%
\vec{\pi}^{\ell };$ \noindent

\item $G_{\gamma }^{(\ell )}(\mathcal{S})-\Pi (\mathcal{S})\mathcal{\varphi }%
(G_{\gamma }^{(\ell )}(\mathcal{S}))\gamma (\Pi (\mathcal{S})^{-1})\in \vec{%
\pi}^{\ell }M_{n}^{\mathcal{S}}.$ We further assume that \noindent \noindent

\item There is no nonzero matrix $H\in M_{n}(\mathcal{O}_{E}[[\mathcal{S}%
]]^{\mid \tau \mid })$ such that $HU=p^{ft}UH$ \noindent for some $t>0,$
where $U=\mathrm{Nm}_{\varphi }(P)\ $and $P=P\left( \mathcal{S}\right)
=\left( P_{1},P_{2},...,P_{f-1},P_{0}\right) ;$ \noindent \noindent

\item For each $s\geq \ell +1$ the operator 
\begin{equation}
H\longmapsto H-Q_{f}H(p^{f(s-1)}Q_{f}^{-1}):M_{n}\left( \mathcal{O}_{E}\left[
\left[ \mathcal{S}\right] \right] \right) \longrightarrow M_{n}\left( 
\mathcal{O}_{E}\left[ \left[ \mathcal{S}\right] \right] \right) ,
\label{operator equation}
\end{equation}%
where $Q_{f}=P_{1}P_{2}\cdots P_{f-1}P_{0}$\noindent\ is surjective.
\noindent Then for each $\gamma \in \Gamma _{K}$ there exists a unique
matrix \noindent $G_{\gamma }(\mathcal{S})\in M_{n}^{\mathcal{S}}\ $such
that $\noindent $

\item[(i)] $G_{\gamma }(\mathcal{S})\equiv \overrightarrow{Id}\func{mod}\vec{%
\pi}$ and

\item[(ii)] $\Pi (\mathcal{S})\mathcal{\varphi }(G_{\gamma }(\mathcal{S}%
))=G_{\gamma }(\mathcal{S})\gamma (\Pi (\mathcal{S})).$
\end{enumerate}
\end{lemma}

\begin{proof}
Uniqueness\textit{:} Suppose that the matrices $G_{\gamma }(\mathcal{S})$
and $G_{\gamma }^{\prime }(\mathcal{S})$ both satisfy the conclusions of the
lemma, and let $H=G_{\gamma }^{\prime }(\mathcal{S})G_{\gamma }(\mathcal{S}%
)^{-1}.$ We easily see that $H\in \vec{Id}+\vec{\pi}M_{n}^{\mathcal{S}}$ and 
$H\Pi (\mathcal{S})=\Pi (\mathcal{S})\varphi \left( H\right) .$ We'll show
that $H=\vec{Id}.$ We write $H=\vec{Id}+\pi ^{t}H_{t}+\cdots ,$ where $%
H_{t}\in M_{n}(\mathcal{O}_{E}[[\mathcal{S}]]^{\mid \tau \mid }),$ $t\geq 1$
and $\Pi (\mathcal{S})=P+\pi P^{(1)}+\pi ^{2}P^{(2)}+\cdots ,$ and we will
show that $H_{t}=0.$ Since $H\Pi (\mathcal{S})=\Pi (\mathcal{S})\varphi
\left( H\right) ,$ we have $(H-\vec{Id})\Pi (\mathcal{S})=\Pi (\mathcal{S}%
)\varphi (H-\vec{Id}).$ We divide both sides of this equation by $\pi ^{t}\
( $using that\ $\varphi (\pi )=q\pi )$ and reduce$\ \func{mod}$\ $\pi .$
This gives $H_{t}P=p^{t}P\varphi \left( H_{t}\right) $ (since $q\equiv p%
\func{mod}\pi $), which implies that $H_{t}U=p^{ft}U\varphi ^{f}\left(
H_{t}\right) ,\ $with $U=\mathrm{Nm}_{\varphi }(P).$ Since $\varphi $ acts
trivially on $X_{i}\ $and $\mathcal{O}_{E},$ $\varphi ^{f}$ acts trivially
on $M_{n}(\mathcal{O}_{E}[[\mathcal{S}]]^{\mid \tau \mid }),$ therefore $%
H_{t}U=p^{ft}UH_{t}$ and $H_{t}=0$ by assumption (iii) of the lemma.

\noindent \noindent \noindent Existence:\textit{\ }Fix a $\gamma \in \Gamma
_{K}.$\textbf{\ }By assumptions (i) and (ii) of the lemma, there exists a
matrix $G_{\gamma }^{(\ell )}\in \vec{Id}+\vec{\pi}^{\ell }$ $M_{n}^{%
\mathcal{S}}$ such that $G_{\gamma }^{(\ell )}-\Pi (\mathcal{S})\varphi
(G_{\gamma }^{(\ell )})\gamma (\Pi (\mathcal{S})^{-1})=\vec{\pi}^{\ell
}R^{\left( \ell \right) }$ for some matrix $R^{\left( \ell \right)
}=R^{\left( \ell \right) }(\gamma )\in M_{n}^{\mathcal{S}}.$ We shall prove
that for each $s\geq \ell +1,$ there exist matrices $R^{\left( s\right)
}=R^{\left( s\right) }(\gamma )\in M_{n}^{\mathcal{S}}\ $and$\ G_{\gamma
}^{(s)}\in M_{n}^{\mathcal{S}}$ such that $G_{\gamma }^{(s)}\equiv G_{\gamma
}^{(s-1)}\func{mod}\vec{\pi}^{s-1}M_{n}^{\mathcal{S}}\ $and$\ G_{\gamma
}^{(s)}-\Pi (\mathcal{S})\varphi (G_{\gamma }^{(s)})\gamma (\Pi (\mathcal{S}%
)^{-1})=\vec{\pi}^{s}R^{\left( s\right) }.$ \smallskip \noindent Let $%
G_{\gamma }^{(s)}=G_{\gamma }^{(s-1)}+\vec{\pi}^{s-1}H^{(s)},$ where $%
H^{\left( s\right) }\in M_{n}(\mathcal{O}_{E}[[\mathcal{S}]]^{\mid \tau \mid
})$ and write$\ R^{(s)}=\bar{R}^{(s)}+\vec{\pi}\cdot C\ $with $C\in M_{n}^{%
\mathcal{S}}.$ We need \noindent 
\begin{equation*}
\left( G_{\gamma }^{(s-1)}+\vec{\pi}^{\left( s-1\right) }H^{\left( s\right)
}\right) -\Pi (\mathcal{S})\left( \varphi (G_{\gamma }^{(s-1)})+\vec{\varphi
\left( \pi \right) }^{\left( s-1\right) }\varphi \left( H^{\left( s\right)
}\right) \right) \gamma (\Pi (\mathcal{S})^{-1})\noindent \in \vec{\pi}%
^{s}M_{n}^{\mathcal{S}},
\end{equation*}%
$\ $or equivalently$\ \ \ $%
\begin{equation*}
G_{\gamma }^{(s-1)}-\Pi (\mathcal{S})\varphi (G_{\gamma }^{(s-1)})\gamma
(\Pi (\mathcal{S})^{-1})+\vec{\pi}^{\left( s-1\right) }H^{\left( s\right) }-%
\vec{\left( q\pi \right) }^{\left( s-1\right) }\Pi (\mathcal{S})\varphi
\left( H^{\left( s\right) }\right) \gamma (\Pi (\mathcal{S})^{-1})\noindent
\in \vec{\pi}^{s}M_{n}^{\mathcal{S}}.
\end{equation*}%
The latter is equivalent to 
\begin{equation*}
\vec{\pi}^{\left( s-1\right) }R^{\left( s-1\right) }+\vec{\pi}^{\left(
s-1\right) }H^{\left( s\right) }-\vec{\left( q\pi \right) }^{\left(
s-1\right) }\Pi (\mathcal{S})\varphi \left( H^{\left( s\right) }\right)
\noindent \gamma (\Pi (\mathcal{S})^{-1})\in \vec{\pi}^{s}M_{n}^{\mathcal{S}%
}\noindent ,\ 
\end{equation*}%
\noindent which is in turn equivalent to $\noindent H^{\left( s\right) }-%
\vec{q}^{\left( s-1\right) }\Pi (\mathcal{S})\varphi \left( H^{\left(
s\right) }\right) \gamma (\Pi (\mathcal{S})^{-1})\equiv -R^{\left(
s-1\right) }\func{mod}\vec{\pi}M_{n}^{\mathcal{S}}.$ This holds if and only
if%
\begin{equation}
H^{\left( s\right) }-\vec{p}^{(s-1)}P\left( \mathcal{S}\right) \varphi
\left( H^{\left( s\right) }\right) P\left( \mathcal{S}\right) ^{-1}=-\bar{R}%
^{\left( s-1\right) }.  \label{equation 3}
\end{equation}%
Notice that \noindent $\vec{p}^{\left( s-1\right) }$\noindent $P\left( 
\mathcal{S}\right) ^{-1}\in M_{n}(\mathcal{O}_{E}[[\mathcal{S}]]^{\mid \tau
\mid })$ since $s-1\geq \ell \geq k=\max \{k_{0},k_{1},...,k_{f-1}\}.$ We
write%
\begin{equation*}
H^{(s)}=\left( H_{1}^{(s)},H_{2}^{(s)},...,H_{f-1}^{(s)},H_{0}^{(s)}\right)
\end{equation*}%
and%
\begin{equation*}
-\bar{R}^{\left( s-1\right) }=\left( \bar{R}_{1}^{(s-1)},\bar{R}%
_{2}^{(s-1)},...,\bar{R}_{f-1}^{(s-1)},\bar{R}_{0}^{(s-1)}\right) .
\end{equation*}%
Equation \ref{equation 3} is equivalent to the system of equations in $%
M_{n}\left( \mathcal{O}_{E}\left[ \left[ \mathcal{S}\right] \right] \right) $
\begin{equation}
H_{i}^{\left( s\right) }-P_{i}\cdot H_{i+1}^{\left( s\right) }\cdot \left(
p^{s-1}P_{i}^{-1}\right) =\bar{R}_{i}^{\left( s-1\right) },\ 
\label{equations5}
\end{equation}%
where $i=1,2,...,f,$ \noindent with indices viewed $\func{mod}$ $f.$
\noindent \noindent These imply that%
\begin{gather*}
H_{1}^{(s)}-Q_{f}H_{1}^{(s)}(p^{f(s-1)}Q_{f}^{-1})=\bar{R}_{1}^{(s-1)}+Q_{1}%
\bar{R}_{2}^{(s-1)}(p^{(s-1)}Q_{1}^{-1})+Q_{2}\bar{R}%
_{3}^{(s-1)}(p^{2(s-1)}Q_{2}^{-1}) \\
+\cdots +Q_{f-1}\bar{R}_{0}^{(s-1)}(p^{(s-1)(f-1)}Q_{f-1}^{-1}),
\end{gather*}%
where$\ Q_{i}=P_{1}\cdots P_{i}\ \ $for$\mathrm{\ }$all$\ i=1,2,...,f.$ From
equations \ref{equations5} we see that the matrices $H_{i}^{(s)},$ $%
i=2,3,...,f,~$are uniquely determined from the matrix $H_{1}^{(s)},$ so it
suffices to prove that the operator defined in formula \ref{operator
equation} \noindent contains\noindent\ 
\begin{equation*}
A=\bar{R}_{1}^{(s-1)}+Q_{1}\bar{R}_{2}^{(s-1)}(p^{(s-1)}Q_{1}^{-1})+Q_{2}%
\bar{R}_{3}^{(s-1)}(p^{2(s-1)}Q_{2}^{-1})+\cdots +Q_{f-1}\bar{R}%
_{0}^{(s-1)}(p^{(s-1)(f-1)}Q_{f-1}^{-1})
\end{equation*}%
in its image. Since $p^{i\left( s-1\right) }Q_{i}^{-1}\in M_{n}\left( 
\mathcal{O}_{E}\left[ \left[ \mathcal{S}\right] \right] \right) $ for all $%
i, $ this is true by assumption (iv) of the lemma. We define $G_{\gamma }(%
\mathcal{S})=\underset{s\rightarrow \infty }{\lim }G_{\gamma }^{(s)}(%
\mathcal{S})$ and the proof is complete.
\end{proof}

\noindent Let $\widetilde{M_{n}}$ be the ring $M_{n}\left( \mathcal{O}_{E}%
\left[ \left[ \mathcal{S}\right] \right] \right) /I$ where $I$ is the ideal
of \noindent $M_{n}\left( \mathcal{O}_{E}\left[ \left[ \mathcal{S}\right] %
\right] \right) $ generated by the set $\left\{ p\cdot Id,\ X_{i}\cdot Id:\
X_{i}\in \mathcal{S}\right\} .$ We use the notation of Lemma \ref{main
technical lemma} and its proof, and we are interested in the image of the
operator $\overline{H}\mapsto \overline{H-Q_{f}H(p^{f\ell }Q_{f}^{-1})}:%
\widetilde{M_{n}}\rightarrow \widetilde{M_{n}}$ where bar denotes reduction
modulo$\ I.$

\begin{proposition}
\label{prop operator copy(1)}If the operator 
\begin{equation}
\overline{H}\mapsto \overline{H-Q_{f}H(p^{f\ell }Q_{f}^{-1})}:\widetilde{%
M_{n}}\rightarrow \widetilde{M_{n}}  \label{operator modp}
\end{equation}%
is surjective, then for each $s\geq \ell +1$ the operator defined in formula 
$\ref{operator equation}$ \noindent is surjective.
\end{proposition}

\begin{proof}
Case (i). $s\geq k+2.$ In this case $f(s-1)-\sum\limits_{i=0}^{f-1}k_{i}\geq
f\left( s-1-k\right) \geq f\geq 1.$ Since $%
Q_{f}^{-1}=P_{0}^{-1}P_{f-1}^{-1}P_{f-2}^{-1}...P_{1}^{-1}$ and $\det
(P_{i})=\bar{c}_{i}p^{k_{i}},$ it follows that $p^{f(s-1)}Q_{f}^{-1}\in
pM_{n}(\mathcal{O}_{E}[[\mathcal{S}]]).$ \noindent Let $B$ be any matrix in $%
M_{n}\left( \mathcal{O}_{E}\left[ \left[ \mathcal{S}\right] \right] \right)
. $ We write $B=B-Q_{f}B\left( p^{f(s-1)}Q_{f}^{-1}\right) +pB_{1}$ for some
matrix

\noindent $B_{1}\in M_{n}\left( \mathcal{O}_{E}\left[ \left[ \mathcal{S}%
\right] \right] \right) .$ \noindent Similarly, \noindent \noindent $%
B_{1}=B_{1}-Q_{f}B_{1}\left( p^{f(s-1)}Q_{f}^{-1}\right) +pB_{2}$ for some
matrix $B_{2}\in M_{n}\left( \mathcal{O}_{E}\left[ \left[ \mathcal{S}\right] %
\right] \right) $ and $B=\left( B+pB_{1}\right) -\noindent Q_{f}\left(
B+pB_{1}\right) \left( p^{f(s-1)}Q_{f}^{-1}\right) +p^{2}B_{2}.$ Continuing
in the same fashion we get 
\begin{equation*}
B=\left( \sum\limits_{i=0}^{N}p^{i}B_{i}\right) -Q_{f}\left(
\sum\limits_{i=0}^{N}p^{i}B_{i}\right) \left( p^{f(s-1)}Q_{f}^{-1}\right)
+p^{N+1}B_{N+1}
\end{equation*}%
for some matrix $B_{N+1}$\noindent $\in M_{n}\left( \mathcal{O}_{E}\left[ %
\left[ \mathcal{S}\right] \right] \right) $ with $B_{0}=B.$ Let $%
H=\sum\limits_{i=0}^{\infty }$ $p^{i}B_{i},$ then \noindent $H\in
M_{n}\left( \mathcal{O}_{E}\left[ \left[ \mathcal{S}\right] \right] \right) $
and $B=H-Q_{f}H\left( p^{f(s-1)}Q_{f}^{-1}\right) .$

\noindent \noindent Case (ii). $\ell =k\ $and $s=k+1.$ We reduce modulo the
ideal $I\ $defined before Proposition \ref{prop operator copy(1)}. Let $A$
be any element of $M_{n}\left( \mathcal{O}_{E}\left[ \left[ \mathcal{S}%
\right] \right] \right) .$ The operator $\overline{H}\longmapsto \overline{%
H-Q_{f}H\left( p^{f\ell }Q_{f}^{-1}\right) }:\widetilde{M_{n}}\rightarrow $ $%
\widetilde{M_{n}}$ contains $\bar{A}=A\func{mod}I$ in its image by the
assumption of the lemma. Let $A=A_{0}-Q_{f}A_{0}\left( p^{f\ell
}Q_{f}^{-1}\right) \func{mod}I$ for some matrix $A_{0}\in M_{n}\left( 
\mathcal{O}_{E}\left[ \left[ \mathcal{S}\right] \right] \right) .$ We
write\noindent 
\begin{equation*}
A=A_{0}-Q_{f}A_{0}\left( p^{f\ell }Q_{f}^{-1}\right)
+pB_{m}+X_{0}B_{0}+\cdots +X_{m-1}B_{m-1}
\end{equation*}%
for matrices $B_{i}\in M_{n}\left( \mathcal{O}_{E}\left[ \left[ \mathcal{S}%
\right] \right] \right) .$ Similarly \noindent $%
B_{i}=B_{i}^{0}-Q_{f}B_{i}^{0}\left( p^{f\ell }Q_{f}^{-1}\right) \func{mod}I$
for matrices $B_{i}^{0}\in M_{n}\left( \mathcal{O}_{E}\left[ \left[ \mathcal{%
S}\right] \right] \right) $ and for all $i.$ Then%
\begin{gather*}
A=A_{0}-Q_{f}A_{0}\left( p^{f\ell }Q_{f}^{-1}\right) +pB_{m}^{0}-Q_{f}\left(
pB_{m}^{0}\right) \left( p^{f\ell }Q_{f}^{-1}\right)
+X_{0}B_{1}^{0}-Q_{f}\left( X_{0}B_{1}^{0}\right) \left( p^{f\ell
}Q_{f}^{-1}\right) \\
+\cdots +X_{m-1}B_{m-1}^{0}-Q_{f}\left( X_{m-1}B_{f-1}^{0}\right) \left(
p^{f\ell }Q_{f}^{-1}\right) \func{mod}I^{2},\ 
\end{gather*}%
therefore%
\begin{equation*}
A=(A_{0}+pB_{m}^{0}+X_{0}B_{1}^{0}+\cdots
+X_{m-1}B_{m-1}^{0})-Q_{f}(A_{0}+pB_{m}^{0}+X_{0}B_{1}^{0}+\cdots
+X_{f-1}B_{m-1}^{0})\left( p^{f\ell }Q_{f}^{-1}\right) \func{mod}I^{2}.
\end{equation*}%
By induction, \noindent $A=H-Q_{f}H\left( p^{f\ell }Q_{f}^{-1}\right) $ for
some \noindent $H\in M_{n}\left( \mathcal{O}_{E}\left[ \left[ \mathcal{S}%
\right] \right] \right) .$
\end{proof}

\noindent The surjectivity assumption of Proposition \ref{prop operator
copy(1)} is usually trivial to check thanks to the following proposition.

\begin{proposition}
\label{prop operator}Assume that $\ell >k$ or $\ell =k$ and the weights $%
k_{i}$ are not all equal. Then the operator defined in formula $\ref%
{operator modp}$ is surjective.
\end{proposition}

\begin{proof}
The proposition follows immediately because $\det Q_{f}=\bar{c}%
p^{k_{1}+k_{2}+\cdots +k_{f}},$ where $\bar{c}=\bar{c}_{1}\bar{c}_{2}\cdots 
\bar{c}_{f},$ since $f\ell >k_{1}+\cdots +k_{f}$ and $p\in I.$
\end{proof}

\noindent The following lemma summarizes the results of this section. We use
the notation of Lemma \ref{main technical lemma}.

\begin{lemma}
\label{outline section theorem}Let $\ell \geq k$ be a fixed integer. We
assume that for each $\gamma \in \Gamma _{K}$ there exists a matrix $%
G_{\gamma }^{(\ell )}=G_{\gamma }^{(\ell )}(\mathcal{S})\in M_{n}^{\mathcal{S%
}}$ such that

\begin{enumerate}
\item \noindent \noindent $G_{\gamma }^{(\ell )}(\mathcal{S})\equiv 
\overrightarrow{Id}\func{mod}\vec{\pi}^{\ell };$

\item $G_{\gamma }^{(\ell )}(\mathcal{S})-\Pi (\mathcal{S})\mathcal{\varphi }%
(G_{\gamma }^{(\ell )}(\mathcal{S}))\gamma (\Pi (\mathcal{S})^{-1})\in \vec{%
\pi}^{\ell }M_{n}^{\mathcal{S}};$

\item There is no nonzero matrix $H\in M_{n}(\mathcal{O}_{E}[[\mathcal{S}%
]]^{\mid \tau \mid })$ such that $HU=p^{ft}UH\ $for some $t>0,$\smallskip
\noindent

\item If $\ell =k$ and $k=k_{i}$ for all $i,$ we additionally assume that
the operator%
\begin{equation*}
\overline{H}\mapsto \overline{H-Q_{f}H(p^{f\ell }Q_{f}^{-1})}:\widetilde{%
M_{n}}\rightarrow \widetilde{M_{n}}
\end{equation*}%
is surjective.
\end{enumerate}

T\noindent hen for each $\gamma \in \Gamma _{K}$ there exists a unique
matrix \noindent $G_{\gamma }(\mathcal{S})\in M_{n}^{\mathcal{S}}\ $such
that \noindent

\begin{enumerate}
\item[(i)] $G_{\gamma }(\mathcal{S})\equiv \overrightarrow{Id}\func{mod}\vec{%
\pi},$ and

\item[(ii)] $\Pi (\mathcal{S})\mathcal{\varphi }(G_{\gamma }(\mathcal{S}%
))=G_{\gamma }(\mathcal{S})\gamma (\Pi (\mathcal{S})).$
\end{enumerate}
\end{lemma}

\noindent For any $\vec{a}=(a_{0},a_{1},...,a_{f-1})\in \mathfrak{m}%
_{E}^{\left\vert \mathcal{S}\right\vert }$ we denote by $\Pi (\vec{a}%
)=\left( \Pi _{1}(a_{1}),\Pi _{2}(a_{2}),...,\Pi _{f-1}(a_{f-1}),\Pi
_{0}(a_{0})\right) $ the matrix obtained from $\Pi (\mathcal{S})=\left( \Pi
_{1}(X_{1}),\Pi _{2}(X_{2}),...,\Pi _{f-1}(X_{f-1}),\Pi _{0}(X_{0})\right) $
by substituting $a_{i}\in \mathfrak{m}_{E}$ in each indeterminate $X_{i}$ of 
$\Pi _{i}(X_{i}).$\noindent\ 

\begin{proposition}
\label{gamma acts}For any $\vec{a}=(a_{0},a_{1},...,a_{f-1})\in \mathfrak{m}%
_{E}^{\left\vert \mathcal{S}\right\vert }$ and any $\gamma _{1},\gamma
_{2},\gamma \in \Gamma _{K},$ the following equations hold:

\begin{enumerate}
\item[(i)\noindent ] $G_{\gamma _{1}\gamma _{2}}(\vec{a})=G_{\gamma _{1}}(%
\vec{a})\gamma _{1}(G_{\gamma _{2}}(\vec{a}))$ and

\item[(ii)\noindent ] $\Pi (\vec{a})\mathcal{\varphi }(G_{\gamma }(\vec{a}%
))=G_{\gamma }(\vec{a})\gamma (\Pi (\vec{a})).$
\end{enumerate}
\end{proposition}

\begin{proof}
Both matrices $G_{\gamma _{1}\gamma _{2}}(\mathcal{S})$ and $G_{\gamma _{1}}(%
\mathcal{S})\gamma _{1}(G_{\gamma _{2}}(\mathcal{S}))$ are $\equiv 
\overrightarrow{Id}\func{mod}\vec{\pi}$ and are solutions in $A$ of the
equation $\Pi (\mathcal{S})\mathcal{\varphi }(A)=A\gamma (\Pi (\mathcal{S}%
)). $ They are equal by the uniqueness part of Lemma \ref{main technical
lemma}. The second equation follows from conclusion (ii) of the same lemma. $%
~$
\end{proof}

\noindent For any $\vec{a}\in \mathfrak{m}_{E}^{\left\vert \mathcal{S}%
\right\vert },$ we equip the module $\mathbb{N}(\vec{a})=\tbigoplus%
\limits_{i=1}^{n}\left( \mathcal{O}_{E}[[\pi ]]^{\mid \tau \mid }\right)
\eta _{i}$ with semilinear $\varphi $ and $\Gamma _{K}$-actions defined by $%
(\varphi (\eta _{1}),\varphi (\eta _{2}),...,\varphi (\eta _{n}))=(\eta
_{1},\eta _{2},...,\eta _{n})\Pi (\vec{a})\ $and$\ (\gamma (\eta
_{1}),\gamma (\eta _{2}),...,\gamma (\eta _{n}))=(\eta _{1},\eta
_{2},...,\eta _{n})G_{\gamma }(\vec{a})$ for any $\gamma \in \Gamma _{K}.$
Proposition \ref{gamma acts} implies that $(\gamma _{1}\gamma _{2})x$%
\noindent $=\gamma _{1}(\gamma _{2}x)$ and $\varphi (\gamma x)=\gamma
(\varphi (x))$ for all $x\in \mathbb{N}(\vec{a})\ $and $\gamma ,\gamma
_{1},\gamma _{2}\in \Gamma _{K}.$ Since $G_{\gamma }(\vec{a})\equiv 
\overrightarrow{Id}\func{mod}\vec{\pi},$ it follows that $\Gamma _{K}$ acts
trivially on $\mathbb{N}(\vec{a})/\pi \mathbb{N}(\vec{a}).$

\begin{proposition}
\label{rank two wach modules construction}For any $\vec{a}\in \mathfrak{m}%
_{E}^{\left\vert \mathcal{S}\right\vert },\ \mathbb{N}(\vec{a})$ equipped
with the $\varphi $ and $\Gamma _{K}$-actions defined above is a Wach module
over $\mathcal{O}_{E}[[\pi ]]^{\mid \tau \mid }$ corresponding (by Theorem $%
\ref{berger thm}$) to some $G_{K}$-stable $\mathcal{O}_{E}$-lattice in some $%
n$-dimensional, crystalline $E$-representation of $G_{K}$ with Hodge-Tate
weights in $[-k;\ 0].$
\end{proposition}

\begin{proof}
The only thing left to prove is that $q^{k}\mathbb{N}(\vec{a})\boldsymbol{%
\subset }\varphi ^{\ast }(\mathbb{N}(\vec{a})).$ \noindent Since $\det (\Pi
_{i})=c_{i}q^{k_{i}}$ we have$\ $ $\ $%
\begin{eqnarray*}
\det \Pi (\vec{a}) &=&(c_{1}q^{k_{1}},c_{2}q^{k_{2}},...,c_{0}q^{k_{0}})\ 
\text{and\ } \\
(\noindent q^{k}\eta _{1},q^{k}\eta _{2},...,q^{k}\eta _{n}) &=&\noindent
(\eta _{1},\eta _{2},...,\eta _{n})\det \Pi (\vec{a})\noindent \left(
c_{1}^{-1}q^{k-k_{1}},c_{2}^{-1}q^{k-k_{2}},...,c_{0}^{-1}q^{k-k_{0}}\right)
\\
&=&\noindent (\eta _{1},\eta _{2},...,\eta _{n})\left( \Pi (\vec{a})\cdot 
\mathrm{adj}\left( \Pi (\vec{a})\right) \right) \noindent \left(
c_{1}^{-1}q^{k-k_{1}},c_{2}^{-1}q^{k-k_{2}},...,c_{0}^{-1}q^{k-k_{0}}\right)
\\
&=&\noindent (\varphi (\eta _{1}),\varphi (\eta _{2}),...,\varphi (\eta
_{n}))\cdot \left( \mathrm{adj}\Pi (\vec{a})\right) \noindent \left(
c_{1}^{-1}q^{k-k_{1}},c_{2}^{-1}q^{k-k_{2}},...,c_{0}^{-1}q^{k-k_{0}}\right)
.
\end{eqnarray*}%
\noindent Hence $(\noindent q^{k}\eta _{1},q^{k}\eta _{2},...,q^{k}\eta
_{n})\in \varphi ^{\ast }(\mathbb{N}(\vec{a}))$ and $q^{k}\mathbb{N}(\vec{a})%
\boldsymbol{\subset }\varphi ^{\ast }(\mathbb{N}(\vec{a})).$
\end{proof}

\noindent We proceed to prove the main theorem concerning the modulo $p$
reductions of the crystalline representations corresponding to the families
of Wach modules constructed in Proposition \ref{rank two wach modules
construction}. By reduction modulo $p$ we mean reduction modulo the maximal
ideal $\mathfrak{m}_{E}$ of the ring of integers of the coefficient field $%
E. $ If $\mathrm{T}$ is a $G_{K}$-stable $\mathcal{O}_{E}$-lattice in some $%
E $-linear representation\ $V$ of $G_{K},\ $we denote by $\overline{V}%
=k_{E}\bigotimes\limits_{\mathcal{O}_{E}}\mathrm{T}$ the reduction of $V$
modulo $p,\ $where $k_{E}$ is the residue field of $\mathcal{O}_{E}.$ The
reduction $\overline{V}$ depends on the choice of the lattice $\mathrm{T},$
and a theorem of Brauer and Nesbitt asserts that the semisimplification%
\begin{equation*}
\overline{V}^{s.s.}=\left( k_{E}\tbigotimes\limits_{\mathcal{O}_{E}}\mathrm{T%
}\right) ^{s.s.}
\end{equation*}%
is independent of $\mathrm{T}.$ Instead of the precise statement
\textquotedblleft there exist $G_{K}$-stable $\mathcal{O}_{E}$-lattices $%
\mathrm{T}_{V}$ and $\mathrm{T}_{W}$ inside the $E$-linear representation\ $%
V $ and $W$ of $G_{K},$ respectively, such that $k_{E}\bigotimes\limits_{%
\mathcal{O}_{E}}\mathrm{T}_{V}\simeq k_{E}\bigotimes\limits_{\mathcal{O}_{E}}%
\mathrm{T}_{W}$\textquotedblright , we abuse notation and write $\overline{V}%
\simeq \overline{W}.$ For each $\vec{a}\in \mathfrak{m}_{E}^{\left\vert 
\mathcal{S}\right\vert },$ let $V(\vec{a})=E\otimes _{\mathcal{O}_{E}}%
\mathrm{T}(\vec{\alpha}),$ where $\mathrm{T}(\vec{\alpha})=\mathbb{T}(%
\mathbb{D}(\vec{a}))$ and $\mathbb{D}(\vec{a})=\mathbb{A}_{K,E}\bigotimes%
\limits_{\mathbb{A}_{K,E}^{+}}\mathbb{N}(\vec{a})$ (see Theorem \ref{berger
thm}). The representations $V(\vec{a})$ are $n$-dimensional crystalline $E$%
-representations of $G_{K}$ with Hodge-Tate weights in $[-k;\ 0].$
Concerning their $\func{mod}$ $p$ reductions, we have the following theorem.

\begin{theorem}
For any $\vec{a}\in \mathfrak{m}_{E}^{\left\vert \mathcal{S}\right\vert },\ $%
the isomorphism $\overline{V}(\vec{a})\simeq \overline{V}(\vec{0})$ holds.%
\label{from blz}
\end{theorem}

\begin{proof}
We prove that the $k_{E}$-linear representations of $G_{K},$ $%
k_{E}\bigotimes\limits_{\mathcal{O}_{E}}\mathrm{T}(\vec{a})$ and $%
k_{E}\bigotimes\limits_{\mathcal{O}_{E}}\mathrm{T}(\vec{0})$ are isomorphic.
Since $\Pi (\mathcal{S})$ and $\ G_{\gamma }(\mathcal{S})\in M_{n}^{\mathcal{%
S}},$ we have $G_{\gamma }(\vec{a})\equiv G_{\gamma }(\vec{0})\func{mod}%
\mathfrak{m}_{E}$ and $\Pi (\vec{a})\equiv \Pi (\vec{0})\func{mod}\mathfrak{m%
}_{E}.$ As $(\varphi ,\Gamma _{K})$-modules\ over $k_{E}((\pi ))^{\mid \tau
\mid },$ $\mathbb{D}(\vec{a})/\mathfrak{m}_{E}\mathbb{D}(\vec{a})\simeq 
\mathbb{D}(\vec{0})/\mathfrak{m}_{E}\mathbb{D}(\vec{0}).$ Hence $\mathbb{T}%
\left( \mathbb{D}(\vec{a})/\mathfrak{m}_{E}\mathbb{D}(\vec{a})\right) \simeq 
\mathbb{T}\left( \mathbb{D}(\vec{0})/\mathfrak{m}_{E}\mathbb{D}(\vec{0}%
)\right) ,$ where $\mathbb{T}$ is Fontaine's functor defined in Theorem\ \ref%
{fontaine|s equivalence}. Since Fontaine's functor is exact,$\ $\noindent
\noindent $\mathbb{T}\left( \mathbb{D}(\vec{a})/\mathfrak{m}_{E}\mathbb{D}(%
\vec{a})\right) \simeq \mathrm{T}(\vec{a})/\mathfrak{m}_{E}\mathrm{T}(\vec{a}%
)\ $and $\mathrm{T}(\vec{a})/\mathfrak{m}_{E}\mathrm{T}(\vec{a})\simeq 
\mathrm{T}(\vec{0})/\mathfrak{m}_{E}\mathrm{T}(\vec{0}).\ $
\end{proof}

\noindent This theorem enables us to explicitly compute the reductions $%
\overline{V}(\vec{a})^{s.s.}$ by computing $\overline{V}(\vec{0})^{s.s.}.$

\section{Families of two-dimensional crystalline representations \label{2-d
families}}

\noindent The main difficulty in applying Lemma \ref{outline section theorem}
is in constructing the matrices $G_{\gamma }^{(\ell )}(\mathcal{S})$ which
satisfy conditions $(1)$ and $(2).$ Conditions $(3)$ and $(4)$ are usually
easy to check. Throughout this section we retain the notations of Lemma \ref%
{outline section theorem}.\ We denote by $E_{ij}$ the $2\times 2$ matrix
with $1$ in the $(i,j)$-entry and $0$ everywhere else. One easily checks
that $E_{ij}\cdot E_{kl}=\delta _{jk}\cdot E_{il},$ where $\delta $ is the
Kronecker delta function. Recall our assumption that at least one of the
weights $k_{i}$ is strictly positive.

\begin{proposition}
\label{prop operator in 4 types case}The operator $\overline{H}\mapsto 
\overline{H-Q_{f}H(p^{f\ell }Q_{f}^{-1})}:\widetilde{M}_{2}\rightarrow 
\widetilde{M}_{2}$ is surjective, unless $\ell =k,\ k=k_{i}$ for all $i$ and 
$\bar{Q}_{f}\in \{E_{11},E_{22}\}.$
\end{proposition}

\begin{proof}
It is straightforward to check that $\bar{Q}_{f}=E_{ij}$ for some $i,j\in
\{1,2\}$ and 
\begin{equation*}
p^{k\ell }Q_{f}^{-1}\func{mod}I=\left\{ 
\begin{array}{c}
\ E_{22}\ \ \ \text{if }\bar{Q}_{f}=E_{11}, \\ 
\ E_{11}\ \ \ \text{if }\bar{Q}_{f}=E_{22}, \\ 
-E_{12}\ \ \text{if }\bar{Q}_{f}=E_{12}, \\ 
-E_{21}\ \ \text{if }\bar{Q}_{f}=E_{21}.%
\end{array}%
\right.
\end{equation*}%
If $\bar{Q}_{f}=E_{11}$ (respectively $E_{22}),$ the image is the set of
matrices with zero $\left( 1,2\right) $ (respectively $\left( 2,1\right) )$
entry, while if $\bar{Q}_{f}=E_{12}$ or $\bar{Q}_{f}=E_{21}$ the operator
becomes%
\begin{equation*}
\left( 
\begin{array}{cc}
h_{11} & h_{12} \\ 
h_{21} & h_{22}%
\end{array}%
\right) \longmapsto \left( 
\begin{array}{cc}
h_{11} & h_{12}+h_{21} \\ 
h_{21} & h_{22}%
\end{array}%
\right)
\end{equation*}%
and%
\begin{equation*}
\left( 
\begin{array}{cc}
h_{11} & h_{12} \\ 
h_{21} & h_{22}%
\end{array}%
\right) \longmapsto \left( 
\begin{array}{cc}
h_{11} & h_{12} \\ 
h_{21}+h_{12} & h_{22}%
\end{array}%
\right)
\end{equation*}%
respectively which is clearly surjective. The proposition follows from
Proposition \ref{prop operator}.
\end{proof}

\noindent In the two-dimensional case, instead of checking condition $(3)$
of Lemma \ref{outline section theorem}, it is often more convenient to use
following lemma.

\begin{lemma}
\label{eigenvalues lemma}If the matrix $Q_{f}=P_{1}P_{2}\cdots P_{f-1}P_{f}$
(with $P_{f}=P_{0}$) does not have eigenvalues which are a scalar multiple
of each other, then the matrix $U=\mathrm{Nm}_{\varphi }(P),$ where $%
P=\left( P_{1},P_{2},...,P_{f-1},P_{0}\right) ,$ satisfies condition $(3)$
of Lemma $\ref{outline section theorem}$.
\end{lemma}

\begin{proof}
Let $H\in M_{n}(\mathcal{O}_{E}[[\mathcal{S}]]^{\mid \tau \mid })$ be a
nonzero matrix such $HU=p^{ft}UH$ for some $t>0.$ We write $H=\left(
H_{1},H_{2},...,H_{f}\right) $ and $U=\left( U_{1},U_{2},...,U_{f}\right) .$
Since $P\cdot \varphi (U)\cdot P^{-1}=U,\ P_{i}U_{i+1}P_{i}^{-1}=U_{i}$ for
all $i.$ Since $Q_{f}=U_{1},$ none of the $U_{i}$ has eigenvalues which are
a scalar multiple of each other. If $H$ is invertible then $U_{1}=Q_{f}$ has
eigenvalues with quotient $p^{ft}$ which contradicts the assumption of the
lemma. If $H$ is nonzero and not invertible, then there exists an index $i$
such that $H_{i}U_{i}=p^{ft}$ $U_{i}H_{i}$ and $\mathrm{rank}\left(
H_{i}\right) =1.$ There also exists invertible matrix $B$ such that%
\begin{equation*}
BH_{i}B^{-1}=\left( 
\begin{array}{cc}
\alpha _{11} & 0 \\ 
\alpha _{21} & 0%
\end{array}%
\right)
\end{equation*}%
with $\left( \alpha _{11},\alpha _{21}\right) \neq \left( 0,0\right) .$ Let $%
\Gamma =BU_{i}B^{-1}$ and write $\Gamma =\left( \gamma _{ij}\right) .$ The
equation $H_{i}U_{i}=p^{ft}U_{i}H_{i}$ is equivalent to $p^{ft}\Gamma
BH_{i}B^{-1}=BH_{i}B^{-1}\Gamma $ which implies that $\gamma _{12}=0$ and $%
p^{ft}\gamma _{11}\alpha _{11}=\alpha _{11}\gamma _{11}.$ If $\alpha
_{11}\neq 0,$ then $\gamma _{11}=0$ a contradiction since $\Gamma $ is
invertible. If $\alpha _{11}=0,$ then $p^{ft}\alpha _{21}\gamma _{22}=\alpha
_{21}\gamma _{11}$ and $p^{ft}\gamma _{22}=\gamma _{11}$ (since $\alpha
_{21}\neq 0).$ Since $\gamma _{12}=0,$ the latter implies that $\Gamma $ has
two eigenvalues with quotient $p^{ft}.$ This in turn implies that $U_{i}$
and its conjugate $Q_{f}=U_{1}$ have eigenvalues with quotient $p^{ft}$ and
contradicts the assumption of the lemma. Hence $H=0.$\noindent
\end{proof}

\begin{corollary}
\label{eigenvalies cor}If $\mathrm{Tr}\left( Q_{f}\right) \not\in \bar{%
\mathbb{Q}%
}_{p},$ then the matrix $U=$ $\mathrm{Nm}_{\varphi }(P)$ satisfies condition 
$(3)$ of Lemma $\ref{outline section theorem}$.
\end{corollary}

\begin{proof}
Since the determinant of $Q_{f}$ is a nonzero scalar, the eigenvalues of $%
Q_{f}$ are a scalar multiple of each other if and only if $\mathrm{Tr}\left(
Q_{f}\right) $ is a scalar.\noindent
\end{proof}

\subsection{\noindent Families of rank two Wach modules\label{the 4 type
representations}}

\noindent We now apply Lemma \ref{outline section theorem} for matrices $\Pi
_{i}$ as in the following definition.

\begin{definition}
\textit{\label{the 8 types copy(1)}}For a fixed integer $\ell \geq k=\max
\{k_{0},k_{1},...,k_{f-1}\}$\ we define matrices of the following four types:%
\begin{equation*}
\text{t}_{1}\mathbf{:}\mathbf{\ }\left( 
\begin{array}{cc}
c_{i}q^{k_{i}} & 0 \\ 
X_{i}\varphi (z_{i}) & 1%
\end{array}%
\right) ,\ \text{t}_{2}\mathbf{:\ }\left( 
\begin{array}{cc}
X_{i}\varphi (z_{i}) & 1 \\ 
c_{i}q^{k_{i}} & 0%
\end{array}%
\right) ,\ \text{t}_{3}\mathbf{:}\mathbf{\ }\left( 
\begin{array}{cc}
1 & X_{i}\varphi (z_{i}) \\ 
0 & c_{i}q^{k_{i}}%
\end{array}%
\right) ,\ \text{t}_{4}\mathbf{:\ }\left( 
\begin{array}{cc}
0 & c_{i}q^{k_{i}} \\ 
1 & X_{i}\varphi (z_{i})%
\end{array}%
\right) ,
\end{equation*}%
where $X_{i}$\ is an indeterminate, $c_{i}\in \mathcal{O}_{E},\ $and$\ z_{i}$%
\ is a polynomial of degree $\leq \ell -1$\ in $%
\mathbb{Z}
_{p}[\pi ]\ $such that $z_{i}\equiv p^{m_{\ell }}\func{mod}\pi ,$\ where $%
m_{\ell }:=\lfloor \frac{\ell -1}{p-1}\rfloor .$\ Matrices of type $t_{1}$
or $t_{3}$ are called of odd type, while matrices of type $t_{2}\ $or $t_{4}$
are called of even type. We write $\Pi ^{\vec{i}}(S)=\left( \Pi
_{1}(X_{1}),\Pi _{2}(X_{2}),...,\Pi _{f-1}(X_{f-1}),\Pi _{0}(X_{0})\right) $
with $\vec{i}=(i_{1},i_{2},...,i_{f-1},i_{0})$ the vector in $%
\{1,2,3,4\}^{f} $ whose $j$-th coordinate $i_{j}$ is the type of the matrix $%
\Pi _{j}$ for all $j\in I_{0}.$ We call $\vec{i}$ the type-vector of the $f$%
-tuple $\Pi ^{\vec{i}}(S).$
\end{definition}

\noindent \indent The polynomials $z_{i}$ appearing in the entries of the
matrices $\Pi _{i}$ will be defined shortly. We will also define functions $%
x_{i}^{\gamma },$ $y_{i}^{\gamma }\in 1+\pi 
\mathbb{Z}
_{p}[[\pi ]]$ such that $G_{\gamma }^{(\ell )}-\Pi (\mathcal{S})\varphi
(G_{\gamma }^{(\ell )})\gamma (\Pi (\mathcal{S})^{-1})\in \vec{\pi}^{\ell
}M_{2}^{\mathcal{S}}\ $for all $\gamma \in \Gamma _{K},\ $where\ 
\begin{equation*}
G_{\gamma }^{(\ell )}=\mathrm{diag}\left( \left( x_{0}^{\gamma
},x_{1}^{\gamma },...,x_{f-1}^{\gamma }\right) ,\left( y_{0}^{\gamma
},y_{1}^{\gamma },...,y_{f-1}^{\gamma }\right) \right) .
\end{equation*}%
Let%
\begin{equation*}
\Pi (\mathcal{S})=\left( 
\begin{array}{ll}
\left( \alpha _{1},\alpha _{2},...\alpha _{f-1},\alpha _{0}\right) & \left(
\beta _{1},\beta _{2},...,\beta _{f-1},\beta _{0}\right) \\ 
\left( \gamma _{1},\gamma _{2},...,\gamma _{f-1},\gamma _{0}\right) & \left(
\delta _{1},\delta _{2},...,\ \delta _{f-1},\delta _{0}\right)%
\end{array}%
\right)
\end{equation*}%
with 
\begin{equation*}
\left( 
\begin{array}{cc}
\alpha _{i} & \beta _{i} \\ 
\gamma _{i} & \delta _{i}%
\end{array}%
\right) \in \left\{ \left( 
\begin{array}{cc}
c_{i}q^{k_{i}} & 0 \\ 
X_{i}\varphi (z_{i}) & 1%
\end{array}%
\right) ,\ \left( 
\begin{array}{cc}
X_{i}\varphi (z_{i}) & 1 \\ 
c_{i}q^{k_{i}} & 0%
\end{array}%
\right) ,~\left( 
\begin{array}{cc}
1 & X_{i}\varphi (z_{i}) \\ 
0 & c_{i}q^{k_{i}}%
\end{array}%
\right) ,\ \left( 
\begin{array}{cc}
0 & c_{i}q^{k_{i}} \\ 
1 & X_{i}\varphi (z_{i})%
\end{array}%
\right) \right\} .
\end{equation*}%
For each $i=1,2,...,f$ we demand that all of the elements%
\begin{eqnarray}
&&x_{i-1}^{\gamma }-\frac{\alpha _{i}\varphi \left( x_{i}^{\gamma }\right)
\left( \gamma \delta _{i}\right) -\beta _{i}\varphi \left( y_{i}^{\gamma
}\right) \left( \gamma \gamma _{i}\right) }{\varepsilon _{i}\left( \gamma
q\right) ^{k_{i}}},\ \frac{\beta _{i}\varphi \left( y_{i}^{\gamma }\right)
\left( \gamma \alpha _{i}\right) -\alpha _{i}\varphi \left( x_{i}^{\gamma
}\right) \left( \gamma \beta _{i}\right) }{\varepsilon _{i}\left( \gamma
q\right) ^{k_{i}}},  \label{solve for G} \\
&&y_{i-1}^{\gamma }-\frac{\delta _{i}\varphi \left( y_{i}^{\gamma }\right)
\left( \gamma \alpha _{i}\right) -\gamma _{i}\varphi \left( x_{i}^{\gamma
}\right) \left( \gamma \beta _{i}\right) }{\varepsilon _{i}\left( \gamma
q\right) ^{k_{i}}},\ \frac{\gamma _{i}\varphi \left( x_{i}^{\gamma }\right)
\left( \gamma \delta _{i}\right) -\delta _{i}\varphi \left( y_{i}^{\gamma
}\right) \left( \gamma \gamma _{i}\right) }{\varepsilon _{i}\left( \gamma
q\right) ^{k_{i}}}  \label{solve for G 1}
\end{eqnarray}

\noindent of$\ \mathcal{O}_{E}[[\pi ,X_{0},...,X_{f-1}]][q^{-1}]\ $which
belong to $\mathcal{O}_{E}[[\pi ]][q^{-1}]$ are zero, and those which
contain an indeterminate belong to $\pi ^{\ell }\mathcal{O}_{E}[[\pi
,X_{0},...,X_{f-1}]],$ where in the formulas above $\varepsilon _{i}=1$ if $%
\Pi _{i}$ has type $1\ $or $3$ and $\varepsilon _{i}=-1$ if $\Pi _{i}$ has
type $2$ or $4.$ As usual, lower indices are viewed modulo $f.$

\begin{proposition}
\label{newest prop}For each $i,$ equations $\ref{solve for G}$ and $\ref%
{solve for G 1}$ imply that%
\begin{equation}
x_{i-1}^{\gamma }=\left( \frac{q}{\gamma q}\right) ^{\ell _{i}}\varphi
\left( w_{i}^{\gamma }\right) \ \text{and}\ y_{i-1}^{\gamma }=\left( \frac{q%
}{\gamma q}\right) ^{\ell _{i}^{\prime }}\varphi \left( \left( w_{i}^{\gamma
}\right) ^{^{\prime }}\right) ,  \label{1xy}
\end{equation}%
with$\ \ell _{i}\in \{0,k_{i}\},\ w_{i}^{\gamma }\in \{x_{i}^{\gamma
},y_{i}^{\gamma }\},\ \ell _{i}^{\prime }=k_{i}-\ell _{i},\ $and $\left(
w_{i}^{\gamma }\right) ^{^{\prime }}=\left\{ 
\begin{array}{l}
x_{i}^{\gamma }\ \text{if }w_{i}^{\gamma }=y_{i}^{\gamma }, \\ 
y_{i}^{\gamma }\ \text{if }w_{i}^{\gamma }=x_{i}^{\gamma }.%
\end{array}%
\right. $
\end{proposition}

\begin{proof}
If $\Pi _{i}$ is of type $1,$ then $\beta _{i}=0,$ $\alpha
_{i}=c_{i}q^{k_{i}}$ and $\delta _{i}=1$ and we must have $q^{k_{i}}\varphi
\left( x_{i}^{\gamma }\right) =x_{i-1}^{\gamma }\left( \gamma q\right)
^{k_{i}}$ and $\varphi \left( y_{i}^{\gamma }\right) =y_{i-1}^{\gamma }.$
The proposition holds with $\ell _{i}=k_{i},$ $w_{i}^{\gamma }=x_{i}^{\gamma
},$ and $\ell _{i}^{\prime }=0,$ $\left( w_{i}^{\gamma }\right) ^{^{\prime
}}=y_{i}^{\gamma }.$ If $\Pi _{i}$ is of type $2,$ then $\delta _{i}=0,$ $%
\beta _{i}=1,$ $\gamma _{i}=c_{i}q^{k_{i}}$ and we must have $\varphi \left(
y_{i}^{\gamma }\right) =x_{i-1}^{\gamma }$ and $q^{k_{i}}\varphi \left(
x_{i}^{\gamma }\right) =y_{i-1}^{\gamma }\left( \gamma q\right) ^{k_{i}}.$
The proposition holds with $\ell _{i}=0,$ $w_{i}^{\gamma }=y_{i}^{\gamma },$
and $\ell _{i}^{\prime }=k_{i},$ $\ \left( w_{i}^{\gamma }\right) ^{^{\prime
}}=x_{i}^{\gamma }.$ The cases where $\Pi _{i}$ is of type $3$ or $4$ are
identical.
\end{proof}

\noindent From Proposition \ref{newest prop} it follows that 
\begin{equation}
x_{0}^{\gamma }=\left( \tprod\limits_{i=0}^{f-1}\varphi ^{i}\left( \frac{q}{%
\gamma q}\right) ^{s_{i}}\right) \varphi ^{f}\left( z_{f}^{\gamma }\right) \ 
\text{and}\ y_{0}^{\gamma }=\left( \tprod\limits_{i=0}^{f-1}\varphi
^{i}\left( \frac{q}{\gamma q}\right) ^{s_{i}^{\prime }}\right) \varphi
^{f}\left( \left( z_{f}^{\gamma }\right) ^{\prime }\right) ,
\label{new x, y eq}
\end{equation}%
with $s_{i}^{\prime },s_{i}\in \{\ell _{i},\ell _{i}^{\prime }\}.$ If $%
z_{f}^{\gamma }=x_{0}^{\gamma },$ then $\left( z_{f}^{\gamma }\right)
^{\prime }=y_{0}^{\gamma }$ and by Lemma \ref{solve}, equations \ref{new x,
y eq} have unique $\equiv 1\func{mod}\pi $ solutions given by 
\begin{equation}
x_{0}^{\gamma }=\tprod\limits_{i=0}^{f-1}\varphi ^{i}\left( \lambda
_{f,\gamma }\right) ^{s_{i}}\ \text{and}\ y_{0}^{\gamma
}=\tprod\limits_{i=0}^{f-1}\varphi ^{i}\left( \lambda _{f,\gamma }\right)
^{s_{i}^{\prime }}.  \label{xf}
\end{equation}%
If If $z_{f}^{\gamma }=y_{0}^{\gamma },$ then $\left( z_{f}^{\gamma }\right)
^{\prime }=x_{0}^{\gamma }$ and equations \ref{new x, y eq} imply that$\ $%
\begin{align}
x_{0}^{\gamma }& =\tprod\limits_{i=0}^{f-1}\left( \varphi ^{i}\left( \frac{q%
}{\gamma q}\right) ^{s_{i}}\cdot \varphi ^{i+f}\left( \frac{q}{\gamma q}%
\right) ^{s_{i}^{\prime }}\right) \varphi ^{2f}\left( x_{0}^{\gamma }\right)
,\   \label{axy} \\
\ y_{0}^{\gamma }& =\tprod\limits_{i=0}^{f-1}\left( \varphi ^{i}\left( \frac{%
q}{\gamma q}\right) ^{s_{i}^{\prime }}\cdot \varphi ^{i+f}\left( \frac{q}{%
\gamma q}\right) ^{s_{i}}\right) \varphi ^{2f}\left( y_{0}^{\gamma }\right) ,
\label{bxy}
\end{align}%
which by Lemma \ref{solve} have unique $\equiv 1\func{mod}\pi $ solutions
given by 
\begin{align}
x_{0}^{\gamma }& =\tprod\limits_{i=0}^{f-1}\left( \varphi ^{i}\left( \lambda
_{2f,\gamma }\right) ^{s_{i}}\cdot \varphi ^{i+f}\left( \lambda _{2f,\gamma
}\right) ^{s_{i}^{\prime }}\right) ,  \label{x2f} \\
y_{0}^{\gamma }& =\tprod\limits_{i=0}^{f-1}\left( \varphi ^{i}\left( \lambda
_{2f,\gamma }\right) ^{s_{i}^{\prime }}\cdot \varphi ^{i+f}\left( \lambda
_{2f,\gamma }\right) ^{s_{i}^{\prime }}\right) .  \label{y2f}
\end{align}%
Equations \ref{1xy} for $i=f$ give the unique $\equiv 1\func{mod}\pi $
solutions for $x_{f-1}^{\gamma }$ and $y_{f-1}^{\gamma },$ and continuing
for $i=f-1,f-2,...,2,$ we get the unique $\equiv 1\func{mod}\pi $ solutions
for $x_{i}^{\gamma }$ and $y_{i}^{\gamma }.$ We now define the polynomials $%
z_{i}$ so that for each $\gamma \in \Gamma _{K},$ the matrix $G_{\gamma
}^{(\ell )}\equiv \overrightarrow{Id}\func{mod}\vec{\pi}$ satisfies the
congruence $G_{\gamma }^{(\ell )}-\Pi (\mathcal{S})\varphi (G_{\gamma
}^{(\ell )})\gamma (\Pi (\mathcal{S})^{-1})\in \vec{\pi}^{\ell }M_{2}^{%
\mathcal{S}}.$

\begin{lemma}
\noindent \label{the z corollary}Let $\mathcal{R}=\{\sum\limits_{i\geq
0}a_{i}\pi ^{i}\in 
\mathbb{Q}
_{p}[[\pi ]]:\mathrm{v}_{\mathrm{p}}(a_{i})+\frac{i}{p-1}\geq 0\ for\ all\
i\geq 0\}.\ $The set $\mathcal{R}$ endowed with the addition and the
multiplication of $%
\mathbb{Q}
_{p}[[\pi ]]$ is a subring of $%
\mathbb{Q}
_{p}[[\pi ]]$ which is stable under the $\varphi $ and the $\Gamma _{K}$%
-actions. Moreover,\smallskip

\begin{enumerate}
\item[(i)] $(\frac{q_{n}}{p})^{\pm 1}\in \mathcal{R}$ for all $n\geq 1\ $and 
$(\lambda _{f})^{\pm 1}\in \mathcal{R}\ $for all $f\geq 1,$ and

\item[(ii)] Let $b=cp^{N}b^{\ast },$ where $c\in \mathcal{O}_{E}^{\times },$ 
$n\in 
\mathbb{Z}
,$ and $b^{\ast }\in $$\mathcal{R}\setminus \{0\}\mathcal{\ }$is such that $%
\frac{b^{\ast }}{\gamma b^{\ast }}\in 1+\pi 
\mathbb{Z}
_{p}[[\pi ]]\ $for all $\gamma \in \Gamma _{K}.$ If $\ell \geq 1$ is a fixed
integer, there exists some polynomial $z=z(\ell ,b)\in 
\mathbb{Z}
_{p}[\pi ]$ with $\deg _{\pi }z\leq \ell -1$ and $z\equiv p^{m_{\ell }}\func{%
mod}\pi ,\ $where $m_{\ell }=\lfloor \frac{\ell -1}{p-1}\rfloor ,\ $such
that $z-\gamma z\frac{b}{\gamma b}\in \pi ^{\ell }%
\mathbb{Z}
_{p}[[\pi ]]$ for all $\gamma \in \Gamma _{K}.$
\end{enumerate}
\end{lemma}

\begin{proof}
We notice that the coefficients $a_{i}$ of $\pi ^{i}$ in $\frac{q}{p}$ are
such that $\mathrm{v}_{\mathrm{p}}(a_{i})+\frac{i}{p-1}\geq 0$ for all $%
i=0,1,...$ Motivated by this we consider the set $\mathcal{R}$ of all
functions of $%
\mathbb{Q}
_{p}[[\pi ]]$ with the same property. This is a ring with the obvious
operations, stable under $\varphi $ and $\Gamma _{K}.$ One easily checks
that $\left( \frac{p}{q}\right) ^{\pm 1}$ $\in \mathcal{R}$ and therefore $%
\left( \frac{q_{n}}{p}\right) ^{\pm 1}\in \mathcal{R}$ for all $n\geq 1$
from which (i) follows easily. (ii) Since $\Gamma _{K}$ acts trivially on $%
\mathcal{O}_{E}^{\times }$ we may replace $b$ by $c^{-1}b$ and assume that $%
c=1.$ We write $b=p^{n}b^{\ast }.$ Let $p^{m}b=z+a,$ where $a\in \pi ^{\ell }%
\mathbb{Q}
_{p}[[\pi ]]\ $and $\deg _{\pi }z\leq \ell -1,$ for integer $m$ which will
be chosen large enough so that $z\in 
\mathbb{Z}
_{p}[\pi ].$ Let $z=\tsum\limits_{j=0}^{\ell -1}z_{j}\pi ^{j}.\ $Since $%
p^{m+n}b^{\ast }=z+a$ and $b^{\ast }\in \mathcal{R},$ we have $\mathrm{v}_{%
\mathrm{p}}(z_{j})-m-n+\frac{j}{p-1}\geq 0$ for all $j\geq 0.$ We need $%
\mathrm{v}_{\mathrm{p}}(z_{j})>-1$ for all $j=0,1,...,\ell -1$ and it
suffices to have $m+n-\frac{\ell -1}{p-1}>-1.$ We choose $m=\lfloor \frac{%
\ell -1}{p-1}\rfloor -n.$ Then $z\in 
\mathbb{Z}
_{p}\left[ \pi \right] ,$ $\deg _{\pi }z\leq \ell -1$ and $z\equiv
p^{m+n}=p^{m_{\ell }}\func{mod}\pi ,$ For any $\gamma \in \Gamma _{K},$ $%
z-\gamma z\frac{b}{\gamma b}=p^{m}b-a-b\gamma (b^{-1})\left( p^{m}(\gamma
b)-\gamma a\right) =$ $b\gamma (b^{-1})\gamma a-a\in \pi ^{\ell }%
\mathbb{Q}
_{p}[[\pi ]].$ Since $z\in 
\mathbb{Z}
_{p}\left[ \pi \right] $ and $b\gamma (b^{-1})\in 1+\pi 
\mathbb{Z}
_{p}[\left[ \pi \right] ],\ $we have $z-\gamma z\frac{b}{\gamma b}\in \pi
^{\ell }%
\mathbb{Z}
_{p}\left[ \left[ \pi \right] \right] $ $=%
\mathbb{Z}
_{p}\left[ \left[ \pi \right] \right] \cap \pi ^{\ell }%
\mathbb{Q}
_{p}[[\pi ]]$ for all $\gamma \in \Gamma _{K}.$
\end{proof}

\begin{lemma}
\label{new technical}For each $\gamma \in \Gamma _{K}$ and $i\in I_{0},$

\begin{enumerate}
\item[(i)] $x_{i}^{\gamma },y_{i}^{\gamma }\in 1+\pi 
\mathbb{Z}
_{p}\left[ \left[ \pi \right] \right] ;$

\item[(ii)] $x_{i}^{\gamma }=\frac{a_{i}}{\gamma a_{i}}$ and $y_{i}^{\gamma
}=\frac{b_{i}}{\gamma b_{i}}$ for some $a_{i}$ and $b_{i}$ with $\left(
a_{i}\right) ^{\pm 1}\ $and $\left( b_{i}\right) ^{\pm 1}\in \mathcal{R}.$
\end{enumerate}
\end{lemma}

\begin{proof}
(i) is clear by the definition of the $x_{i}^{\gamma },y_{i}^{\gamma }$ and
Lemma \ref{proto}. (ii) For $i=0,$ if $z_{f}^{\gamma }=x_{0}^{\gamma },\ $%
then by equation \ref{xf}, $x_{0}^{\gamma }=\frac{a_{0}}{\gamma a_{0}},\ $%
where $a_{0}=\tprod\limits_{i=0}^{f-1}\varphi ^{i}\left( \lambda _{f}\right)
^{s_{i}}\in \mathcal{R}.$ Since $\left( \lambda _{f}\right) ^{\pm 1}\in 
\mathcal{R}$ and $\mathcal{R}$ is $\varphi $-stable, $\left( a_{0}\right)
^{\pm 1}\in \mathcal{R}.$ If $z_{f}^{\gamma }=y_{0}^{\gamma },$ then by
equation \ref{x2f},$\ x_{0}^{\gamma }=\frac{a_{0}}{\gamma a_{0}},\ $where $%
a_{0}=\tprod\limits_{i=0}^{f-1}\left( \varphi ^{i}\left( \lambda _{f}\right)
^{s_{i}}\varphi ^{i+f}\left( \lambda _{f}\right) ^{s_{i}^{\prime }}\right)
\in \mathcal{R},$ therefore $\left( a_{0}\right) ^{\pm 1}\in \mathcal{R}.$
The proof for $y_{0}^{\gamma }$ and $\left( b_{i}\right) ^{\pm 1}$ is
similar. For $x_{f-1}^{\gamma },$ notice that $x_{f-1}^{\gamma }=\left( 
\frac{q}{\gamma q}\right) ^{\ell _{i}}\varphi \left( w_{i}^{\gamma }\right) =%
\frac{\gamma \left( \varphi \left( c_{0}\right) \left( \frac{q}{p}\right)
^{\ell _{f}}\right) }{\varphi \left( c_{0}\right) \left( \frac{q}{p}\right)
^{\ell _{f}}}$ with $c_{0}\in \{a_{0},b_{0}\}.$ Since $\left( a_{0}\right)
^{\pm 1},\left( b_{0}\right) ^{\pm 1}\in \mathcal{R},$ it follows that $%
x_{f-1}^{\gamma }\in \mathcal{R}.$ Since $\left( \varphi \left( c_{0}\right)
\left( \frac{q}{p}\right) ^{\ell _{f}}\right) ^{\pm 1}\in \mathcal{R},\ $it
follows that $\left( a_{f-1}\right) ^{\pm 1}\in \mathcal{R}.$ Similarly $%
y_{f-1}^{\gamma }\ $and $\left( b_{f-1}\right) ^{\pm 1}\in \mathcal{R}.$ The
lemma follows by induction.
\end{proof}

\noindent To define the polynomials $z_{i}$ we will also need the following
lemma.

\begin{lemma}
\label{zeussent}If $\alpha \in \pi ^{\ell }\mathcal{O}_{E}[[\pi ]]\ $and $%
0\leq k\leq \ell $ is an integer, then $\frac{\varphi \left( \alpha \right) 
}{\left( \gamma q\right) ^{k}}\in \pi ^{\ell }\mathcal{O}_{E}[[\pi ]].$
\end{lemma}

\begin{proof}
Since $\frac{\gamma q}{q}\equiv 1\func{mod}\pi ,$ it suffices to prove that $%
\frac{\varphi \left( \alpha \right) }{q^{k}}\in \pi ^{\ell }\mathcal{O}%
_{E}[[\pi ]].$ Let $\alpha =\pi ^{\ell }\beta $ for some $\beta \in \mathcal{%
O}_{E}[[\pi ]].$ We have $\varphi \left( \frac{\alpha }{\pi ^{k}}\right)
=\varphi \left( \pi \right) ^{\ell -k}\varphi \left( \beta \right) =q^{\ell
-k}\pi ^{\ell -k}\varphi \left( \beta \right) .$ Hence $\frac{\varphi \left(
\alpha \right) }{q^{k}}=\pi ^{k}\varphi \left( \frac{\alpha }{\pi ^{k}}%
\right) =\pi ^{k}q^{\ell -k}\pi ^{\ell -k}\varphi \left( \beta \right) =\pi
^{\ell }q^{\ell -k}\varphi \left( \beta \right) \in \pi ^{\ell }\mathcal{O}%
_{E}[[\pi ]].$
\end{proof}

\begin{proposition}
\label{zi lemma}Let $k=\max \{k_{0},k_{1},...,k_{f-1}\},\ $let $\ell \geq k$
be a fixed integer and let $\ m_{\ell }=\lfloor \frac{\ell -1}{p-1}\rfloor .$
There exist polynomials $z_{i}\in 
\mathbb{Z}
_{p}\left[ \pi \right] $ with $\deg _{\pi }z_{i}\leq \ell -1$ such that $%
z_{i}\equiv p^{m_{\ell }}\func{mod}\pi $ with the following properties:

\begin{enumerate}
\item[(i)] $G_{\gamma }^{(\ell )}\equiv \overrightarrow{Id}\func{mod}\vec{\pi%
}$ and

\item[(ii)] $G_{\gamma }^{(\ell )}-\Pi (\mathcal{S})\varphi (G_{\gamma
}^{(\ell )})\gamma (\Pi (\mathcal{S})^{-1})\in \vec{\pi}^{\ell }M_{2}^{%
\mathcal{S}}$ for each $\gamma \in \Gamma _{K}.$
\end{enumerate}
\end{proposition}

\begin{proof}
Suppose that $P_{i}\ $is of type $2\ $and $\alpha _{i}=X_{i}\varphi (z_{i})$
for some polynomial $z_{i}$ to be defined. Then $\beta _{i}=1$ and $\ \beta
_{i}\varphi \left( y_{i}^{\gamma }\right) =x_{i-1}^{\gamma }\left( \gamma
\beta _{i}\right) $ implies $x_{i-1}^{\gamma }=\varphi \left( y_{i}^{\gamma
}\right) .$ We need%
\begin{equation*}
X_{i}\left( \varphi (z_{i})\varphi \left( x_{i}^{\gamma }\right)
-x_{i-1}^{\gamma }\varphi (\gamma z_{i})\right) \frac{1}{\left( \gamma
q\right) ^{k_{i}}}\in \pi ^{\ell }\mathcal{O}_{E}[[\pi ,X_{0},...,X_{f-1}]]
\end{equation*}%
for all $\gamma \in \Gamma _{K}.$ By Lemma \ref{zeussent} it suffices to
define $z_{i}$ so that $z_{i}x_{i}^{\gamma }-y_{i}^{\gamma }\gamma z_{i}\in
\pi ^{\ell }\mathcal{O}_{E}[[\pi ]]$ for all $\gamma \in \Gamma _{K}.\ $%
Since $x_{i}^{\gamma }\in $ $1+\pi 
\mathbb{Z}
_{p}\left[ \left[ \pi \right] \right] $ for all $\gamma \in \Gamma _{K},$
this is equivalent to $z_{i}-\frac{y_{i}^{\gamma }}{x_{i}^{\gamma }}\gamma
z_{i}\in \pi ^{\ell }\mathcal{O}_{E}[[\pi ]].$ By Lemma \ref{new technical}, 
$\frac{y_{i}^{\gamma }}{x_{i}^{\gamma }}=\frac{b}{\gamma b},$ where $%
b=a_{i}\left( b_{i}\right) ^{-1}\in \mathcal{R}.$ Since$\ \frac{%
y_{i}^{\gamma }}{x_{i}^{\gamma }}\in 1+\pi 
\mathbb{Z}
_{p}\left[ \left[ \pi \right] \right] ,$ the existence of the $z_{i}$
follows from Lemmata \ref{the z corollary} and \ref{new technical}. The
proof for $P_{i}$ of type $1,3$ and $4$ is identical.\textrm{\textrm{\ }}%
\noindent
\end{proof}

\begin{proposition}
\bigskip \label{p remark}If $\alpha \left( \pi \right)
=\tsum\limits_{n=0}^{\infty }\alpha _{n}\pi ^{n}\in 
\mathbb{Q}
_{p}[[\pi ]]$ is such that $\mathrm{v}_{\mathrm{p}}\left( \alpha _{i}\right)
\geq 0$ for all $i=0,1,2,...,p-2$ and $\mathrm{v}_{\mathrm{p}}\left( \alpha
_{p-1}\right) >-1,$ then the first $p-1$ coefficients of $\alpha \left( \pi
\right) ^{p}$ are in $%
\mathbb{Z}
_{p}.$ In particular, the first $p-1$ coefficients of the $p$-th power of
any element of $\mathcal{R}$ are in $%
\mathbb{Z}
_{p}.$
\end{proposition}

\begin{proof}
Follows easily using the binomial expansion.
\end{proof}

\begin{proposition}
\label{1nov}If $k_{i}=p$ for all $i,$ then there exist polynomials $z_{i}\in 
\mathbb{Z}
_{p}\left[ \pi \right] $ with $\deg _{\pi }z_{i}\leq p-1$ such that $%
z_{i}\equiv 1\func{mod}\pi ,$ and such that $G_{\gamma }^{(p)}-\Pi (\mathcal{%
S})\varphi (G_{\gamma }^{(p)})\gamma (\Pi (\mathcal{S})^{-1})\in \vec{\pi}%
^{p}M_{2}^{\mathcal{S}}$ for any $\gamma \in \Gamma _{K}.$
\end{proposition}

\begin{proof}
Assume that $P_{i}$ is of type $2$ and let $x_{0}^{\gamma }$ and $%
y_{0}^{\gamma }$ be as in the proof of Proposition \ref{zi lemma}. First we
notice that the exponents $s_{i}\ $and$\ s_{i}^{\prime }$ in formulas \ref%
{xf} or \ref{x2f} and \ref{y2f} for the $x_{0}^{\gamma }$ and $y_{0}^{\gamma
}$ are either $0$ or $p.$ With the notation of Lemma \ref{new technical}, $%
\frac{y_{0}^{\gamma }}{x_{0}^{\gamma }}=c_{0}\left( \gamma c_{0}^{-1}\right)
,$ where $c_{0}=a_{0}^{-1}b_{0}.$ The formulas for $a_{0}^{-1}$ and $b_{0}$
in the proof of Lemma \ref{new technical} imply that they are both $p$-th
powers of elements of $\mathcal{R}.$ From the same formulas and Lemma \ref%
{proto} it follows that $a_{0}^{-1}\left( 0\right) =b_{0}\left( 0\right) =1.$
By Lemma \ref{p remark}, $c_{0}=z_{0}+a$ for some polynomial $z_{0}\in 
\mathbb{Z}
_{p}[\pi ]$ of degree $\leq p-1$ and constant term $1$ and some $a\in \pi
^{p}%
\mathbb{Q}
_{p}[[\pi ]].$ For any $\gamma \in \Gamma _{K},$ $z_{0}-\frac{y_{0}^{\gamma }%
}{x_{0}^{\gamma }}\gamma z_{0}=c_{0}-a-c_{0}\left( \gamma c_{0}^{-1}\right)
\left( \gamma c_{0}-\gamma a\right) =c_{0}\left( \gamma c_{0}^{-1}\right)
\gamma a-a\in \pi ^{p}%
\mathbb{Q}
_{p}[[\pi ]].$ Since $\frac{y_{0}^{\gamma }}{x_{0}^{\gamma }}\in 1+\pi 
\mathbb{Z}
_{p}[[\pi ]]$ and $z_{0}\in 
\mathbb{Z}
_{p}[\pi ],$ $z_{0}-\frac{y_{0}^{\gamma }}{x_{0}^{\gamma }}\gamma z_{0}\in 
\mathbb{Z}
_{p}[[\pi ]]\cap \pi ^{p}%
\mathbb{Q}
_{p}[[\pi ]]=\pi ^{p}%
\mathbb{Z}
_{p}[[\pi ]].$ The proof for the other $z_{i}$ is similar, using formulas %
\ref{1xy} and noticing that $\left( \frac{q}{\gamma q}\right) ^{\pm 1}\in
1+\pi 
\mathbb{Z}
_{p}[[\pi ]].$ The proof for $P_{i}$ of type $1,$ $3$ or $4$ is identical.
\end{proof}

\begin{remark}
\label{k=p remark}If $k_{i}=p$ for all $i,$ then there exist polynomials $%
z_{i}\in 
\mathbb{Z}
_{p}\left[ \pi \right] $ with $\deg _{\pi }z_{i}\leq p-1$ such that $%
z_{i}\equiv 1\func{mod}\pi $ satisfying properties (i) and (ii) of
Proposition \ref{zi lemma}. This follows immediately from Proposition \ref%
{1nov}.
\end{remark}

\noindent Next, we explicitly determine when $\mathrm{Tr}\left( Q_{f}\right)
\not\in \bar{%
\mathbb{Q}%
}_{p}.$ We first need some definitions.

\begin{definition}

\begin{enumerate}
\item[(i)] \label{stupid types def}We define $C_{1}$ to be the set of $f$%
-tuples $\left( P_{1},P_{2},...,P_{f}\right) $ where the types of the
matrices $P_{i}$ are chosen as follows: $P_{1}\in \{t_{2},t_{3}\}.$ For $%
i=2,3,...,f-1,$ $P_{i}\in \{t_{2},t_{3}\}$ if an even number of coordinates
of $\left( P_{1},P_{2},...,P_{i-1}\right) $ is of even type, and $P_{i}\in
\{t_{1},t_{4}\}$ if an odd number of coordinates of $\left(
P_{1},P_{2},...,P_{i-1}\right) $ is of even type. Finally, $P_{0}=t_{3}$ if
an even number of coordinates of $\left( P_{1},P_{2},...,P_{f-1}\right) $ is
of even type, and $P_{0}=t_{4}$ otherwise.

\item[(ii)] We define $C_{2}$ to be the set of $f$-tuples $\left(
P_{1},P_{2},...,P_{f}\right) $ where the types of the matrices $P_{i}$ are
chosen as follows: $P_{1}\in \{t_{1},t_{4}\}.$ For $i=2,3,...,f-1,$ $%
P_{i}\in \{t_{1},t_{4}\}$ if an even number of coordinates of $\left(
P_{1},P_{2},...,P_{i-1}\right) $ is of even type, and $P_{i}\in
\{t_{2},t_{3}\}$ if an odd number of coordinates of $\left(
P_{1},P_{2},...,P_{i-1}\right) $ is of even type. Finally, $P_{0}=t_{1}$ if
an even number of coordinates of $\left( P_{1},P_{2},...,P_{f-1}\right) $ is
of even type, and $P_{0}=t_{2}$ otherwise.
\end{enumerate}
\end{definition}

\noindent In Definition \ref{stupid types def} the type of the matrix $P_{0}$
has been chosen so that an even number of coordinates of the $f$-tuple $%
\left( P_{1},P_{2},...,P_{f-1},P_{0}\right) $ is of even type.

\begin{definition}

\begin{enumerate}
\item[(i)] \label{stupid types def1}We define $C_{1}^{\ast }$ to be the set
of $f$-tuples $\left( P_{1},P_{2},...,P_{f}\right) $ where the types of the
matrices $P_{i}$ are chosen as follows: $P_{1}\in \{t_{2},t_{3}\}.$ For $%
i=2,3,...,f-1,$ $P_{i}\in \{t_{2},t_{3}\}$ if an even number of coordinates
of $\left( P_{1},P_{2},...,P_{i-1}\right) $ is of even type, and $P_{i}\in
\{t_{1},t_{4}\}$ if an odd number of coordinates of $\left(
P_{1},P_{2},...,P_{i-1}\right) $ is of even type. Finally, $P_{0}=t_{2}$ if
an even number of coordinates of $\left( P_{1},P_{2},...,P_{f-1}\right) $ is
of even type, and $P_{0}=t_{1}$ otherwise.

\item[(ii)] We define $C_{2}^{\ast }$ to be the set of $f$-tuples $\left(
P_{1},P_{2},...,P_{f}\right) $ where the types of the matrices $P_{i}$ are
chosen as follows: $P_{1}\in \{t_{1},t_{4}\}.$ For $i=2,3,...,f-1,$ $%
P_{i}\in \{t_{1},t_{4}\}$ if an even number of coordinates of $\left(
P_{1},P_{2},...,P_{i-1}\right) $ is of even type, and $P_{i}\in
\{t_{2},t_{3}\}$ if an odd number of coordinates of $\left(
P_{1},P_{2},...,P_{i-1}\right) $ is of even type. Finally, $P_{0}=t_{4}$ if
an even number of coordinates of $\left( P_{1},P_{2},...,P_{f-1}\right) $ is
of even type, and $P_{0}=t_{3}$ otherwise.
\end{enumerate}
\end{definition}

\noindent In Definition \ref{stupid types def1} the type of the matrix $%
P_{0} $ has been chosen so that an odd number of coordinates of the $f$%
-tuple $\left( P_{1},P_{2},...,P_{f-1},P_{0}\right) $ is of even type.

\begin{lemma}
\label{mat}Assume that $f\geq 2$ and, as before, let $Q_{f}=P_{1}P_{2}\cdots
P_{f}.$

\begin{enumerate}
\item[(i)] If $\left( P_{1},P_{2},...,P_{f}\right) \in C_{1}^{\ast },$ then $%
Q_{f}=\left( 
\begin{array}{cc}
\alpha & \beta \\ 
\gamma & 0%
\end{array}%
\right) $ with $\beta ,\gamma $ nonconstant polynomials in $X_{1},X_{2},...,$

\noindent $X_{f}$ (with $X_{f}=X_{0}$), linearly independent over $\bar{%
\mathbb{Q}%
}_{p},$ and $\alpha $ nonscalar.

\item[(ii)] If $\left( P_{1},P_{2},...,P_{f}\right) \in C_{2}^{\ast },$ then 
$Q_{f}=\left( 
\begin{array}{cc}
0 & \beta \\ 
\gamma & \delta%
\end{array}%
\right) $ with $\beta ,\gamma $ nonconstant polynomials in $X_{1},X_{2},...,$

\noindent $X_{f},$ linearly independent over $\bar{%
\mathbb{Q}%
}_{p},$ and $\delta $ nonscalar.

\item[(iii)] If $\left( P_{1},P_{2},...,P_{f}\right) \in C_{1},$ then $%
Q_{f}=\left( 
\begin{array}{cc}
\alpha & \beta \\ 
0 & \delta%
\end{array}%
\right) $ with $\beta $ a nonzero polynomial in $X_{1},X_{2},...,X_{f},$ and 
$\alpha ,\delta $ nonzero scalars.

\item[(iv)] If $\left( P_{1},P_{2},...,P_{f}\right) \in C_{2},\ $then $%
Q_{f}=\left( 
\begin{array}{cc}
\alpha & 0 \\ 
\gamma & \delta%
\end{array}%
\right) $ with $\gamma $ a nonzero polynomial in $X_{1},X_{2},...,X_{f},$
and $\alpha ,\delta $ nonzero scalars.
\end{enumerate}
\end{lemma}

\begin{proof}
Follows easily by induction on $f.$
\end{proof}

\begin{lemma}
\label{mat1}Assume that $f\geq 2.$

\begin{enumerate}
\item[(i)] If an odd number of coordinates of $\left(
P_{1},P_{2},...,P_{f}\right) $ is of even type, then $Q_{f}$ has one of the
following forms:

\begin{enumerate}
\item[(a)] $Q_{f}=\left( 
\begin{array}{cc}
0 & \beta \\ 
\gamma & \delta%
\end{array}%
\right) $ with $\beta ,\gamma $ nonconstant polynomials in $%
X_{1},X_{2},...,X_{f},$ linearly independent over $\bar{%
\mathbb{Q}%
}_{p},$ and $\delta $ nonscalar. This case occurs if and only if $\left(
P_{1},P_{2},...,P_{f}\right) \in C_{2}^{\ast }.$

\item[(b)] $Q_{f}=\left( 
\begin{array}{cc}
\alpha & \beta \\ 
\gamma & 0%
\end{array}%
\right) $ with $\beta ,\gamma $ nonconstant polynomials in $%
X_{1},X_{2},...,X_{f},$ linearly independent over $\bar{%
\mathbb{Q}%
}_{p},$ and $\alpha $ nonscalar. This case occurs if and only if $\left(
P_{1},P_{2},...,P_{f}\right) \in C_{1}^{\ast }.$

\item[(c)] In any other case, $Q_{f}=\left( 
\begin{array}{cc}
\alpha & \beta \\ 
\gamma & \delta%
\end{array}%
\right) $ with $\beta ,\gamma $ nonconstant polynomials in $%
X_{1},X_{2},...,X_{f},$ linearly independent over $\bar{%
\mathbb{Q}%
}_{p},$ $\alpha \delta \neq 0,$ and $\mathrm{Tr}\left( Q_{f}\right) $
nonscalar.
\end{enumerate}

\item[(ii)] If an even number of coordinates of $\left(
P_{1},P_{2},...,P_{f}\right) $ is of even type, then $Q_{f}$ has one of the
following forms:

\begin{enumerate}
\item[\noindent (d)] $Q_{f}=\left( 
\begin{array}{cc}
\alpha & \beta \\ 
0 & \delta%
\end{array}%
\right) $ with $\beta $ a nonzero polynomial in $X_{1},X_{2},...,X_{f},$ and 
$\alpha ,\delta $ nonzero scalars. This case occurs if and only if $\left(
P_{1},P_{2},...,P_{f}\right) \in C_{1}.$

\item[\noindent (e)] $Q_{f}=\left( 
\begin{array}{cc}
\alpha & 0 \\ 
\gamma & \delta%
\end{array}%
\right) $ with $\gamma $ a nonzero polynomial in $X_{1},X_{2},...,X_{f},$
and $\alpha ,\delta $ nonzero scalars. This case occurs if and only if $%
\left( P_{1},P_{2},...,P_{f}\right) \in C_{2}.$

\item[(f)] In any other case, $Q_{f}=\left( 
\begin{array}{cc}
\alpha & \beta \\ 
\gamma & \delta%
\end{array}%
\right) $ with $\beta ,\gamma $ nonconstant polynomials in $%
X_{1},X_{2},...,X_{f},$ linearly independent over $\bar{%
\mathbb{Q}%
}_{p},$ $\alpha \gamma \neq 0$ and $\mathrm{Tr}\left( Q_{f}\right) $ is
nonscalar.
\end{enumerate}
\end{enumerate}
\end{lemma}

\begin{proof}
By induction on $f.$ If $f=2$ the proof of the lemma is by a direct
computation. Suppose $f\geq3.$ Case (i). An odd number of coordinates of $%
\left( P_{1},P_{2},...,P_{f}\right) $ is of even type.

\noindent(a) If an odd number of coordinates of $\left(
P_{1},P_{2},...,P_{f-1}\right) $ is of even type, then $P_{0}\in%
\{t_{1},t_{3}\}.$ We have the following three subcases:

\noindent($\alpha$) $Q_{f-1}=\left( 
\begin{array}{cc}
0 & \beta \\ 
\gamma & \delta%
\end{array}
\right) $ with $\beta,\gamma$ nonconstant polynomials in $%
X_{1},X_{2},,,.X_{f-1},$ linearly independent over $\bar{%
\mathbb{Q}%
}_{p},$ and $\delta$ nonscalar. If $P_{0}=t_{1},$ then $Q_{f}\ $is as in
case (c), and by Lemma \ref{mat} $\left( P_{1},P_{2},...,P_{f}\right) \not
\in C_{1}^{\ast}\cup C_{2}^{\ast}.$ If $P_{0}=t_{3},$ then $Q_{f}$ is as in
case (a). By the inductive hypothesis $\left( P_{1},P_{2},...,P_{f-1}\right)
\in C_{2}^{\ast},$ and since $P_{0}=t_{3},$ $\left(
P_{1},P_{2},...,P_{f}\right) \in C_{2}^{\ast}.$

\noindent($\beta$) $Q_{f-1}=\left( 
\begin{array}{cc}
\alpha & \beta \\ 
\gamma & 0%
\end{array}
\right) $ with $\beta,\gamma$ nonconstant polynomials in $%
X_{1},X_{2},,,.X_{f-1},$ linearly independent over $\bar{%
\mathbb{Q}%
}_{p},$ and $\alpha$ nonscalar. If $P_{0}=t_{1},$ then $Q_{f}\ $is as in
case (b). By the inductive hypothesis $\left( P_{1},P_{2},...,P_{f-1}\right)
\in C_{1}^{\ast},$ and since $P_{0}=t_{1},$ $\left(
P_{1},P_{2},...,P_{f}\right) \in C_{1}^{\ast}.$ If $P_{0}=t_{3},$ then $%
Q_{f}\ $is as in case (c), and by Lemma \ref{mat} $\left(
P_{1},P_{2},...,P_{f}\right) \not \in C_{1}^{\ast }\cup C_{2}^{\ast}.$

\noindent($\gamma$) $Q_{f-1}=\left( 
\begin{array}{cc}
\alpha & \beta \\ 
\gamma & \delta%
\end{array}
\right) $ with $\beta,\gamma$ nonconstant polynomials in $%
X_{1},X_{2},,,.X_{f-1},$ linearly independent over $\bar{%
\mathbb{Q}%
}_{p},$ $\alpha\delta\neq0,$ and $\mathrm{Tr}\left( Q_{f}\right) $
nonscalar. If $P_{0}\in\{t_{1},t_{3}\}$ then $Q_{f}$ is as in case (c), and
by Lemma \ref{mat} $\left( P_{1},P_{2},...,P_{f}\right) \not \in C_{1}^{\ast
}\cup C_{2}^{\ast}.$

\noindent(b) If an even number of coordinates of $\left(
P_{1},P_{2},...,P_{f-1}\right) $ is of even type, then $P_{0}\in%
\{t_{2},t_{4}\}.$ We have the following three subcases:

\noindent($\alpha$) $Q_{f-1}=\left( 
\begin{array}{cc}
\alpha & \beta \\ 
0 & \delta%
\end{array}
\right) $ with $\beta$ a nonzero polynomial in $X_{1},X_{2},,,.X_{f-1},$ and 
$\alpha,\delta$ nonzero scalars. If $P_{0}=t_{2},$ then $Q_{f}$ is as in
case (b). Since $\left( P_{1},P_{2},...,P_{f-1}\right) \in C_{1}$ and $%
P_{0}=t_{2},\ \left( P_{1},P_{2},...,P_{f}\right) \in C_{1}^{\ast}.$ If $%
P_{0}=t_{4},\ $then $Q_{f}$ is as in case (c), and by Lemma \ref{mat} $%
\left( P_{1},P_{2},...,P_{f}\right) \not \in C_{1}^{\ast}\cup C_{2}^{\ast }.$

\noindent ($\beta $) $Q_{f-1}=\left( 
\begin{array}{cc}
\alpha & 0 \\ 
\gamma & \delta%
\end{array}%
\right) $ with $\gamma $ a nonzero polynomial in $X_{1},X_{2},,,.X_{f-1},$
and $\alpha ,\delta $ nonzero scalars. If $P_{0}=t_{2},$ then $Q_{f}\ $is as
in case (c), and by Lemma \ref{mat} $\left( P_{1},P_{2},...,P_{f}\right)
\not\in C_{1}^{\ast }\cup C_{2}^{\ast }.$ If $P_{0}=t_{4},$ then $Q_{f}\ $is
as in case (a). Since $\left( P_{1},P_{2},...,P_{f-1}\right) \in C_{2}$ and $%
P_{0}=t_{4},\ \left( P_{1},P_{2},...,P_{f}\right) \in C_{2}^{\ast }.$

\noindent ($\gamma $) $Q_{f-1}=\left( 
\begin{array}{cc}
\alpha & \beta \\ 
\gamma & \delta%
\end{array}%
\right) $ with $\beta ,\gamma $ nonconstant polynomials in $%
X_{1},X_{2},,,.X_{f},$ linearly independent over $\bar{%
\mathbb{Q}%
}_{p},$ $\alpha \gamma \neq 0$ and $\mathrm{Tr}\left( Q_{f}\right) $ is
nonscalar. Then $\left( P_{1},P_{2},...,P_{f-1}\right) \not\in C_{1}\cup
C_{2}.$ If $P_{0}\in \{t_{2},t_{4}\},$ then $Q_{f}\ $is as in case (c), and
by Lemma \ref{mat} $\left( P_{1},P_{2},...,P_{f}\right) \not\in C_{1}^{\ast
}\cup C_{2}^{\ast }.$

\noindent Case (ii). An even number of coordinates of $\left(
P_{1},P_{2},...,P_{f}\right) $ is of even type. The rest of the lemma is
proved by a case-by-case analysis similar to that of Case (i).
\end{proof}

\begin{corollary}
\label{c cor}$\mathrm{Tr}\left( Q_{f}\right) \in \bar{%
\mathbb{Q}%
}_{p}$ if and only if $\left( P_{1},P_{2},...,P_{f-1},P_{0}\right) \in
C_{1}\cup C_{2}.$
\end{corollary}

\begin{remark}
If $\left( P_{1},P_{2},...,P_{f}\right) \in C_{1}\cup C_{2},$ the filtered $%
\varphi $-modules $\mathbb{D}_{\vec{k}\ }^{\vec{i}}(\vec{0})$ are weakly
admissible and the corresponding crystalline representation is
split-reducible and ordinary (see \S \ref{red}). The filtered $\varphi $%
-modules $\mathbb{D}_{\vec{k}\ }^{\vec{i}}\left( \vec{\alpha}\right) $ make
sense\ for any $\vec{\alpha}\in \mathcal{O}_{E}^{f}.$ One can check by
induction that $\mathrm{Tr}\left( \varphi ^{f}\right)
=1+p^{\tsum\limits_{i=0}^{f-1}k_{i}},$ therefore whenever $\mathbb{D}_{\vec{k%
}\ }^{\vec{i}}\left( \vec{\alpha}\right) $ is weakly admissible, the
corresponding crystalline representation is reducible (see Proposition \ref%
{existance of inf families}). Since we have not constructed the Wach modules
which give rise to these filtered modules, weak admissibility is not
automatic and has to be checked.
\end{remark}

\noindent We now turn our attention to condition (iv) of Lemma \ref{outline
section theorem}. By Proposition \ref{prop operator in 4 types case} the
problematic cases are those with $\ell =k,$ with all the weights $k_{i}\ $%
equal, and $Q_{f}\in \{E_{11},E_{22}\}.$ We have the following.

\begin{lemma}
\label{cond 4}If $\bar{Q}_{f}=E_{11}$ then $\left( P_{1},...,P_{f}\right)
\in C_{1};$ (ii) If $\bar{Q}_{f}=E_{22},$ then $\left(
P_{1},...,P_{f}\right) \in C_{2}.$
\end{lemma}

\begin{proof}
By induction on $f.$ First, we notice that 
\begin{equation*}
P\func{mod}(p\cdot Id,\ X_{i}\cdot Id)=\left\{ 
\begin{array}{c}
E_{22}\ \mathrm{if}\text{ }P=t_{1}, \\ 
E_{12}\ \mathrm{if}\text{\ }P=t_{2}, \\ 
E_{11}\ \mathrm{if}\text{ }P=t_{3}, \\ 
E_{21}\ \mathrm{if}\text{ }P=t_{4}.%
\end{array}%
\right.
\end{equation*}%
Suppose that $\bar{Q}_{f}=E_{11}\ $and $f=2.$ Then $P_{1}\in
\{t_{2},t_{3}\}. $ If $P_{1}=t_{2},$ then $P_{0}=t_{4}$ and if $P_{1}=t_{3},$
then $P_{0}=t_{3}.$ Suppose $\bar{Q}_{f}=E_{11}\ $and $f>2.$ Then $\bar{Q}%
_{f-1}=E_{11}\ $and$\ P_{f}=t_{3}$ or $\bar{Q}_{f-1}=E_{12}\ $and$\
P_{f}=t_{4}.$ In the first case the inductive hypothesis implies that $%
\left( P_{1},P_{2},...,P_{f-1}\right) \in C_{1}$ and $P_{f}=t_{3}.$ If an
even number of coordinates of $\left( P_{1},P_{2},...,P_{f-2}\right) $ is of
even type, then $P_{f-1}=t_{3}.$ In this case an even number of coordinates
of $\left( P_{1},P_{2},...,P_{f-1}\right) $ is of even type and $%
P_{f}=t_{3}, $ hence $\left( P_{1},...,P_{f}\right) \in C_{1}.$ If an odd
number of coordinates of $\left( P_{1},P_{2},...,P_{f-2}\right) $ is of even
type, then $P_{f-1}=t_{4}.$ In this case an even number of coordinates of $%
\left( P_{1},P_{2},...,P_{f-1}\right) $ is of even type and $P_{f}=t_{3},$
hence $\left( P_{1},...,P_{f}\right) \in C_{1}.$ Now assume that $\bar{Q}%
_{f-1}=E_{12}\ $and$\ P_{f}=t_{4}.$ This implies that either $\bar{Q}%
_{f-2}=E_{12},$ $P_{f-1}=t_{4}$ and$\ P_{f}=t_{4}$ which is absurd since in
this case $\bar{Q}_{f}=0,$ or $\bar{Q}_{f-2}=E_{11},$ $P_{f-1}=t_{2}$ and$\
P_{f}=t_{4}.$ If $f=3,$ then $P_{1}=t_{3},$ $P_{2}=t_{2},$ $P_{3}=t_{4}$ and
the lemma holds. If $f\geq 4$ and an even number of coordinates $\left(
P_{1},P_{2},...,P_{f-3}\right) $ is of even type, then $P_{f-2}=t_{3},$ $%
P_{f-1}=t_{2}$ and$\ P_{f}=t_{4}.$ Then an odd number of coordinates $\left(
P_{1},P_{2},...,P_{f-1}\right) $ is of even type and$\ P_{f}=t_{4},$ hence $%
\left( P_{1},...,P_{f}\right) \in C_{1}.\ $If an odd number of coordinates $%
\left( P_{1},P_{2},...,P_{f-3}\right) $ is of even type, then $%
P_{f-2}=t_{4}, $ $P_{f-1}=t_{2}$ and$\ P_{f}=t_{4}.$ Then an odd number of
coordinates $\left( P_{1},P_{2},...,P_{f-1}\right) $ is of even type and$\
P_{f}=t_{4},$ hence $\left( P_{1},...,P_{f}\right) \in C_{1}.$ Part (ii) is
proved similarly.
\end{proof}

\begin{corollary}
\label{trace gives surject}If $\left( P_{1},P_{2},...,P_{f}\right) \in 
\mathcal{P}\ $and $\mathrm{Tr}\left( Q_{f}\right) \not\in \bar{%
\mathbb{Q}%
}_{p},$ then the operator%
\begin{equation*}
\overline{H}\mapsto \overline{H-Q_{f}H(p^{f\ell }Q_{f}^{-1})}:\widetilde{M}%
_{2}\rightarrow \widetilde{M}_{2}
\end{equation*}%
is surjective.
\end{corollary}

\subsection{Corresponding$\ $families of rank two filtered $\protect\varphi $%
-modules \label{section families of crystalline reps}}

Let $\Pi ^{\vec{i}}(\mathcal{S})=(\Pi _{1}(X_{1}),\Pi _{2}(X_{2}),$

\noindent $...,\Pi _{f-1}(X_{f-1}),$\noindent $\Pi _{0}(X_{0}))$ with $\vec{i%
}\in \{1,2,3,4\}^{f}\ $and $\Pi _{i}(X_{i})$ as in Definition \ref{the 8
types copy(1)}. The definition of the matrices $\Pi _{i}$ and $P_{i}=\Pi _{i}%
\func{mod}\pi $ depends on the choice of the $z_{i}$ in Proposition \ref{zi
lemma} and therefore on $\ell .\ $For the rest of the paper we assume that $%
\left( P_{1},P_{2},...,P_{0}\right) \not\in C_{1}\cup C_{2}$ and we choose $%
\ell =k=\max \{k_{0},k_{1},...,k_{f-1}\}.$

\begin{proposition}
\label{kati}For any $\gamma \in \Gamma _{K},$ there exists a unique matrix $%
G_{\gamma }(\mathcal{S})=G_{\gamma }(\mathcal{S})\in M_{2}^{\mathcal{S}}$
such that

\begin{enumerate}
\item[(i)] $G_{\gamma }(\mathcal{S})\equiv \overrightarrow{Id}\func{mod}\vec{%
\pi}$ and

\item[(ii)] $\Pi ^{\vec{i}}(\mathcal{S})\mathcal{\varphi }(G_{\gamma }(%
\mathcal{S}))=G_{\gamma }(\mathcal{S})\gamma (\Pi ^{\vec{i}}(\mathcal{S})).$
\end{enumerate}
\end{proposition}

\begin{proof}
Conditions $(1)$ and $(2)$ of Lemma \ref{outline section theorem} are
satisfied by Proposition \ref{zi lemma}. Condition $(3)$ of Lemma \ref%
{outline section theorem} is satisfied by the assumption that $\left(
P_{1},P_{2},...,P_{0}\right) \not\in C_{1}\cup C_{2}$ and Corollaries \ref%
{eigenvalies cor} and \ref{c cor}. Finally, condition $(4)$ of Lemma \ref%
{outline section theorem} is satisfied by Proposition \ref{prop operator in
4 types case} and Lemma \ref{cond 4}.\ The proposition follows from Lemma %
\ref{outline section theorem}.
\end{proof}

\noindent For any $\vec{a}\in \mathfrak{m}_{E}^{f}$ we equip the module $%
\mathbb{N}_{\vec{k}}^{\vec{i}}(\vec{a})=\left( \mathcal{O}_{E}[[\pi ]]^{\mid
\tau \mid }\right) \eta _{1}\oplus \left( \mathcal{O}_{E}[[\pi ]]^{\mid \tau
\mid }\right) \eta _{2}$ with semilinear $\varphi $ and $\Gamma _{K}$%
-actions defined as in Proposition \ref{rank two wach modules construction}%
.\ \noindent For any $\vec{a}\in \mathfrak{m}_{E}^{f}$ we consider the
matrices of $\mathrm{GL}_{2}\left( E^{\mid \tau \mid }\right) $ obtained
from the matrices $P^{\vec{i}}(\vec{a})=\left(
P_{1}(X_{1}),P_{2}(X_{2}),...,P_{f-1}(X_{f-1}),P_{0}(X_{0})\right) $ by
substituting $X_{j}=a_{j}$ in $P_{j}(X_{j}).\ $We define families of
filtered $\varphi $-modules $\mathbb{D}_{\vec{k}}^{\vec{i}}\left( \vec{a}%
\right) =\left( E^{\mid \tau \mid }\right) \eta _{1}\oplus \left( E^{\mid
\tau \mid }\right) \eta _{2}$ with Frobenius endomorphisms given by $%
(\varphi (\eta _{1}),\varphi (\eta _{2}))=(\eta _{1},\eta _{2})P^{\vec{i}}(%
\vec{a}),$ and filtrations given by

\begin{equation}
\ \ \ \ \ \ \ \ \ \ \ \ \ \ \ \ \ \ ~\ \ \ \mathrm{Fil}^{\mathrm{j}}(\mathbb{%
D}_{\vec{k}}^{\vec{i}}\left( \vec{a}\right) )=\left\{ 
\begin{array}{l}
E^{\mid \tau \mid }\eta _{1}\oplus E^{\mid \tau \mid }\eta _{2}\ \ \ \ \ 
\text{if }j\leq 0, \\ 
E^{\mid \tau _{I_{0}}\mid }\left( \vec{x}\eta _{1}+\vec{y}\eta _{2}\right) \
\ \text{if }1\leq j\leq w_{0}, \\ 
E^{\mid \tau _{I_{1}}\mid }\left( \vec{x}\eta _{1}+\vec{y}\eta _{2}\right) \
\ \text{if }1+w_{0}\leq j\leq w_{1}, \\ 
\ \ \ \ \ \ \ \ \ \ \ \ \ \ \ \ \ \ \ \ \cdots \cdots \\ 
E^{\mid \tau _{I_{t-1}}\mid }\left( \vec{x}\eta _{1}+\vec{y}\eta _{2}\right)
\ \text{if }1+w_{t-2}\leq j\leq w_{t-1}, \\ 
\ \ \ \ \ \ \ \ \ \ \ \ \ \ \ \ \ 0\ \ \ \ \ \ \ \ \ \text{if }j\geq
1+w_{t-1},%
\end{array}%
\right.  \label{Filitations}
\end{equation}%
where $\vec{x}=(x_{0},x_{1},...,x_{f-1})$ and $\vec{y}%
=(y_{0},y_{1},...,y_{f-1}),$ with 
\begin{equation}
\ \ \ \ \ \ \ \ \ \ \ \ (x_{i},y_{i})=\left\{ 
\begin{array}{l}
(1,-\alpha _{i})\ \ \ \text{\ \ if }P_{i}\ \text{has type }1\ \text{or}\ 2,
\\ 
(-\alpha _{i},1)\ \ \ \ \ \text{if }P_{i}\ \text{has type }3\ \text{or }4,%
\end{array}%
\right.  \label{Filitations1}
\end{equation}%
and $\alpha _{i}=a_{i}z_{i}(0)\ $for all $i.$ Since $\ell =k,$ Remark \ref%
{k=p remark} implies that $\alpha _{i}\in p^{m}\mathfrak{m}_{E}\ $for all $%
i, $ where 
\begin{equation*}
m:=\left\{ 
\begin{array}{l}
\ \lfloor \frac{k-1}{p-1}\rfloor \ \ \ \ \text{if }k\geq p\ \text{and }%
k_{i}\neq p\ \text{for some }i, \\ 
\ \ \ \ 0\ \ \ \ \ \ \ \text{if }k\leq p-1\ \text{or }k_{i}=p\ \text{for all 
}i.%
\end{array}%
\right.
\end{equation*}

\begin{proposition}
\label{main theorem}For any $\vec{a}\in \mathfrak{m}_{E}^{f},$ the filtered $%
\varphi $-modules $(\mathbb{D}_{\vec{k}}^{\vec{i}}\left( \vec{a}\right)
,\varphi )$ defined above are weakly admissible and 
\begin{equation*}
\mathbb{D}_{\vec{k}}^{\vec{i}}\left( \vec{a}\right) \simeq E^{\mid \tau \mid
}\bigotimes\limits_{\mathcal{O}_{E}^{\mid \tau \mid }}(\mathbb{N}_{\vec{k}}^{%
\vec{i}}(\vec{a})/\pi \mathbb{N}_{\vec{k}}^{\vec{i}}(\vec{a}))
\end{equation*}%
as filtered $\varphi $-modules over $E^{\mid \tau \mid }.$
\end{proposition}

\begin{proof}
By Theorem \ref{berger thm}, $\vec{x}\eta _{1}+\vec{y}\eta _{2}\in \mathrm{%
Fil}^{\mathrm{j}}(\mathbb{N}_{\vec{k}}^{\vec{i}}(\vec{a}))$ if and only if $%
\varphi (\vec{x})\varphi (\eta _{1})+$ $\varphi (\vec{y})\varphi (\eta _{2})$
$\in q^{j}\mathbb{N}_{\vec{k}}^{\vec{i}}(\vec{a})$ or equivalently 
\begin{equation}
e_{i}\varphi (\vec{x})\varphi (\eta _{1})+e_{i}\varphi (\vec{y})\varphi
(\eta _{2})\in q^{j}e_{i}\mathbb{N}_{\vec{k}}^{\vec{i}}(\vec{a})\ \text{for\
all}\ i\in I_{0},  \label{starcircle}
\end{equation}%
with the idempotents $e_{i}$ as in \S \ref{product ring}. We fix some $i\in
I_{0}$ and calculate in the case where $\Pi _{i}$ is of type 2. \noindent
\noindent Then \noindent $\Pi _{i}(a_{i})=\left( 
\begin{array}{cc}
0 & c_{i}q^{k_{i}} \\ 
1 & a_{i}\mathcal{\varphi }(z_{i})%
\end{array}%
\right) $ and equation \ref{starcircle} is equivalent to $\left\{ 
\begin{array}{l}
q^{j}\mid \varphi (y_{i})q^{k_{i}}, \\ 
q^{j}\mid \varphi (x_{i}+y_{i}a_{i}z_{i}).%
\end{array}%
\right. $\noindent \noindent \noindent\ We use that $q^{j}\mid $ $\varphi
(x) $ if and only if $\pi ^{j}\mid x$ for any $x\in \mathcal{O}_{E}[[\pi ]].$
If $j\geq 1+k_{i},$ then $x_{i},y_{i}\equiv 0\func{mod}\pi .$ If $1\leq
j\leq k_{i},$ the system above is equivalent to $\pi ^{j}\mid
x_{i}+y_{i}a_{i}z_{i}.$ Since $a_{i}z_{i}\equiv \alpha _{i}\func{mod}\pi ,\ $%
\begin{equation*}
\ \ \ \ \ \ \ \ \ \ e_{i}\vec{x}\eta _{1}+e_{i}\vec{y}\eta _{2}+\pi \mathbb{N%
}_{\vec{k}}^{\vec{i}}(\vec{a})=\left\{ 
\begin{array}{c}
\alpha _{i}\bar{y}_{i}e_{i}\eta _{1}+\bar{y}_{i}e_{i}\eta _{2}+\pi \mathbb{N}%
_{\vec{k}}^{\vec{i}}(\vec{a})\ \ \mathrm{if}\text{\ \ }1\leq \text{ }j\leq
k_{i},\ \  \\ 
\ \ \ \ 0\ \ \ \ \ \ \ \ \ \ \ \ \ \ \ \ \ \ \ \ \ \ \ \ \mathrm{if}\text{\
\ }j\geq k_{i}%
\end{array}%
\right.
\end{equation*}%
where $\bar{y}_{i}=y_{i}\func{mod}\pi $ can be any element of $\mathcal{O}%
_{E}.$ Since $\mathrm{Fil}^{\mathrm{0}}(\mathbb{N}_{\vec{k}}^{\vec{i}}(\vec{a%
})/\pi \mathbb{N}_{\vec{k}}^{\vec{i}}(\vec{a}))=(\mathcal{O}_{E}^{\mid \tau
\mid })\eta _{1}\bigoplus (\mathcal{O}_{E}^{\mid \tau \mid })\eta _{2},$ we
get 
\begin{equation*}
e_{i}\mathrm{Fil}^{\mathrm{j}}(\mathbb{N}_{\vec{k}}^{\vec{i}}(\vec{a})/\pi 
\mathbb{N}_{\vec{k}}^{\vec{i}}(\vec{a}))=\left\{ 
\begin{array}{l}
e_{i}(\mathcal{O}_{E}^{\mid \tau \mid })\eta _{1}\bigoplus e_{i}(\mathcal{O}%
_{E}^{\mid \tau \mid })\eta _{2}\ \text{if }j\leq 0, \\ 
e_{i}(\mathcal{O}_{E}^{\mid \tau \mid })(\vec{x}^{i}\eta _{1}+\vec{y}%
^{i}\eta _{2})\ \ \ \ \text{if }1\leq j\leq k_{i}, \\ 
\ \ \ \ \ \ 0\ \ \ \ \ \ \ \ \ \ \ \ \ \ \ \ \ \ \ \ \ \ \ \ \text{if }j\geq
1+k_{i},%
\end{array}%
\right.
\end{equation*}%
with $e_{i}\vec{x}^{i}=(0,...,x_{i},...,0),$ $e_{i}\vec{y}%
^{i}=(0,...,y_{i},...,0)$ and $(x_{i},y_{i})=(-\alpha _{i},1).$ Calculating
for the other choices of $\Pi _{i}(a_{i}),$ we see that for all $i\in I_{0},$
$(x_{i},y_{i})$ is as in formula \ref{Filitations}. Since $\mathrm{Fil}^{%
\mathrm{j}}(\mathbb{N}_{\vec{k}}^{\vec{i}}(\vec{a})/\pi \mathbb{N}_{\vec{k}%
}^{\vec{i}}(\vec{a}))=\tbigoplus\limits_{i=0}^{f-1}e_{i}\mathrm{Fil}^{%
\mathrm{j}}(\mathbb{N}_{\vec{k}}^{\vec{i}}(\vec{a})/\pi \mathbb{N}_{\vec{k}%
}^{\vec{i}}(\vec{a})),\ $arguing as in the proof of Proposition \ref{The
crystalline characters} we get 
\begin{equation*}
\mathrm{Fil}^{\mathrm{j}}(\mathbb{N}_{\vec{k}}^{\vec{i}}(\vec{a})/\pi 
\mathbb{N}_{\vec{k}}^{\vec{i}}(\vec{a}))=\left\{ 
\begin{array}{l}
(\mathcal{O}_{E}^{\mid \tau \mid })\eta _{1}\oplus (\mathcal{O}_{E}^{\mid
\tau \mid })\eta _{2}\ \ \ \ \ \ \mathrm{if}\text{ }j\leq 0, \\ 
(\mathcal{O}_{E}^{\mid \tau \mid })f_{I_{0}}(\vec{x}\eta _{1}+\vec{y}\eta
_{2})\ \ \ \ \ \mathrm{if}\text{ }1\leq j\leq w_{0}, \\ 
(\mathcal{O}_{E}^{\mid \tau \mid })f_{I_{1}}(\vec{x}\eta _{1}+\vec{y}\eta
_{2})\ \ \ \ \ \mathrm{if}\text{ }1+w_{0}\leq j\leq w_{1}, \\ 
\ \ \ \ \ \ \ \ \ \ \ \ \ \ \ \ \ \ \ \ \ \ \ \ \ \ \cdots \cdots \\ 
(\mathcal{O}_{E}^{\mid \tau \mid })f_{I_{t-1}}(\vec{x}\eta _{1}+\vec{y}\eta
_{2})\ \ \ \mathrm{if}\text{ }1+w_{t-2}\leq j\leq w_{t-1}, \\ 
\ \ \ \ 0\ \ \ \ \ \ \ \ \ \ \ \ \ \ \ \ \ \ \ \ \ \ \ \ \ \ \ \mathrm{if}%
\text{ }j\geq 1+w_{t-1}%
\end{array}%
\right.
\end{equation*}%
with $\vec{x}=(x_{0},x_{1},...,x_{f-1})$ and $\vec{y}%
=(y_{0},y_{1},...,y_{f-1})$ and $(x_{i},y_{i})$ as in formula \ref%
{Filitations}. The isomorphism 
\begin{equation*}
\mathbb{D}_{\vec{k}}^{\vec{i}}\left( \vec{a}\right) \simeq E^{\mid \tau \mid
}\bigotimes\limits_{\mathcal{O}_{E}^{\mid \tau \mid }}(\mathbb{N}_{\vec{k}}^{%
\vec{i}}(\vec{a})/\pi \mathbb{N}_{\vec{k}}^{\vec{i}}(\vec{a}))
\end{equation*}%
is now obvious.
\end{proof}

\noindent The crystalline representation corresponding to $\mathbb{D}_{\vec{k%
}}^{\vec{i}}\left( \vec{a}\right) $ is denoted by $V_{\vec{k},\vec{a}}^{\vec{%
i}}$\noindent $.$

\section{Reductions of crystalline representations\label{Reductions section}}

\noindent In this section we explicitly compute the semisimplified modulo $p$
reductions of the families of crystalline representations constructed in \S %
\ref{2-d families}.\ We will need the following lemma.

\begin{lemma}
\label{frob lemma}Let $F$ be any field, $G$ any group and $H$ any finite
index subgroup. Let $V$ be an irreducible finite-dimensional $FG$-module
whose restriction to $H$ contains some $FH$-submodule $W$ with $\dim
_{F}V=[G:H]\dim _{F}W.$ Then $V\simeq \mathrm{Ind}_{H}^{G}\left( W\right) .$
\end{lemma}

\begin{proof}
By Frobenius reciprocity there exists some nonzero $\alpha \in \mathrm{Hom}%
_{FG}(\mathrm{Ind}_{H}^{G}\left( W\right) ,V).$ It is an isomorphism because 
$V$ is irreducible and $\mathrm{Ind}_{H}^{G}\left( W\right) \ $and $V$ have
the same dimension over $F.$
\end{proof}

\noindent We start with the reductions of crystalline characters and
reducible two-dimensional crystalline representations of $G_{K}.$ We compose
the embeddings of $K_{f}$ into $E$ with the embedding $E\overset{\varepsilon 
}{\hookrightarrow }\bar{%
\mathbb{Q}%
}_{p}$ that we fixed in the introduction and get all the embeddings$\ $of $%
K_{f}\ $in$\ \bar{%
\mathbb{Q}%
}_{p},\ $which we still denote by $\tau _{i}.$ They induce embeddings of
residue fields $k_{K_{f}}\overset{\bar{\tau}_{i}}{\rightarrow }\mathbb{\bar{F%
}}_{p},$ where $k_{K_{f}}$ is the residue field of $K_{f}.$ The level $f$
fundamental characters $\omega _{f,\bar{\tau}_{i}}$ of $I_{K_{f}}$ are
defined by composing the embeddings $\bar{\tau}_{i}$ with the homomorphism $%
I_{K_{f}}\rightarrow k_{K_{f}}^{\times }$ obtained from local class field
theory, with uniformizers corresponding to geometric Frobenius elements. We
recall the following lemma which follows immediately from \cite[Lemma 3.8]%
{BDJ}, where the $\chi _{i}$ are as in \S \ref{the crystalline characters}.

\begin{lemma}
\label{BDJ lemma}(i) $(\bar{\chi}_{i})_{\mid I_{K_{f}}}=\omega _{f,\bar{\tau}%
_{i+1}}^{-1}$ for $i=0,1,...,f-1;$ (ii) $\omega _{f,\bar{\tau}_{i}}=\omega
_{f,\bar{\tau}_{0}}^{p^{i}}$ for all $i;$ (iii) $\omega _{2f,\bar{\tau}%
_{0}}^{1+p^{f}}=\omega _{f,\bar{\tau}_{0}};$ (iv) $\omega
=\tprod\limits_{i\in I_{0}}\omega _{f,\bar{\tau}_{i}},$ where $\omega $ is
the cyclotomic character modulo $\mathfrak{m}_{E}.$\noindent
\end{lemma}

\noindent Our next goal is to compute the determinant of a two-dimensional
crystalline representations in terms of its labeled Hodge-Tate weights. To
do this, we will need some facts about weakly admissible filtered $\varphi $%
-modules which we briefly recall. For the missing details we refer to \cite%
{DO10}. We remark that similar results for odd $p$ have been obtained by
Imai in \cite{Ima09}.

\begin{proposition}
\label{class thm}Let $(\mathbb{D},\varphi )$ be a rank two F-semisimple,
nonscalar filtered $\varphi $-module over $E^{\mid \tau \mid }$ with labeled
Hodge-Tate weights $\{0,-k_{i}\}_{\tau _{i}}.$ After enlarging $E$ if
necessary, there exists an ordered basis $\underline{\eta }$ of $\mathbb{D}$
over $E^{\mid \tau \mid }$ with respect to which the matrix of Frobenius
takes the form $\mathrm{Mat}_{\underline{\eta }}(\varphi )=\mathrm{diag}(%
\vec{\alpha},\vec{\delta})$ for some vectors $\vec{\alpha},\vec{\delta}\in $ 
$(E^{\times })^{\mid \tau \mid }\ $with$\ \mathrm{Nm}_{\varphi }(\vec{\alpha}%
)\neq \mathrm{Nm}_{\varphi }(\vec{\delta}).$ \noindent The filtration in the
same basis has the form of formula $\ref{Filitations}$ for some vectors $%
\vec{x},\vec{y}\in E^{\mid \tau \mid }$ with $(x_{i},y_{i})\neq (0,0)$ for
all $i\in I_{0}.$ \noindent We call such a basis $\underline{\eta }$ a
standard basis of $(\mathbb{D},\varphi ).$\noindent\ The Frobenius-fixed
submodules are $0,\ \mathbb{D},$ $\mathbb{D}_{1}:=\left( E^{\mid \tau \mid
}\right) \eta _{1}$ and $\mathbb{D}_{2}:=\left( E^{\mid \tau \mid }\right)
\eta _{2}.$ \noindent \noindent The module $\mathbb{D}$ is weakly admissible
if and only if%
\begin{align*}
& (1)\text{\ }\mathrm{v}_{\mathrm{p}}(\mathrm{Nm}_{\varphi }(\vec{\alpha})%
\mathrm{Nm}_{\varphi }(\vec{\delta}))=\sum\limits_{i\in I_{0}}k_{i};\  \\
& (2)\ \mathrm{v}_{\mathrm{p}}(\mathrm{Nm}_{\varphi }(\vec{\alpha}))\geq
\sum\limits_{\{i\in I_{0}:\ y_{i}=0\}}k_{i}; \\
& (3)\ \mathrm{v}_{\mathrm{p}}(\mathrm{Nm}_{\varphi }(\vec{\delta}))\geq
\sum\limits_{\{i\in I_{0}:\ x_{i}=0\}}k_{i}.
\end{align*}%
Assuming that $\mathbb{D}$ is weakly admissible, \noindent \noindent

\begin{enumerate}
\item[(i)] The filtered $\varphi $-module $\mathbb{D}$ is irreducible if and
only if both the inequalities $(2)$ and $(3)$ above are strict; \noindent
\noindent \noindent \noindent

\item[(ii)] The filtered $\varphi $-module $\mathbb{D}$ is split-reducible
if and only if both inequalities $(2)$\ and $(3)$ are equalities, or
equivalently $I_{0}^{+}\cap J_{\vec{x}}\cap J_{\vec{y}}=\varnothing .$ In
this case, the only nontrivial weakly admissible submodules are the $\mathbb{%
D}_{i},$ $i=1,2.$ Moreover, $\mathbb{D}=\mathbb{D}_{1}\oplus \mathbb{D}_{2};$

\item[(iii)] In any other case the filtered $\varphi $-module $\mathbb{D}$
is reducible, non-split. \ 
\end{enumerate}
\end{proposition}

\noindent \noindent \noindent In Proposition \ref{class thm}, if $\mathrm{v}%
_{\mathrm{p}}(\mathrm{Nm}_{\varphi }(\vec{\alpha}))=\sum\limits_{\{i\in
I_{0}:\ y_{i}=0\}}k_{i},$ the only nontrivial weakly admissible submodule is 
$(\mathbb{D}_{1},\varphi ).$ If $\mathrm{v}_{\mathrm{p}}(\mathrm{Nm}%
_{\varphi }(\vec{\delta}))=\sum\limits_{\{i\in I_{0}:\ x_{i}=0\}}k_{i},$
then the only nontrivial weakly admissible submodule is $(\mathbb{D}%
_{2},\varphi ).\ $\noindent If $(\mathbb{D},\varphi )$ is not F-semisimple,
after extending $E$ if necessary, there exists an ordered basis $\underline{%
\eta }=(\eta _{1},\eta _{2})$ of $\mathbb{D}$ over $E^{\mid \tau \mid }$
with respect to which the matrix of Frobenius takes the form%
\begin{equation*}
\mathrm{Mat}_{\underline{\eta }}(\varphi )=\left( 
\begin{array}{cc}
\vec{\alpha} & \vec{0} \\ 
\vec{\gamma} & \vec{\alpha}%
\end{array}%
\right)
\end{equation*}%
for some vectors $\vec{\alpha}\in $ $(E^{\times })^{\mid \tau \mid }\ $and $%
\vec{\gamma}\in E\ $(see \cite[\S 2.1]{DO10}).$\ $The filtration in this
basis has the shape of formula \ref{Filitations}. The filtered $\varphi $%
-module$\ (\mathbb{D},\varphi )$ is weakly admissible if any only if $2\cdot 
\mathrm{v}_{\mathrm{p}}(\mathrm{Nm}_{\varphi }(\vec{\alpha}%
))=\sum\limits_{i\in I_{0}}k_{i}$ and $\mathrm{v}_{\mathrm{p}}(\mathrm{Nm}%
_{\varphi }(\vec{\alpha}))\geq \sum\limits_{\{i\in I_{0}:x_{i}=0\}}k_{i}.$
The corresponding crystalline representation is irreducible if and only if
the inequality is strict and reducible, non-split otherwise. In this case,
the only $\varphi $-stable weakly admissible submodule is $(\mathbb{D}%
_{2},\varphi )\ $(see also \cite[\S\ 5.4]{DO10}). If $(\mathbb{D},\varphi )$
is F-scalar,\ there exists an ordered basis $\underline{\eta }=(\eta
_{1},\eta _{2})$ of $\mathbb{D}$ over $E^{\mid \tau \mid }$ with respect to
which $\mathrm{Mat}_{\underline{\eta }}(\varphi )=\mathrm{diag}\left( \alpha
\cdot \vec{1},\alpha \cdot \vec{1}\right) $ for some $\alpha \in E^{\times
}\ $and the filtration is as in formula \ref{Filitations}. The only $\varphi 
$-stable submodules are the $\mathbb{D}_{i},$ $i=1,2$ and $\mathbb{D}%
(c)=\left( E^{\mid \tau \mid }\right) (\eta _{1}+c\cdot \vec{1}\cdot \eta
_{2})$ for any $c\in E^{\times }$ (cf. \cite[\S 5.3]{DO10}). To summarize,
we have the following.

\begin{proposition}
\label{reducible section prop}Let $(\mathbb{D},\varphi )$ be a reducible
weakly admissible rank two filtered $\varphi $-module over $E^{\mid \tau
\mid },$ with labeled Hodge-Tate weights $\{0,-k_{i}\}_{\tau _{i}}.$ After
enlarging $E$ if necessary, there exists an ordered basis $\underline{\eta }%
=(\eta _{1},\eta _{2})$ of $\mathbb{D}$ over $E^{\mid \tau \mid }$ with
respect to which the matrix of Frobenius takes the form $\mathrm{Mat}_{%
\underline{\eta }}(\varphi )=\left( 
\begin{array}{cc}
\vec{\alpha} & \vec{0} \\ 
\ast & \vec{\delta}%
\end{array}%
\right) $ and is such that $\mathbb{D}_{2}=\left( E^{\mid \tau \mid }\right)
\eta _{2}$ is a $\varphi $-stable, weakly admissible submodule.
\end{proposition}

\noindent \noindent The following propositions which will be used in \S \S\ %
\ref{irreducibles} and \ref{red}.

\begin{proposition}
\label{existance of inf families}Let $(\mathbb{D},\varphi )$ be a rank two
weakly admissible filtered $\mathcal{\varphi }$-modules. If $\mathrm{v}_{%
\mathrm{p}}\left( \mathrm{Tr}(\varphi ^{f})\right) =0,$ then the
corresponding crystalline representation is reducible.
\end{proposition}

\begin{proof}
If $\mathbb{D}$ is F-semisimple and nonscalar, see \cite[Corollary 7.2]{DO10}%
. Suppose that this is not the case. Since we assume that $k_{i}>0$ for at
least one $i,$ for any F-scalar or non-F-semisimple filtered $\varphi $%
-module with labeled weights $\{-k_{i},0\}_{\tau _{i}},$ $\mathrm{v}_{%
\mathrm{p}}\left( \mathrm{Tr}(\varphi ^{f})\right) \not=0.$ Indeed, in this
case there exists an ordered basis $\underline{\eta }$ of $\mathbb{D}$ over $%
E^{\mid \tau \mid }$ with respect to which the matrix of Frobenius takes the
form%
\begin{equation*}
\mathrm{Mat}_{\underline{\eta }}(\varphi )=\left( 
\begin{array}{cc}
\vec{\alpha} & \vec{0} \\ 
\vec{\gamma} & \vec{\alpha}%
\end{array}%
\right)
\end{equation*}%
for some vectors $\vec{\alpha}\in $ $(E^{\times })^{\mid \tau \mid }\ $and $%
\vec{\gamma}\in E\ $(see \cite[\S 2.1]{DO10}). Weak admissibility implies
that $2\cdot \mathrm{v}_{\mathrm{p}}(\mathrm{Nm}_{\varphi }(\vec{\alpha}%
))=\sum\limits_{i\in I_{0}}k_{i}>0\ $(see \cite[Propositions 4.3 and 4.4]%
{DO10}), therefore $\mathrm{v}_{\mathrm{p}}\left( \mathrm{Tr}(\varphi
^{f})\right) =\mathrm{v}_{\mathrm{p}}\left( 2\cdot \mathrm{Nm}_{\varphi }(%
\vec{\alpha})\right) >0.$
\end{proof}

\noindent \noindent The following lemma allows us to compute determinants of
two-dimensional crystalline representations in terms of their labeled
Hodge-Tate weights.

\begin{lemma}
\label{callmedet}If $(\mathbb{D},\varphi )$ is a rank two weakly admissible
filtered $\varphi $-module over $K$ with $E$-coefficients and labeled
Hodge-Tate weights $\{0,-k_{i}\}_{\tau _{i}},$ then $(\wedge _{E\otimes
K}^{2}\mathbb{D},\wedge _{E\otimes K}^{2}\varphi )$ is weakly admissible
with labeled Hodge-Tate weights $\{-k_{i}\}_{\tau _{i}}.$
\end{lemma}

\begin{proof}
Let $\underline{\eta }=(\eta _{1},\eta _{2})$ be a standard basis of $(%
\mathbb{D},\varphi )$ such that $\mathrm{Mat}_{\underline{\eta }}(\varphi )$
is as in Proposition \ref{reducible section prop} and $\mathrm{Fil}^{\mathrm{%
j}}\mathbb{D}$ as in Formula \ref{Filitations}. Clearly $(\wedge ^{2}\varphi
)(\eta _{1}\wedge \eta _{2})=\vec{\alpha}\cdot \vec{\delta}(\eta _{1}\wedge
\eta _{2}).$ Since $\mathrm{Fil}^{\mathrm{j}}(\mathbb{D}\wedge \mathbb{D}%
)=\sum\limits_{j_{1}+j_{2}=j}(\mathrm{Fil}^{\mathrm{j}_{\mathrm{1}}}\mathbb{D%
}\wedge _{E\otimes K}\mathrm{Fil}^{\mathrm{j}_{\mathrm{2}}}\mathbb{D})$ and $%
J_{\vec{x}}\cup J_{\vec{y}}=I_{0},$ a simple computation yields 
\begin{equation*}
\mathrm{Fil}^{\mathrm{j}}(\mathbb{D}\wedge \mathbb{D})=\left\{ 
\begin{array}{l}
E^{\mid \tau _{I_{0}}\mid }(\eta _{1}\wedge \eta _{2})\text{ \ }\mathrm{if}%
\text{\ }j\leq w_{0}, \\ 
E^{\mid \tau _{I_{1}}\mid }(\eta _{1}\wedge \eta _{2})\text{ \ }\mathrm{if}%
\text{ \ }1+w_{0}\leq j\leq w_{1}, \\ 
\ \ \ \ \ \ \ \ \ \ \ \ \ \ \ \ \ \ \ \ \cdots \cdots \\ 
E^{\mid \tau _{I_{t-1}}\mid }(\eta _{1}\wedge \eta _{2})\text{\ }\mathrm{if}%
\text{\ }1+w_{t-2}\leq j\leq w_{t-1}, \\ 
\ \ \ \ \ \ \ \ \ \ 0\text{ \ \ \ \ \ \ \ \ \ \ }\mathrm{if}\text{\ }j\geq
1+w_{t-1}.%
\end{array}%
\right.
\end{equation*}%
from which the statement about the labeled Hodge-Tate weights follows
immediately. Weak admissibility is clear.
\end{proof}

\begin{corollary}
\label{detV lemma}If $V\ $is the crystalline representation corresponding to 
$\mathbb{D},$ then 
\begin{equation*}
\det V\simeq \eta \cdot \chi _{e_{0}}^{k_{1}}\cdot \chi
_{e_{1}}^{k_{2}}\cdot \cdots \cdot \chi _{e_{f-2}}^{k_{f-1}}\cdot \chi
_{e_{f-1}}^{k_{0}}\ \text{and }(\det \overline{V})_{\mid I_{K}}\simeq \omega
_{f,\bar{\tau}_{0}}^{\alpha },
\end{equation*}%
where $\eta $ is an unramified character of $G_{K}\ $and $\alpha
=-\tsum\limits_{i=0}^{f-1}p^{i}k_{i}.$
\end{corollary}

\begin{proof}
By Proposition \ref{The crystalline characters} and Lemma \ref{callmedet},
the crystalline character $\det V\otimes \left( \chi _{e_{0}}^{k_{1}}\cdot
\chi _{e_{1}}^{k_{2}}\cdot \cdots \cdot \chi _{e_{f-2}}^{k_{f-1}}\cdot \chi
_{e_{f-1}}^{k_{0}}\right) ^{-1}$ has labeled Hodge-Tate weights $\{0\}_{\tau
_{i}}.$ If the corresponding filtered $\varphi $-module has Frobenius
endomorphism $\varphi (\eta )=\vec{a}\cdot \eta ,$ then by Proposition \ref%
{rank 1 isom} $\mathrm{Nm}_{\varphi }(\vec{a})=c\cdot \vec{1}$ for some $%
c\in E^{\times }$ with $\mathrm{v}_{\mathrm{p}}(c)=0.$ By Lemma\ \ref{chic
prop} $\det V\otimes \left( \chi _{e_{0}}^{k_{1}}\cdot \chi
_{e_{1}}^{k_{2}}\cdot \cdots \cdot \chi _{e_{f-2}}^{k_{f-1}}\cdot \chi
_{e_{f-1}}^{k_{0}}\right) ^{-1}$ is the unramified character of $G_{K}$
which maps $\mathrm{Frob}_{K}$ to $c.$ The rest of the corollary follows
from Lemma \ref{BDJ lemma}.\noindent
\end{proof}

\noindent We recall the following well-known proposition, in which the field 
$K$ is assumed to have absolute inertia degree $f$ and need not be
unramified over $%
\mathbb{Q}
_{p}.$

\begin{proposition}
\cite[Prop. 2.7]{BRE07} \label{breuil prop}Let $\bar{\rho}:G_{K}\rightarrow 
\mathrm{GL}_{2}(\mathbb{\bar{F}}_{p})$ be a continuous representation. Then 
\begin{equation*}
\bar{\rho}_{\mid I_{K}}\simeq \left( 
\begin{array}{cc}
\omega _{2f}^{m} & \boldsymbol{\ast } \\ 
0 & \omega _{2f}^{mp^{f}}%
\end{array}%
\right)
\end{equation*}%
for some integer $m.$ The representation $\bar{\rho}$ is irreducible if and
only $1+p^{f}\nmid m,$ and in this case $\boldsymbol{\ast }=0.$
\end{proposition}

\begin{corollary}
\label{redmodp}Let $\chi $ be a crystalline character of $G_{K_{f}}$ with
labeled Hodge-Tate weights $\{-k_{i}\}_{\tau _{i}},$ with $k_{i}\ $arbitrary
integers for all $i=0,1,...,2f-1.$ If $V=\mathrm{Ind}_{K_{2f}}^{K_{f}}\left(
\chi \right) ,$ then $\overline{V}$ (reduction with respect to any lattice)
is irreducible if and only if $1+p^{f}\nmid
\tsum\limits_{i=0}^{2f-1}p^{i}k_{i}.$
\end{corollary}

\begin{proof}
Follows immediately from Lemma \ref{BDJ lemma} and Proposition \ref{breuil
prop}.
\end{proof}

\subsection{Reductions of reducible two-dimensional crystalline
representations \label{reducible modulo p section}}

In this section, we compute the semisimplified modulo$\ p$ reduction of any
reducible two-dimensional crystalline representation of $G_{K_{f}}.$

\begin{lemma}
\label{ermhs}Let $k_{0},k_{1},...,k_{f-1}$ be arbitrary integers and let 
\begin{equation}
\mathrm{Fil}^{\mathrm{j}}\mathbb{D}=\left\{ 
\begin{array}{l}
E^{\mid \tau _{I_{0}}\mid }\eta \text{ \ \ \ }\mathrm{if}\text{\ }j\leq
w_{0}, \\ 
E^{\mid \tau _{I_{1}}\mid }\eta \ \text{ \ \ }\mathrm{if}\text{ \ }%
1+w_{0}\leq j\leq w_{1}, \\ 
\ \ \ \ \ \ \ \ \ \ \ \ \ \cdots \cdots \\ 
E^{\mid \tau _{I_{t-1}}\mid }\eta \text{ \ }\mathrm{if}\text{\ }%
1+w_{t-2}\leq j\leq w_{t-1}, \\ 
\ \ \ 0\text{\ \ \ \ \ \ \ \ \ }\mathrm{if}\text{\ }j\geq 1+w_{t-1}.%
\end{array}%
\right.  \label{lemaf}
\end{equation}%
For each $i\in I_{0},$%
\begin{equation*}
e_{i}\mathrm{Fil}^{\mathrm{j}}\mathbb{D}=\left\{ 
\begin{array}{l}
e_{i}E^{\mid \tau _{I_{0}}\mid }\eta \ \ \ \text{if }j\leq k_{i}, \\ 
\ \ \ \ \ 0\ \ \ \ \ \ \ \text{if }r\geq 1+k_{i}.%
\end{array}%
\right.
\end{equation*}
\end{lemma}

\begin{proof}
Let $k_{i}=w_{r}$ for some $r\in \{1,...,t-1\}.$ Since $w_{r}>w_{r-1},$ $%
k_{i}>w_{r-1}$ and $i\in I_{r}.$ Also $i\not\in I_{r+1}$ since $k_{i}=w_{r}.$
The same assertion is clear for $r=0.$ Hence $e_{i}f_{I_{r}}=e_{i}$ and $%
e_{i}f_{I_{r+1}}=0.$ Multiplying formula \ref{lemaf} by $e_{i},$ we get%
\begin{equation*}
e_{i}\mathrm{Fil}^{\mathrm{j}}\mathbb{D}=\left\{ 
\begin{array}{l}
e_{i}E^{\mid \tau _{I_{0}}\mid }\eta \ \ \ \text{if }j\leq w_{r}, \\ 
\ \ \ \ \ 0\ \ \ \ \ \ \ \ \text{if }r\geq 1+w_{r}.%
\end{array}%
\right.
\end{equation*}
\end{proof}

\noindent Let $\underline{\eta }$ be an ordered basis of $\mathbb{D}$ as in
Proposition \ref{reducible section prop} and let $V$ be the corresponding
reducible crystalline representation. $V$ contains a subrepresentation $%
V_{2} $ which corresponds to the weakly admissible submodule $\mathbb{D}%
_{2}=\left( E^{\mid \tau \mid }\right) \eta _{2}.$ The filtration for the
ordered basis $\underline{\eta }$ is as in formula \ref{Filitations} for
some vectors $\vec{x},\vec{y}\in E^{\mid \tau \mid }.$ By Proposition 2.10
in \cite{DO10} (or by a direct computation), 
\begin{equation*}
\mathrm{Fil}^{\mathrm{j}}(\mathbb{D}_{2})=\mathbb{D}_{2}\cap \mathrm{Fil}^{%
\mathrm{j}}\mathbb{D}=\left\{ 
\begin{array}{l}
\mathbb{D}_{2}\text{\ \ \ \ \ \ \ \ \ \ \ }\mathrm{if}\text{\ \ }j\leq 0, \\ 
E^{\mid \tau _{I_{0,\vec{x}}}\mid }\eta _{2}\text{ \ \ }\mathrm{if}\text{\ \ 
}1\leq j\leq w_{0}, \\ 
\ \ \ \ \ \ \ \ \ \ \cdots \cdots \\ 
E^{\mid \tau _{I_{t-1,\vec{x}}}\mid }\eta _{2}\ \mathrm{if}\text{\ \ }%
1+w_{t-2}\leq j\leq w_{t-1}, \\ 
\ 0\text{\ \ \ \ \ \ \ \ \ \ \ \ \ }\mathrm{if}\text{ \ }j\geq 1+w_{t-1},%
\end{array}%
\right.
\end{equation*}%
where $I_{r,\vec{x}}=I_{r}\cap J_{\vec{x}}^{\prime }=\{i\in I_{r}:x_{i}=0\}.$
By Lemma \ref{ermhs},

\begin{equation*}
e_{i}\mathrm{Fil}^{\mathrm{j}}(\mathbb{D}_{2})=\left\{ 
\begin{array}{l}
e_{i}E^{\mid \tau \mid }\eta _{2}\ \ \ \ \ \text{if }j\leq 0 \\ 
e_{i}E^{\mid \tau \mid }f_{J_{\vec{x}}^{\prime }}\eta _{2}\ \text{if }1\leq
j\leq k_{i}, \\ 
\ \ \ 0\ \ \ \ \ \ \ \ \ \ \ \ \text{if }j\geq 1+k_{i},%
\end{array}%
\right.
\end{equation*}%
therefore the labeled Hodge-Tate weight of $\mathbb{D}_{2}$ with respect to
the embedding $\tau _{i}$ is 
\begin{equation*}
m_{i}=\left\{ 
\begin{array}{l}
0\ \ \ \text{if }x_{i}\neq 0~, \\ 
k_{i}\ \ \text{if }x_{i}=0,%
\end{array}%
\right.
\end{equation*}%
and $(\mathbb{D}_{2},\varphi _{2})$ corresponds to the effective crystalline
character $\chi _{c,\vec{0}}\cdot \chi _{e_{f-1}}^{m_{0}}\cdot \chi
_{e_{0}}^{m_{1}}\cdot \cdots \cdot \chi _{e_{f-2}}^{m_{f-1}},$ where $%
c=\left( \tprod\limits_{i\in I_{0}}\delta _{i}\right) \cdot
p^{-\sum\limits_{i\in I_{0}}k_{i}}$ and $\vec{\delta}=(\delta _{0},\delta
_{1},...,\delta _{f-1}).$ The following theorem follows immediately from
Corollary \ref{detV lemma}.\noindent

\begin{theorem}
\begin{enumerate}
\item 

\item[(i)] \label{reducible reductions}%
\begin{equation*}
\ V\simeq \left( 
\begin{array}{cc}
\psi _{1} & \boldsymbol{\ast } \\ 
0 & \psi _{2}%
\end{array}%
\right) ,\ 
\end{equation*}%
where $\psi _{1}=\eta _{1}\cdot \chi _{e_{f-1}}^{m_{0}}\cdot \chi
_{e_{0}}^{m_{1}}\cdot \cdots \cdot \chi _{e_{f-2}}^{m_{f-1}}$ and $\psi
_{2}=\eta _{2}\cdot \chi _{e_{0}}^{k_{1}-m_{1}}\cdot \chi
_{e_{1}}^{k_{2}-m_{2}}\cdot \cdots \cdot \chi
_{e_{f-2}}^{k_{f-1}-m_{f-1}}\cdot \chi _{e_{f-1}}^{k_{0}-m_{0}},$ where $%
\eta _{i}$ are unramified characters of $G_{K_{f}}.$

\item[(ii)] 
\begin{equation*}
\left( \overline{V}_{\mid I_{K}}\right) ^{s.s.}=\omega _{f,\bar{\tau}%
_{0}}^{\alpha _{1}}\oplus \omega _{f,\bar{\tau}_{0}}^{\alpha _{2}},\ 
\end{equation*}%
where $\alpha _{1}=-\tsum\limits_{i=0}^{f-1}m_{i}p^{i}$ and $\alpha
_{2}=\tsum\limits_{i=0}^{f-1}\left( m_{i}-k_{i}\right) p^{i}.$
\end{enumerate}
\end{theorem}

\noindent Notice that for an ordered basis is in Proposition \ref{reducible
section prop}, $\left( \overline{V}_{\mid I_{K_{f}}}\right) ^{s.s.}\ $only
depends on the filtration with respect to that basis.

\subsection{Proof of Theorem \protect\ref{irre}\label{irreducibles}}

Let $\{\ell _{i},\ell _{i+f}\}=\{0,k_{i}\}$ for $i=0,1,...,f-1$ and assume
that at least one $k_{i}$ is strictly positive. In this section we construct
infinite families of crystalline representations of Hodge-Tate type $%
\{0,-k_{i}\}_{\tau _{i}}\ $which contain the irreducible representations $V_{%
\vec{\ell}}=\mathrm{Ind}_{G_{K_{2f}}}^{G_{K_{f}}}\left( \chi _{e_{0}}^{\ell
_{1}}\cdot \chi _{e_{1}}^{\ell _{2}}\cdot \cdots \cdot \chi
_{e_{2f-2}}^{\ell _{2f-1}}\cdot \chi _{e_{2f-1}}^{\ell _{0}}\right) $ of
Proposition \ref{induction in char 0}, and have the same $\func{mod}p$
reductions with $V_{\vec{\ell}}.$ We choose $f$-tuples of matrices $\left(
\Pi _{1},\Pi _{2},...,\Pi _{f}\right) $ (with $\Pi _{f}=\Pi _{0}$), where
the types of the matrices $\Pi _{i}\ $(see Definition \ref{the 8 types
copy(1)}) are chosen as follows: \noindent

(1) If $\ell _{1}=0,$ $\Pi _{1}\in \{t_{2},t_{3}\};$

(\noindent 2) If $\ell _{1}=k_{1}>0,$ $\Pi _{1}\in \{t_{1},t_{4}\}.$

\noindent\noindent For $i=2,3,...,f-1,\ $we choose the type of the matrix $%
\Pi _{i}$ as follows:

\noindent

(1) If $\ell _{i}=0,$ then: \noindent

\begin{itemize}
\item If an even number of coordinates of $(\Pi _{1},\Pi _{2},...,\Pi
_{i-1}) $ is of even type, $\Pi _{i}\in \{t_{2},t_{3}\};$ \noindent

\item If an odd number of coordinates of $(\Pi _{1},\Pi _{2},...,\Pi _{i-1})$
is of even type, $\Pi _{i}\in \{t_{1},t_{4}\}.$
\end{itemize}

\noindent

(\noindent 2) If $\ell _{i}=k_{i}>0,$ then: \noindent

\begin{itemize}
\item If an even number of coordinates of $(\Pi _{1},\Pi _{2},...,\Pi
_{i-1}) $ is of even type, $\Pi _{i}\in \{t_{1},t_{4}\};$ \noindent

\item If an odd number of coordinates of $(\Pi _{1},\Pi _{2},...,\Pi _{i-1})$
is of even type, $\Pi _{i}\in \{t_{2},t_{3}\}.$
\end{itemize}

\noindent Finally, we choose the type of the matrix $\Pi _{0}$ as follows:

\noindent

(\noindent 1) If $\ell _{0}=0,$ then:

\begin{itemize}
\item \noindent If an even number of coordinates of $(\Pi _{1},\Pi
_{2},...,\Pi _{f-1})$ is of even type, $\Pi _{0}=t_{4};$

\item If an odd number of coordinates of $(\Pi _{1},\Pi _{2},...,\Pi _{f-1})$
is of even type, $\Pi _{0}=t_{3}.$
\end{itemize}

(\noindent 2) If $\ell _{0}=k_{0}>0,$ then:

\begin{itemize}
\item If an even number of coordinates of $(\Pi _{1},\Pi _{2},...,\Pi
_{f-1}) $ is of even type, $\Pi _{0}=t_{2};$

\item If an odd number of coordinates of $(\Pi _{1},\Pi _{2},...,\Pi _{f-1})$
is of even type, $\Pi _{0}=t_{1}.$
\end{itemize}

\noindent Notice that from the choice of $\Pi _{0},$ an odd number of
coordinates of $(\Pi _{1},\Pi _{2},...,\Pi _{f})$ is of even type. Let $\vec{%
i}=\left( i_{1},i_{2},...,i_{0}\right) \in \{1,2,3,4\}^{f}$ be the
type-vector attached to $(\Pi _{1},\Pi _{2},...,\Pi _{f}).\ $For the
matrices $\Pi _{i},$ we assume that $c_{i}=1$ for all $i.$ Let $P_{i}=\Pi
_{i}\func{mod}\pi $ for each $i$ and notice that from the choice of the
matrices $\Pi _{i}$ it follows that $\left( P_{1},P_{2},...,P_{f}\right)
\not\in C_{1}\cup C_{2}.$ The type of $P_{i}$ is defined to be the type of $%
\Pi _{i}.\ $For any $\vec{a}\in \mathfrak{m}_{E}^{f}$ we consider the
families of crystalline $E$-representations $V_{\vec{k}}^{\vec{i}}\left( 
\vec{a}\right) $ of $G_{K_{f}}$ with labeled Hodge-Tate weights $%
\{0,-k_{i}\}_{\tau _{i}}$ constructed in \S \ref{section families of
crystalline reps}. We prove the following.

\begin{proposition}
\begin{enumerate}
\item[(i)] \label{prop1}$V_{\vec{k}}^{\vec{i}}(\vec{0})=\mathrm{Ind}%
_{K_{2f}}^{K_{f}}\left( \chi _{e_{0}}^{\ell _{1}}\cdot \chi _{e_{1}}^{\ell
_{2}}\cdot \cdots \cdot \chi _{e_{2f-2}}^{\ell _{2f-1}}\cdot \chi
_{e_{2f-1}}^{\ell _{0}}\right) $ and $V_{\vec{k}}^{\vec{i}}(\vec{0})$ is
irreducible;\noindent

\item[(ii)] For any $\vec{a}\in \mathfrak{m}_{E}^{f},$ $\overline{V}_{\vec{k}%
}^{\vec{i}}\left( \vec{a}\right) =\overline{V}_{\vec{k}}^{\vec{i}}(\vec{0});$

\item[(iii)] For any $\vec{a}\in \mathfrak{m}_{E}^{f},$ $\left( \overline{V}%
_{\vec{k}}^{\vec{i}}\left( \vec{a}\right) _{\mid I_{K_{f}}}\right)
^{s.s.}=\omega _{2f,\bar{\tau}_{0}}^{\beta }\oplus \omega _{2f,\bar{\tau}%
_{0}}^{p^{f}\beta },$ where $\beta =-\tsum\limits_{i=0}^{2f-1}p^{i}\ell
_{i}; $

\item[(iv)] $\overline{V}_{\vec{k}}^{\vec{i}}\left( \vec{a}\right) $ is
irreducible if and only if $1+p^{f}\nmid \beta ;$

\item[(v)] Any irreducible member of the family $\left\{ V_{\vec{k}}^{\vec{i}%
}\left( \vec{a}\right) ,\ \vec{a}\in \mathfrak{m}_{E}^{f}\right\} ,$ other
than $V_{\vec{k}}^{\vec{i}}(\vec{0}),$ is non-induced.
\end{enumerate}
\end{proposition}

\begin{proof}
We restrict $V_{\vec{k}}^{\vec{i}}(\vec{0})$ to $G_{K_{2f}}.$ By the
construction of the representation $V_{\vec{k}}^{\vec{i}}(\vec{0})$ in \S %
\ref{the 4 type representations}, there exists some $G_{K_{f}}$-stable
lattice $\left( \mathrm{T}_{\vec{k}}^{\vec{i}}(\vec{0})\right) _{G_{K_{f}}}$
inside $V_{\vec{k}}^{\vec{i}}(\vec{0})$ whose Wach module has $\varphi $%
-action defined by $\left( \varphi \left( \eta _{1}\right) ,\varphi \left(
\eta _{2}\right) \right) =\left( \eta _{1},\eta _{2}\right) \Pi (\vec{0}),\ $%
where $\Pi (\vec{0})=\left( \Pi _{1}\left( 0\right) ,\Pi _{2}\left( 0\right)
,...,\Pi _{f-1}\left( 0\right) ,\Pi _{0}\left( 0\right) \right) .$ By
Proposition \ref{comment}, the Wach module of the $G_{K_{2f}}$-stable
lattice $\left( \mathrm{T}_{\vec{k}}^{\vec{i}}(\vec{0})\right) _{\mid
G_{K_{2f}}}$ inside $\left( V_{\vec{k}}^{\vec{i}}(\vec{0})\right) _{\mid
G_{K_{2f}}}$ is defined by $\left( \varphi \left( \eta _{1}\right) ,\varphi
\left( \eta _{2}\right) \right) =\left( \eta _{1},\eta _{2}\right) \Pi
\left( 0\right) ^{\otimes 2},$ therefore the filtered $\varphi $-module
corresponding to $\left( V_{\vec{k},\vec{0}}^{\vec{i}}\right) _{\mid
G_{K_{2f}}}$ has Frobenius endomorphism $\left( \varphi \left( \eta
_{1}\right) ,\varphi \left( \eta _{2}\right) \right) =\left( \eta _{1},\eta
_{2}\right) P(\vec{0})^{\otimes 2}.$ By Corollary \ref{weights of
restrictions}\ the restricted representation $\left( V_{\vec{k}}^{\vec{i}}(%
\vec{0})\right) _{\mid G_{K_{2f}}}$ has labeled Hodge-Tate weights $\left(
\{0,-k_{i}\}\right) _{\tau _{i}},$ $i=0,1,...,2f-1,$ with $k_{i+f}=k_{i}$
for $i=0,1,...,f-1,$ and filtration as in formula \ref{Filitations} for some
vectors $\vec{x},\vec{y},$ with the sets $I_{j}$ being defined for the $2f$
weights above. We prove that $\left( V_{\vec{k}}^{\vec{i}}(\vec{0})\right)
_{\mid G_{K_{2f}}}$ is reducible and determine its irreducible constituents.
First, we change the basis to diagonalize the matrix of Frobenius. We see
that 
\begin{equation*}
\ \ \ \ \ \ \ \ \ \ \ \ \ \ P_{i}\left( 0\right) =\left\{ 
\begin{array}{c}
R\left( \bar{\beta}_{i},\bar{\gamma}_{i}\right) ,\ \ \ \ \text{with }\{\bar{%
\beta}_{i},\bar{\gamma}_{i}\}=\{1,p^{k_{i}}\}\ \text{if }P_{i}\ \text{has
type }2\ \text{or }4, \\ 
\mathrm{diag}\left( \bar{\alpha}_{i},\bar{\delta}_{i}\right) ,\ \text{with }%
\{\bar{\alpha}_{i},\bar{\delta}_{i}\}=\{1,p^{k_{i}}\}\ \text{if }P_{i}\ 
\text{has type }1\ \text{or }3,%
\end{array}%
\ \right.
\end{equation*}%
$\allowbreak $where $R\left( \bar{\beta}_{i},\bar{\gamma}_{i}\right) $ is
the $2\times 2$ matrix with $\bar{\beta}_{i}$ in the $\left( 1,2\right) $
entry, $\bar{\gamma}_{i}$ in the $\left( 2,1\right) $ entry, and zero on the
diagonal. Let$\ Q_{0}=Id,\ $%
\begin{equation*}
Q_{1}=\left\{ 
\begin{array}{l}
Id\ \text{if }P_{1}\in \{t_{1},t_{3}\}, \\ 
R\ \ \text{if\ }P_{1}\in \{t_{2},t_{4}\},%
\end{array}%
\right.
\end{equation*}%
where$\ R:=R\left( 1,1\right) ,$%
\begin{equation}
Q_{i}=\left\{ 
\begin{array}{l}
Id\ \text{if }Q_{i-1}=Id\ \text{and\ }P_{i}\in \{t_{1},t_{3}\}, \\ 
R\ \text{if\ }Q_{i-1}=Id\ \ \text{and\ }P_{i}\in \{t_{2},t_{4}\}, \\ 
R\ \text{if\ }Q_{i-1}=R\ \ \ \text{and\ }P_{i}\in \{t_{1},t_{3}\}, \\ 
Id\mathrm{\ }\text{if\ }Q_{i-1}=R\ \ \text{and\ }P_{i}\in \{t_{2},t_{4}\}%
\end{array}%
\right.  \label{Qiiii}
\end{equation}%
for $i=2,3,...,2f-1.$ Let $Q=\left( Q_{0},Q_{1},...,Q_{2f-1}\right) ,$ then
by the definition of the matrices $Q_{i},\ $the\ matrix$\ Q\cdot P(\vec{0}%
)^{\otimes 2}\cdot \varphi \left( Q^{-1}\right) $ is diagonal. By induction, 
$Q_{0}=Id\ $and$\ $%
\begin{equation}
\ \ \ \ \ \ \ Q_{i}=\left\{ 
\begin{array}{l}
Id\ \text{if an even number of coordinates of }\left(
P_{1},P_{2},...,P_{i}\right) \ \text{is\ of\ even type,} \\ 
R\ \ \text{if an odd number of coordinates of }\left(
P_{1},P_{2},...,P_{i}\right) \ \text{is\ of\ even type,}%
\end{array}%
\right.  \label{Qi}
\end{equation}%
for $i=1,2,...,2f-1,$ where $P_{i+f}=P_{i}$ for $i=0,1,...,f-1.$ We claim
that for each $i=0,1,...,f-1,$ $Q_{i}=Id$ if and only if $Q_{i+f}=R.$
Indeed, for $i=0,$ $Q_{0}=Id.$ Since an odd number of coordinates of $\left(
P_{1},P_{2},...,P_{f}\right) $ is of even type, $Q_{f}=R.$ Let $q_{ij}^{r}$
be the $r$-th coordinate of the $\left( i,j\right) $-entry $\vec{q}_{ij}$ of 
$Q,$ for each $i,j\in \{1,2\}$ and $r\in \{0,1,...,2f-1\}.$ Assume that $%
i\in \{1,2,...,f-1\}.$ From the definition of the matrices $Q_{i}$ we see
that 
\begin{equation}
\ \ \ \ q_{11}^{i}=\left\{ 
\begin{array}{c}
1\ \text{if\ an\ even number of coordinates of }\left(
P_{1},P_{2},...,P_{i}\right) \ \text{is of even type,} \\ 
0\ \text{if an odd number of coordinates of }\left(
P_{1},P_{2},...,P_{i}\right) \ \text{is of even type.}%
\end{array}%
\right.  \label{q11i}
\end{equation}%
For any $i=0,1,...,f-1$ we have%
\begin{equation}
q_{11}^{i+f}=\left\{ 
\begin{array}{c}
1\ \text{if\ an\ even number of coordinates of }\left(
P_{1},P_{2},...,P_{f},...P_{i+f}\right) \ \text{is of even type,} \\ 
0\ \text{if\ an\ odd number of coordinates of }\left(
P_{1},P_{2},...,P_{f},...P_{i+f}\right) \ \text{is of even type.}%
\end{array}%
\right.  \label{qiii+ff}
\end{equation}%
Since an odd number of coordinates of $\left( P_{1},P_{2},...,P_{f}\right) $
is of even type and $P_{i}=P_{i+f}$ for all $i,$ this is equivalent to 
\begin{equation}
q_{11}^{i+f}=\left\{ 
\begin{array}{c}
1\ \text{if\ an\ odd number of coordinates of }\left(
P_{1},P_{2},...,P_{i}\right) \ \text{is of even type,} \\ 
0\ \text{if\ an\ even number of coordinates of }\left(
P_{1},P_{2},...,P_{i}\right) \ \text{is of even type,}%
\end{array}%
\right.  \label{qiii+f}
\end{equation}%
which implies that $q_{11}^{i+f}=1-q_{11}^{i}\ $for all $i=0,1,...,f-1.$
Similarly $q_{ij}^{r+f}=1-q_{ij}^{r}$ for all entries $ij.$ Consider the
ordered basis $\underline{\zeta }=\left( \zeta _{1},\zeta _{2}\right) $
defined by $\left( \zeta _{1},\zeta _{2}\right) :=\left( \eta _{1},\eta
_{2}\right) Q^{-1}.$ In the ordered basis $\underline{\zeta }$ the
filtration is as in formula \ref{Filitations} with the vector $\vec{x}\eta
_{1}+\vec{y}\eta _{2}$ replaced by $\vec{x}\cdot \left( \vec{q}_{11}\cdot
\zeta _{1}+\vec{q}_{12}\cdot \zeta _{2}\right) +\vec{y}\cdot \left( \vec{q}%
_{12}\cdot \zeta _{1}+\vec{q}_{22}\cdot \zeta _{2}\right) .$ Let $\vec{z}=%
\vec{x}\cdot \vec{q}_{11}+\vec{y}\cdot \vec{q}_{12}$ and $\vec{w}=\vec{x}%
\cdot \vec{q}_{12}+\vec{y}\cdot \vec{q}_{22}.$ From the definition of the
matrices $Q_{i},$ the matrix of Frobenius in this new basis is the diagonal
matrix 
\begin{equation*}
\mathrm{diag}\left( \vec{\lambda},\vec{\mu}\right) :=\left( Q_{0}\cdot
P_{1}\cdot Q_{1}^{-1},...,Q_{f-1}\cdot P_{f}\cdot Q_{f}^{-1},Q_{f}\cdot
P_{f+1}\cdot Q_{f+1}^{-1},...,Q_{2f-1}\cdot P_{0}\cdot Q_{0}^{-1}\right) .
\end{equation*}%
We prove that $\mathrm{Nm}_{\varphi }(\vec{\lambda})=\mathrm{Nm}_{\varphi
}\left( \vec{\mu}\right) =p^{\tsum\limits_{i=0}^{f-1}k_{i}}\cdot \vec{1}.$
First, we see $\lambda _{i}\mu _{i}=p^{k_{i}}$ for all $i.$ Since $Q_{i}=Id$
if and only if $Q_{i+f}=R,\ $bearing in mind that $P_{i+f}=P_{i},$ a case by
case analysis for the choices of $Q_{i}$ and $Q_{i+1}$ implies that $%
Q_{i}\cdot P_{i+1}\cdot Q_{i+1}^{-1}=\mathrm{diag}\left( \lambda _{i+1},\mu
_{i+1}\right) $ if and only if $Q_{i+f}\cdot P_{i+f+1}\cdot Q_{i+f+1}^{-1}=%
\mathrm{diag}\left( \mu _{i+1},\lambda _{i+1}\right) .$ Therefore,%
\begin{eqnarray*}
\tprod\limits_{i=0}^{2f-1}\left( Q_{i}\cdot P_{i+1}\cdot Q_{i+1}^{-1}\right)
&=&\tprod\limits_{i=0}^{f-1}\left( Q_{i}\cdot P_{i+1}\cdot
Q_{i+1}^{-1}\right) \cdot \tprod\limits_{i=0}^{f-1}\left( Q_{i+f}\cdot
P_{i}\cdot Q_{i+f+1}^{-1}\right) \\
&=&\tprod\limits_{i=0}^{f-1}\mathrm{diag}\left( \lambda _{i+1},\mu
_{i+1}\right) \cdot \tprod\limits_{i=0}^{f-1}\mathrm{diag}\left( \mu
_{i+1},\lambda _{i+1}\right) =p^{\tsum\limits_{i=0}^{f-1}k_{i}}\cdot Id.
\end{eqnarray*}%
Next, notice that $\vec{y}=\vec{1}-\vec{x}$ and $\vec{q}_{12}=\vec{1}-\vec{q}%
_{11},$ so $\vec{z}=\vec{x}\cdot \vec{q}_{11}+\left( \vec{1}-\vec{x}\right)
\cdot \left( \vec{1}-\vec{q}_{11}\right) =\vec{1}+2\cdot \vec{x}\cdot \vec{q}%
_{11}-\vec{q}_{11}-\vec{x}.$ Since $x_{i},q_{11}^{i}\in \{0,1\}$ for all $i,$
$z_{i}=0$ if and only if $q_{11}^{i}=1$ and $x_{i}=0,$ or $q_{11}^{i}=0$ and 
$x_{i}=1.$ Recall from formula\ \ref{Filitations1} that $x_{i}=0$ if and
only if $P_{i}\in \{t_{3},t_{4}\}$ and $x_{i}=1\ $if and only if $P_{i}\in
\{t_{1},t_{2}\}.$ This combined with the definition of the matrices $Q_{i}$
gives 
\begin{equation}
z_{i}=0\Leftrightarrow \left\{ 
\begin{array}{l}
i=0\ \text{and\ }P_{0}\in \{t_{3},t_{4}\}, \\ 
i\geq 1,\text{ }P_{i}\in \{t_{3},t_{4}\}\ \text{and an even number of
coordinates of }\left( P_{1},P_{2},...,P_{i}\right) \ \text{is of even type,}
\\ 
i\geq 1,\ P_{i}\in \{t_{1},t_{2}\}\ \text{and an odd number of coordinates
of }\left( P_{1},P_{2},...,P_{i}\right) \ \text{is of even type.}%
\end{array}%
\right.  \label{f1}
\end{equation}%
Similarly,%
\begin{equation}
z_{i}=1\Leftrightarrow \left\{ 
\begin{array}{l}
i=0\ \text{and\ }P_{0}\in \{t_{1},t_{2}\}, \\ 
i\geq 1,\text{ }P_{i}\in \{t_{1},t_{2}\}\ \text{and an even number of
coordinates of }\left( P_{1},P_{2},...,P_{i}\right) \ \text{is of even type,}
\\ 
i\geq 1,\ P_{i}\in \{t_{3},t_{4}\}\ \text{and an odd number of coordinates
of }\left( P_{1},P_{2},...,P_{i}\right) \ \text{is of even type.}%
\end{array}%
\right.  \label{f2}
\end{equation}%
We claim that $z_{i+f}=1-z_{i}$ for all $i=0,1,...,f-1.$ Indeed, $%
z_{i+f}=1+2\cdot x_{i+f}\cdot q_{11}^{i+f}-q_{11}^{i+f}-x_{i+f}.$ Since $%
P_{i}=P_{i+f},$ we have $x_{i}=x_{i+f}$ for all $i,$ and$\ $since $%
q_{11}^{i+f}=1-q_{11}^{f}$ we get $z_{i+f}=1-z_{i}.$ Since $z_{i}\in \{0,1\}$
for all $i,$ 
\begin{equation*}
\tsum\limits_{\substack{ i=0  \\ z_{i}=0}}^{2f-1}k_{i}=\tsum\limits 
_{\substack{ i=0  \\ z_{i}=0}}^{f-1}k_{i}+\tsum\limits_{\substack{ i=0  \\ %
z_{i+f}=0}}^{f-1}k_{i}=\tsum\limits_{i=0}^{f-1}k_{i}.
\end{equation*}%
Since $\mathrm{v}_{\mathrm{p}}\left( \mathrm{Nm}_{\varphi }\left( \vec{\mu}%
\right) \right) =\tsum\limits_{i=0}^{f-1}k_{i}=\tsum\limits_{\substack{ i=0 
\\ z_{i}=0}}^{2f-1}k_{i},$ by Proposition \ref{class thm}\footnote{%
F-semisimplicity is not assumed here, but the part of Proposition \ref{class
thm} used still holds.} the representation $\left( V_{\vec{k}}^{\vec{i}}(%
\vec{0})\right) _{\mid G_{K_{2f}}}$ is reducible. If $\mathbb{D}_{2}:=\left(
E^{\mid \tau _{K_{2f}}\mid }\right) \zeta _{2},$ by \cite[proof of Prop. 4.3]%
{DO10} (or by a direct computation), 
\begin{equation}
\ \ \ \ \ \ \ \ \ \ \ \ \ \ \mathrm{Fil}^{\mathrm{j}}\mathbb{D}_{2}=\left\{ 
\begin{array}{l}
\ \ \ \mathbb{D}_{2}\text{\ \ \ \ \ \ \ \ \ \ \ \ \ \ \ \ if\ \ }j\leq 0, \\ 
\left( E^{\mid \tau _{K_{2f}}\mid }f_{I_{i,\vec{z}}}\right) \zeta _{2}\text{%
\ if\ \ }1+w_{i-1}\leq j\leq w_{i},\ i=0,1,...,t-1, \\ 
\ \ \ 0\ \ \ \text{\ \ \ \ \ \ \ \ \ \ \ \ \ \ \ if \ }j\geq 1+w_{t-1},%
\end{array}%
\right.  \label{D2 weights}
\end{equation}%
where $I_{i.\vec{z}}=I_{i}\cap \{j\in \{0,1,...,2f-1\}:z_{j}=0\}.$ Let $i\in
\{0,1,...,2f-1\}$. Arguing as in Lemma \ref{ermhs} we see that%
\begin{equation*}
e_{i}\mathrm{Fil}^{\mathrm{j}}\mathbb{D}_{2}=\left\{ 
\begin{array}{l}
~\ \ e_{i}E^{\mid \tau _{K_{2f}}\mid }\zeta _{2}\text{\ \ \ \ \ \ \ \ \ \ \
\ if\ \ }j\leq 0, \\ 
e_{i}\left( \tsum\limits_{\substack{ r=0  \\ z_{r}=0}}^{2f-1}e_{i}\right)
E^{\mid \tau _{K_{2f}}\mid }\zeta _{2}\text{\ if\ \ }1\leq j\leq k_{i}, \\ 
\ \ \ \ \ \ \ \ 0\ \ \ \ \text{\ \ \ \ \ \ \ \ \ \ \ \ \ \ \ \ \ \ if \ }%
j\geq 1+k_{i}.%
\end{array}%
\right.
\end{equation*}%
Hence%
\begin{equation*}
e_{i}\mathrm{Fil}^{\mathrm{j}}\mathbb{D}_{2}=\left\{ 
\begin{array}{l}
e_{i}E^{\mid \tau _{K_{2f}}\mid }\zeta _{2}\text{\ if\ \ }j\leq 0, \\ 
\ \ \ 0\ \ \ \text{\ \ \ \ \ \ \ \ if \ }j\geq 1%
\end{array}%
\right.
\end{equation*}%
if $z_{i}=1$ and 
\begin{equation*}
e_{i}\mathrm{Fil}^{\mathrm{j}}\mathbb{D}_{2}=\left\{ 
\begin{array}{l}
e_{i}E^{\mid \tau _{K_{2f}}\mid }\zeta _{2}\text{\ if\ \ }j\leq k_{i}, \\ 
\ \ 0~\ \ \ \ \ \ \ \ \ \ \ \text{if \ }j\geq 1+k_{i}%
\end{array}%
\right.
\end{equation*}%
if $z_{i}=0.$ The labeled Hodge-Tate weight of $\mathbb{D}_{2}$ with respect
to the embedding $\tau _{i}\ $of $K_{2f}\ $into $E$ is $0$ if $z_{i}=1$ and $%
-k_{i}$ if $z_{i}=0.$ Next, we prove that 
\begin{equation*}
\ \ z_{i}=\left\{ 
\begin{array}{l}
0\ \text{if }\ell _{i}=0, \\ 
1\ \text{if }\ell _{i}=k_{i}>0%
\end{array}%
\right.
\end{equation*}%
for $i=0,1,...,f-1,$ and%
\begin{equation*}
z_{i+f}=\left\{ 
\begin{array}{l}
1\ \text{if }\ell _{i}=0, \\ 
0\ \text{if }\ell _{i}=k_{i}>0.%
\end{array}%
\right.
\end{equation*}%
Since $z_{i+f}=1-z_{i}$ for all $i=0,1,...,f-1,$ it suffices to prove the
first formula. Suppose that $\ell _{1}=0.$ Then $P_{1}\in \{t_{2},t_{3}\}\ $%
and by formula \ref{f1}, $z_{1}=0.$ If $\ell _{1}=k_{1}>0,$ then $P_{1}\in
\{t_{1},t_{4}\}$ and by formula \ref{f1}, $z_{1}=1.$ Let $i\in
\{1,2,...,f-2\}$ and assume that $\ell _{i}=0.$ If an even number of
coordinates of $\left( P_{1},P_{2},...,P_{i-1}\right) \ $is of even type,
then $P_{i}\in \{t_{2},t_{3}\}$ and formula \ref{f1} implies $z_{i}=0.$
Arguing similarly we see that if $\ell _{i}=k_{i}>0,$ formula \ref{f2}
implies $z_{i}=1.$ Finally, assume that $i=f$ and $\ell _{f}=0.$ If an even
number of coordinates of $\left( P_{1},P_{2},...,P_{f-1}\right) $ is of even
type, then $P_{0}=P_{f}=t_{4}$ and by formula \ref{f1} $z_{0}=z_{f}=0.$ We
finish the proof by verifying similarly the remaining cases. By the formulas
above, the labeled Hodge-Tate weight of $\mathbb{D}_{2}$ with respect to the
embedding $\tau _{i}$ is%
\begin{equation*}
\left\{ 
\begin{array}{l}
-k_{i}\ \text{if }\ell _{i}=0, \\ 
\ \ 0\ \ \ \text{if }\ell _{i}=k_{i}>0%
\end{array}%
\right.
\end{equation*}%
for $i=0,1,...,f-1$ and%
\begin{equation*}
\left\{ 
\begin{array}{l}
-k_{i}\ \text{if }\ell _{i}=k_{i}>0, \\ 
\ \ 0\ \ \ \text{if }\ell _{i}=0%
\end{array}%
\right.
\end{equation*}%
for $i=f,f+1,...,2f-1.$ Therefore the labeled Hodge-Tate weight of $\mathbb{D%
}_{2}$ with respect to the embedding $\tau _{i}$ is%
\begin{equation*}
\left\{ 
\begin{array}{l}
-\left( k_{i}-\ell _{i}\right) \ \text{if }i=0,1,...,f-1 \\ 
-\ell _{i}\ \ \ \ \ \ \ \ \ \ \text{if }i=f,f+1,...,2f-1.%
\end{array}%
\right.
\end{equation*}%
Since $\{\ell _{i},\ell _{i+f}\}=\{0,k_{i}\}$ for all $i=0,1,...,f-1,$ the
labeled Hodge-Tate weights of $\mathbb{D}_{2}$ are $\left( -\ell _{0},-\ell
_{1},...,-\ell _{f-1},-\ell _{f},-\ell _{f+1},...,-\ell _{2f-1}\right) .$
Since $\mathrm{Nm}_{\varphi }\left( \vec{\mu}\right)
=p^{\tsum\limits_{i=0}^{f-1}k_{i}}\cdot \vec{1},$ by Proposition \ref{rank 1
isom}\ the crystalline character corresponding to $\mathbb{D}_{2}$ is $\chi
_{e_{0}}^{\ell _{1}}\cdot \chi _{e_{1}}^{\ell _{2}}\cdot \cdots \cdot \chi
_{e_{2f-2}}^{\ell _{2f-1}}\cdot \chi _{e_{2f-1}}^{\ell _{0}}.$ Suppose that $%
V_{\vec{k}}^{\vec{i}}(\vec{0})$ is reducible. Then there exists some
irreducible constituent of $V_{\vec{k}}^{\vec{i}}(\vec{0})$ whose
restriction to $G_{K_{2f}}$ is $\chi _{e_{0}}^{\ell _{1}}\cdot \chi
_{e_{1}}^{\ell _{2}}\cdot \cdots \cdot \chi _{e_{2f-2}}^{\ell _{2f-1}}\cdot
\chi _{e_{2f-1}}^{\ell _{0}}.$ Since the labeled weights of this character
are $\left( -\ell _{0},-\ell _{1},...,-\ell _{f-1},-\ell _{f},...,-\ell
_{2f-1}\right) ,$ Corollary \ref{weights of restrictions} implies that $\ell
_{i}=\ell _{i+f}$ for all $i=0,1,...,f-1.$ Since $\{\ell _{i},\ell
_{i+f}\}=\{0,k_{i}\}$ for $i=0,1,...,f-1,$ and since some labeled weight is
strictly positive this is absurd. Hence $V_{\vec{k}}^{\vec{i}}(\vec{0})$ is
irreducible and its restriction to $G_{K_{2f}}$ contains $\chi
_{e_{0}}^{\ell _{1}}\cdot \chi _{e_{1}}^{\ell _{2}}\cdot \cdots \cdot \chi
_{e_{2f-2}}^{\ell _{2f-1}}\cdot \chi _{e_{2f-1}}^{\ell _{0}}$ as an
irreducible constituent. By Frobenius reciprocity, 
\begin{equation*}
V_{\vec{k}}^{\vec{i}}(\vec{0})=\mathrm{Ind}_{K_{2f}}^{K_{f}}\left( \chi
_{e_{0}}^{\ell _{1}}\cdot \chi _{e_{1}}^{\ell _{2}}\cdot \cdots \cdot \chi
_{e_{2f-2}}^{\ell _{2f-1}}\cdot \chi _{e_{2f-1}}^{\ell _{0}}\right) .
\end{equation*}%
This finishes the proof of part (i). Part (ii) follows from Theorem \ref%
{from blz} and parts (iii) and (4) follow from Corollary \ref{redmodp}. For
part (iv), notice that any irreducible induced member of $V_{\vec{k}}^{\vec{i%
}}\left( \vec{a}\right) $ has the form $\eta _{c}\cdot \mathrm{Ind}%
_{K_{2f}}^{K_{f}}\left( \chi _{e_{0}}^{\ell _{1}^{\prime }}\cdot \chi
_{e_{1}}^{\ell _{2}^{\prime }}\cdot \cdots \cdot \chi _{e_{2f-2}}^{\ell
_{2f-1}^{\prime }}\cdot \chi _{e_{2f-1}}^{\ell _{0}^{\prime }}\right) $ for
some unramified character $\eta _{c}$ and some nonnegative integers with $%
\{\ell _{i}^{\prime },\ell _{i+f}^{\prime }\}=\{0,k_{i}\}$ for all $i\ $%
(Proposition \ref{induction in char 0}). All the members of $V_{\vec{k}}^{%
\vec{i}}\left( \vec{a}\right) $ have determinant $\left( -1\right)
^{t}p^{\tsum\limits_{i=0}^{f-1}k_{i}},$ where $t$ is the number of
coordinates of $\vec{i}$ equaling $2$ or $4.\ $This equals the determinant
of $\mathrm{Ind}_{K_{2f}}^{K_{f}}\left( \chi _{e_{0}}^{\ell _{1}^{\prime
}}\cdot \chi _{e_{1}}^{\ell _{2}^{\prime }}\cdot \cdots \cdot \chi
_{e_{2f-2}}^{\ell _{2f-1}^{\prime }}\cdot \chi _{e_{2f-1}}^{\ell
_{0}^{\prime }}\right) $ and forces the unramified character $\eta _{c}$ to
be trivial. Hence the only irreducible induced member of the family $V_{\vec{%
k}}^{\vec{i}}\left( \vec{a}\right) $ is $V_{\vec{k}}^{\vec{i}}(\vec{0}).$
\end{proof}

\begin{remark}
\label{remark about types}Let $R$ be as in the proof of Proposition $\ref%
{prop1}$. If $\mathcal{A\ }$is a set of $2\times 2$ matrices, let $R\mathcal{%
A}:=\{R\cdot A:A\in \mathcal{A}\}\ $and $\mathcal{A}R:=\{A\cdot R:A\in 
\mathcal{A}\}.$ We write $\{t_{i},t_{j}\}$ for a set which contains matrices
of type $t_{i}$ and $t_{j}.$ Then $R\{t_{1},t_{2}\}=\{t_{1},t_{2}\},\
R\{t_{3},t_{4}\}=\{t_{3},t_{4}\},\ \{t_{1},t_{2}\}R=\{t_{3},t_{4}\}$ and $%
\{t_{3},t_{4}\}R=\{t_{1},t_{2}\}.$ In the definition of the matrices $P_{i}$
we may always assume that $P_{i}\in \{t_{1},t_{2}\}$ for all $i=1,2,...f-1.$
Indeed, let $Q_{0}=R,$ and for$\ i=1,2,...,f-1$ let 
\begin{equation*}
\ \ \ Q_{i}=\left\{ 
\begin{array}{c}
Id\mathrm{\ }\text{if }P_{i}\in \{t_{1},t_{2}\}, \\ 
R\ \ \text{if }P_{i}\in \{t_{3},t_{4}\}.%
\end{array}%
\right.
\end{equation*}%
After changing the basis by the matrix $Q=\left(
Q_{0},Q_{1},...,Q_{f-1}\right) $ we have $P_{i}\in \{t_{1},t_{2}\}$ for all $%
i=1,2,...f-1.$ By the definition preceding Proposition $\ref{prop1}$, the
type of the matrix $P_{0}\in \{t_{1},t_{2},t_{3},t_{4}\}$ is uniquely
determined by $\left( P_{1},P_{2},...,P_{f-1}\right) .$
\end{remark}

\begin{theorem}
Theorem $\ref{irre}$ holds.
\end{theorem}

\begin{proof}
Follows from Proposition $\ref{prop1}$ and Remark $\ref{remark about types}$.
\end{proof}

\begin{example}
\label{touparadeigmatos}Let $f=2$ and $k_{i}>0$ for $i=0,1.$ Up to twist by
some unramified character, there exist $2$ distinct isomorphism classes of
irreducible two dimensional crystalline $E$-representations of $G_{K_{2}}$
with labeled Hodge-Tate weights $\left( \{0,-k_{0}\},\{0,-k_{1}\}\right) \ $%
induced from $G_{K_{4}}.$

\noindent (i) If $\ell _{0}=k_{0}$ and $\ell _{1}=k_{1},$ then from the
definition of the matrices $\Pi _{i}$ preceding Proposition $\ref{prop1}$
and Remark $\ref{remark about types}$, $\left( \Pi _{1},\Pi _{0}\right)
=\left( t_{1},t_{2}\right) .$ Let $P_{i}=\Pi _{i}\func{mod}\pi .$ The
polynomials $z_{i}$ in the definition of the matrices $\Pi _{i}$ are such
that $z_{i}\equiv p^{m}\func{mod}\pi ,$ where $m:=\lfloor \frac{\max
\{k_{0},k_{1}\}-1}{p-1}\rfloor ,$ unless $k_{0}=k_{1}=p$ in which case we
define $m=0.$ For any $\vec{a}=\left( a_{0},a_{1}\right) \in \mathfrak{m}%
_{E}^{2}\ $we consider the family of crystalline representations $V_{\vec{k},%
\vec{a}}^{\left( 1,2\right) }$ constructed in \S $\ref{section families of
crystalline reps}$. The corresponding filtered $\varphi $-module is $\left( 
\mathbb{D}_{\vec{k},\vec{a}}^{\left( 1,2\right) },\varphi \right) ,$ with $%
\left( \varphi \left( \eta _{1}\right) ,\varphi \left( \eta _{2}\right)
\right) =\left( \eta _{1},\eta _{2}\right) P^{\left( 1,2\right) }(\vec{0}),$
where 
\begin{equation*}
P^{\left( 1,2\right) }(\vec{0})=\left( 
\begin{array}{cc}
\left( p^{k_{1}},a_{0}p^{m}\right) & \left( 0,1\right) \\ 
\left( a_{1}p^{m},p^{k_{0}}\right) & \left( 1,0\right)%
\end{array}%
\right)
\end{equation*}%
and filtration%
\begin{equation}
\ \ \ \ \ \ \ \ \ \ \ \ \ \ \ \ \ \ \ \ \ \ \ \mathrm{Fil}^{\mathrm{j}}(%
\mathbb{D}_{\vec{k},\vec{a}}^{\left( 1,2\right) })=\left\{ 
\begin{array}{l}
E^{2}\eta _{1}\oplus E^{2}\eta _{2}\ \ \ \ \ \ \ \ \ \ \ \ \mathrm{if}\text{
\ }j\leq 0, \\ 
E^{2}\left( \vec{x}\cdot \eta _{1}+\vec{y}\cdot \eta _{2}\right) \ \ \ \ \ 
\mathrm{if}\text{ \ }1\leq j\leq w_{0}, \\ 
E^{1}f_{I_{1}}\left( \vec{x}\cdot \eta _{1}+\vec{y}\cdot \eta _{2}\right) \
\ \mathrm{if}\text{ \ }1+w_{0}\leq j\leq w_{1}, \\ 
\ \ \ \ \ \ \ \ \ \ \ \ 0\ \ \ \ \ \ \ \ \ \ \ \ \ \ \ \ \ \mathrm{if}\text{
\ }j\geq 1+w_{1},%
\end{array}%
\right.  \label{example f=2 filtration}
\end{equation}%
with $w_{0}=\min \{k_{0},k_{1}\}$ and $w_{1}=\max \{k_{0},k_{1}\},$ 
\begin{equation*}
f_{I_{1}}=\left\{ 
\begin{array}{c}
\left( 0,1\right) \ \mathrm{if}\text{ }k_{0}<k_{1}, \\ 
\left( 1,0\right) \ \mathrm{if}\text{ }k_{1}<k_{0}, \\ 
\left( 0,0\right) \ \mathrm{if}\text{ }k_{0}=k_{1},%
\end{array}%
\right.
\end{equation*}%
and$\ \left( \vec{x},\vec{y}\right) =\left( \left( 1,1\right) ,\left(
0,0\right) \right) .$ We have 
\begin{equation*}
V_{\vec{k},\vec{0}}^{\left( 1,2\right) }\simeq \mathrm{Ind}%
_{K_{4}}^{K_{2}}\left( \chi _{e_{0}}^{k_{1}}\cdot \chi
_{e_{3}}^{k_{0}}\right) ,
\end{equation*}%
and for any $\vec{a}\in \mathfrak{m}_{E}^{2},$ 
\begin{equation*}
\left( \left( \overline{V}_{\vec{k},\vec{a}}^{\left( 1,2\right) }\right)
_{\mid I_{K_{2}}}\right) ^{s.s.}\simeq \omega _{4,\bar{\tau}_{0}}^{-\left(
k_{0}+pk_{1}\right) }\oplus \omega _{4,\bar{\tau}_{0}}^{-\left(
k_{0}+pk_{1}\right) p^{2}}.
\end{equation*}%
Let $\alpha _{i}=a_{i}p^{m}\ $and $A=\alpha _{1}+p^{k_{1}}\alpha _{0}.\ $%
Assume that $A^{2}\neq -4p^{k_{0}+k_{1}}$ and let $\varepsilon _{0}\neq
\varepsilon _{1}$ be the distinct roots of the characteristic polynomial $%
X^{2}-A\cdot X+p^{k_{0}+k_{1}}.$ Arguing as in the proof of Proposition $2.2$
in \cite{DO10}, we get the following \textquotedblleft standard
parametrization\textquotedblright\ for the family $V_{\vec{k},\vec{a}%
}^{\left( 1,2\right) }:$%
\begin{equation*}
\varphi \left( \eta _{1}\right) =\left( 1,\varepsilon _{0}\right) \eta
_{1},\ \varphi \left( \eta _{2}\right) =\left( \lambda ,\frac{\varepsilon
_{1}}{\lambda }\right) \eta _{2},
\end{equation*}%
where 
\begin{equation*}
\lambda =\lambda \left( \alpha _{0}\right) =\frac{\varepsilon _{1}}{%
\varepsilon _{0}}\cdot \frac{\left( \varepsilon _{1}-A+p^{k_{1}}\alpha
_{0}\right) }{\left( \varepsilon _{0}-A+p^{k_{1}}\alpha _{0}\right) }
\end{equation*}%
(notice that $\varepsilon _{i}\neq A-\alpha _{0}p^{k_{1}}$), and the
filtration is as in Formula \ref{example f=2 filtration} with $\vec{x}=\vec{y%
}=\vec{1}.$

\noindent (ii) If $\ell _{0}=\ell _{1}=0.$ Again, taking into account Remark 
$\ref{remark about types}$, we may only consider the case $\left( \Pi
_{1},\Pi _{0}\right) =\left( t_{2},t_{3}\right) .$ For each $\vec{a}\in 
\mathfrak{m}_{E}^{2}$ consider the family $V_{\vec{k},\vec{a}}^{\left(
2,3\right) }$ of two-dimensional crystalline $E$-representations of $%
G_{K_{2}}$ with labeled Hodge-Tate weights $\{0,-k_{i}\}_{\tau _{i}},$ $%
i=0,1.$ We have 
\begin{equation*}
V_{\vec{k},\vec{0}}^{\left( 2,3\right) }\simeq \mathrm{Ind}%
_{K_{4}}^{K_{2}}\left( \chi _{e_{2}}^{k_{1}}\cdot \chi
_{e_{1}}^{k_{0}}\right) \simeq \mathrm{Ind}_{K_{4}}^{K_{2}}\left( \chi
_{e_{0}}^{k_{1}}\cdot \chi _{e_{3}}^{k_{0}}\right) .
\end{equation*}%
For any $\vec{a}\in \mathfrak{m}_{E}^{2},$%
\begin{equation*}
\left( \left( \overline{V}_{\vec{k},\vec{a}}^{\left( 2,3\right) }\right)
_{\mid I_{K_{2}}}\right) ^{s.s.}\simeq \omega _{4,\tau _{0}}^{-\left(
k_{0}+pk_{1}\right) }\oplus \omega _{4,\tau _{0}}^{-\left(
k_{0}+pk_{1}\right) p^{2}}.
\end{equation*}%
Notice that the family $\left\{ V_{\vec{k},\vec{a}}^{\left( 1,2\right) },\ 
\vec{a}\in \mathfrak{m}_{E}^{2}\right\} $ of case (i) coincides with the
family $\left\{ V_{\vec{k},\vec{a}}^{\left( 2,3\right) },\ \vec{a}\in 
\mathfrak{m}_{E}^{2}\right\} ,$ as the second family is obtained from the
first one by changing the basis by the matrix $Q=\left( R,R\right) .$

\noindent (iii) Let $f=2,$ $\ell _{0}=0$ and $\ell _{1}=k_{1}>0.$ Then $%
\left( \Pi _{1},\Pi _{0}\right) =\left( t_{1},t_{4}\right) .\ $For each $%
\vec{a}\in \mathfrak{m}_{E}^{2}$ consider the family $V_{\vec{k},\vec{a}%
}^{\left( 1,4\right) }$ of two-dimensional crystalline $E$-representations
of $G_{K_{2}}$ with labeled Hodge-Tate weights $\{0,-k_{i}\}_{\tau _{i}},$ $%
i=0,1.$ Then 
\begin{equation*}
V_{\vec{k},\vec{0}}^{\left( 1,4\right) }\simeq \mathrm{Ind}%
_{K_{4}}^{K_{2}}\left( \chi _{e_{0}}^{k_{1}}\cdot \chi
_{e_{1}}^{k_{0}}\right) ,
\end{equation*}%
and for any $\vec{a}\in \mathfrak{m}_{E}^{2},$ 
\begin{equation*}
\left( \left( \overline{V}_{\vec{k},\vec{a}}^{\left( 1,4\right) }\right)
_{\mid I_{K_{2}}}\right) ^{s.s.}\simeq \omega _{4,\bar{\tau}_{0}}^{-\left(
pk_{1}+p^{2}k_{0}\right) }\oplus \omega _{4,\bar{\tau}_{0}}^{-\left(
pk_{1}+p^{2}k_{0}\right) p^{2}}.
\end{equation*}%
Let $\alpha _{i}=a_{i}p^{m}\ $and $A=\alpha _{0}+p^{k_{0}}\alpha _{1}.\ $%
Assume that $A^{2}\neq -4p^{k_{0}+k_{1}}$ and let $\varepsilon _{0}\neq
\varepsilon _{1}$ be the distinct roots of the characteristic polynomial $%
X^{2}-A\cdot X+p^{k_{0}+k_{1}}.$ Arguing as in the proof of Proposition $2.2$
in \cite{DO10}, we get the following \textquotedblleft standard
parametrization\textquotedblright\ for the family $V_{\vec{k},\vec{a}%
}^{\left( 1,4\right) }:$%
\begin{equation*}
\varphi \left( \eta _{1}\right) =\left( 1,\varepsilon _{0}\right) \eta
_{1},\ \varphi \left( \eta _{2}\right) =\left( \lambda ,\frac{\varepsilon
_{1}}{\lambda }\right) \eta _{2},
\end{equation*}%
where 
\begin{equation*}
\lambda =\lambda \left( \alpha _{1}\right) =\left( \frac{\varepsilon _{1}}{%
\varepsilon _{0}}\right) ^{2}\cdot \frac{\left( A-p^{k_{0}}\alpha
_{1}-\varepsilon _{0}\right) }{\left( A-p^{k_{0}}\alpha _{1}-\varepsilon
_{1}\right) }
\end{equation*}%
(notice that $\varepsilon _{i}\neq A-\alpha _{1}p^{k_{0}}$), and the
filtration is as in Formula \ref{example f=2 filtration} with $\vec{x}=\vec{y%
}=\vec{1}.$

\noindent (iv) If $f=2$ $\ell _{0}=k_{0}>0$ and $\ell _{1}=0.$ Then $\left(
\Pi _{1},\Pi _{0}\right) =\left( t_{2},t_{1}\right) $ and this gives the
family obtained in case (iii).
\end{example}

\begin{example}
If $f=2,$ $k_{0}>0$ and $k_{1}=0.$ Then $\mathrm{Ind}_{K_{4}}^{K_{2}}\left(
\chi _{e_{3}}^{k_{0}}\right) $ is a unique isomorphism class of
two-dimensional crystalline irreducible $E$-representations of $G_{K_{2}}$
with labeled weights $\left( \{0,-k_{0}\},\{0,0\}\right) $ induced from $%
G_{K_{4}}.$ We have 
\begin{equation*}
V_{\vec{k},\vec{0}}^{\left( 2,3\right) }\simeq \mathrm{Ind}%
_{K_{4}}^{K_{2}}\left( \chi _{e_{3}}^{k_{0}}\right) \simeq \mathrm{Ind}%
_{K_{4}}^{K_{2}}\left( \chi _{e_{1}}^{k_{0}}\right) ,
\end{equation*}%
and for any $\vec{a}\in \mathfrak{m}_{E}^{2},$%
\begin{equation*}
\left( \left( \overline{V}_{\vec{k},\vec{a}}^{\left( 2,3\right) }\right)
_{\mid I_{K_{2}}}\right) ^{s.s.}\simeq \omega _{4,\bar{\tau}%
_{0}}^{-k_{0}}\oplus \omega _{4,\bar{\tau}_{0}}^{-p^{2}k_{0}}.
\end{equation*}
\end{example}

\begin{example}
Let $f=3,$ $k_{i}>0$ for all $i=0,1,2.$ Up to twist by some unramified
character, there exist $4$ distinct isomorphism classes of irreducible
two-dimensional crystalline $E$-representations of $G_{K_{3}}$ with labeled
Hodge-Tate weights $\left( \{0,-k_{0}\},\{0,-k_{1}\},\{0,-k_{2}\}\right) $
induced from $G_{K_{6}}.$ One of those classes is represented by $\mathrm{Ind%
}_{K_{6}}^{K_{3}}\left( \chi _{e_{0}}^{k_{1}}\cdot \chi
_{e_{1}}^{k_{2}}\cdot \chi _{e_{2}}^{k_{0}}\right) .$ For the families
containing it we have $\ell _{i}=k_{i}>0$ for all $i=0,1,2.$ Since $k_{0}>0,$
$\Pi _{0}=t_{2}$ if $\Pi _{2}=t_{4}$ and $\Pi _{0}=t_{1}$ if $\Pi
_{2}=t_{1}. $ Hence $\left( \Pi _{1},\Pi _{2},\Pi _{0}\right) \in \{\left(
t_{4},t_{4},t_{2}\right) ,\left( t_{4},t_{1},t_{1}\right) ,\left(
t_{1},t_{2},t_{1}\right) ,\left( t_{1},t_{3},t_{2}\right) \}.$ By Remark $%
\ref{remark about types}$ we may only consider the case $\left( \Pi _{1},\Pi
_{2},\Pi _{0}\right) =\left( t_{1},t_{2},t_{1}\right) .$ For any $\vec{a}\in 
\mathfrak{m}_{E}^{3},$ consider the the families $V_{\vec{k},\vec{a}%
}^{\left( 1,2,1\right) }$ of two-dimensional crystalline $E$-representations
of $G_{K_{3}}$ with labeled Hodge-Tate weights $\{0,-k_{i}\}_{\tau _{i}},$ $%
i=0,1,2.$ We have 
\begin{equation*}
V_{\vec{k},\vec{0}}^{\left( 1,2,1\right) }\simeq \mathrm{Ind}%
_{K_{6}}^{K_{3}}\left( \chi _{e_{0}}^{k_{1}}\cdot \chi _{e_{1}}^{k_{2}}\cdot
\chi _{e_{2}}^{k_{0}}\right) ,
\end{equation*}%
and for any $\vec{a}\in \mathfrak{m}_{E}^{3},$ 
\begin{equation*}
\left( \left( \overline{V}_{\vec{k},\vec{a}}^{\left( 1,2,1\right) }\right)
_{\mid I_{K_{3}}}\right) ^{s.s.}\simeq \omega _{6,\bar{\tau}_{0}}^{-\left(
k_{0}+pk_{1}+p^{2}k_{2}\right) }\oplus \omega _{6,\bar{\tau}_{0}}^{-\left(
k_{0}+pk_{1}+p^{2}k_{2}\right) p^{3}}.
\end{equation*}
\end{example}

\subsection{Proof of Theorem \protect\ref{reducibles0}\label{red}}

Let $V_{\vec{\ell},\vec{\ell}^{\prime }}\left( \eta \right) =\eta _{c}\cdot
\chi _{e_{0}}^{\ell _{1}}\cdot \chi _{e_{1}}^{\ell _{2}}\cdot \cdots \cdot
\chi _{e_{f-1}}^{\ell _{0}}\oplus \chi _{e_{0}}^{\ell _{1}^{\prime }}\cdot
\chi _{e_{1}}^{\ell _{2}^{\prime }}\cdot \cdots \cdot \chi _{e_{f-1}}^{\ell
_{0}^{\prime }}$ with $\{\ell _{i},\ell _{i}^{\prime }\}=\{0,k_{i}\}$ for
all $i,$ where $\eta _{c}$ is the unramified character of $G_{K_{f}}$ which
maps $\mathrm{Frob}_{K_{f}}$ to $c\in \mathcal{O}_{E}^{\times }.$ As usual,
we assume that at least one $k_{i}$ is strictly positive. We choose $f$%
-tuples of matrices $\left( \Pi _{1},\Pi _{2},...,\Pi _{f}\right) $ (with $%
\Pi _{f}=\Pi _{0}$) as follows:

(1) If $\ell _{1}=0,$ $\Pi _{1}\in \{t_{2},t_{3}\};$

(2) If $\ell _{1}=k_{1}>0,$ $\Pi _{1}\in \{t_{1},t_{4}\}.$

\noindent For $i=2,3,...,f-1,\ $\noindent we choose the type of the matrix $%
\Pi _{i}$ as follows:

(1) If $\ell _{i}=0,$ then: \noindent

\begin{itemize}
\item If an even number of coordinates of $\left( \Pi _{1},\Pi _{2},...,\Pi
_{i-1}\right) $ is of even type, $\Pi _{i}\in \{t_{2},t_{3}\};$

\item If an odd number of coordinates of $\left( \Pi _{1},\Pi _{2},...,\Pi
_{i-1}\right) $ is of even type, $\Pi _{i}\in \{t_{1},t_{4}\}.$ \noindent
\end{itemize}

\noindent

(2) If $\ell _{i}=k_{i}>0,$ then: \noindent

\begin{itemize}
\item If an even number of coordinates of $\left( \Pi _{1},\Pi _{2},...,\Pi
_{i-1}\right) $ is of even type, $\Pi _{i}\in \{t_{1},t_{4}\};$

\item If an odd number of coordinates of $\left( \Pi _{1},\Pi _{2},...,\Pi
_{i-1}\right) $ is of even type, $\Pi _{i}\in \{t_{2},t_{3}\}.$
\end{itemize}

Finally, we choose the type of the matrix $\Pi _{0}$ as follows: \noindent

(1) If $\ell _{0}=0,$ then: \noindent

\begin{itemize}
\item If an even number of coordinates of $\left( \Pi _{1},\Pi _{2},...,\Pi
_{f-1}\right) $ is of even type, $\Pi _{0}=t_{3};$

\item If an odd number of coordinates of $\left( \Pi _{1},\Pi _{2},...,\Pi
_{f-1}\right) $ is of even type, $\Pi _{0}=t_{4}.$
\end{itemize}

\noindent

(2) If $\ell _{0}=k_{i}>0,$ then:

\begin{itemize}
\item If an even number of coordinates of $\left( \Pi _{1},\Pi _{2},...,\Pi
_{f-1}\right) $ is of even type, $\Pi _{0}=t_{1};$

\item If an odd number of coordinates of $\left( \Pi _{1},\Pi _{2},...,\Pi
_{f-1}\right) $ is of even type, $\Pi _{0}=t_{2}.$
\end{itemize}

\noindent Notice that from the choice of $\Pi _{0},$ an even number of
coordinates of $\left( \Pi _{1},\Pi _{2},...,\Pi _{f}\right) $ is of even
type. If in the proposition below $\eta =\eta _{c}$ is the unramified
character which maps $\mathrm{Frob}_{K_{f}}$ to $c,$ we choose the units
appearing in the entries of the matrices $\Pi _{i}$ such that $c_{i}=1$ for
all $i=1,2,...,f-1,$ while $c_{0}$ will be chosen appropriately. Let $\vec{i}
$ be the type-vector attached to $\left( \Pi _{1},\Pi _{2},...,\Pi
_{f}\right) .$ We exclude those vectors $\vec{i}$ for which $\left( \Pi
_{1},\Pi _{2},...,\Pi _{f}\right) \in C_{1}\cup C_{2}.$ That is to exclude
the cases where $\vec{\ell}=\vec{0}$ or $\vec{\ell}^{\prime }=\vec{0}.$ For
any $\vec{a}\in \mathfrak{m}_{E}^{f}$ we consider the families of
crystalline $E$-representations $V_{\vec{k}}^{\vec{i}}\left( \vec{a}\right) $
of $G_{K_{f}}$ with labeled Hodge-Tate weights $\{0,-k_{i}\}_{\tau _{i}}$
constructed in \S \ref{section families of crystalline reps}.

\begin{proposition}
\begin{enumerate}
\item[(i)] \label{se xreiaaaaaaaaaazomai!!!}For any $\vec{i}$ as above, $V_{%
\vec{k}}^{\vec{i}}(\vec{0})\simeq V_{\vec{\ell},\vec{\ell}^{\prime }}\left(
\eta \right) ,\ $after possibly twisting $V_{\vec{k}}^{\vec{i}}(\vec{0})$ by
some unramified character;

\item[(ii)] For any $\vec{a}\in \mathfrak{m}_{E}^{f},\ \overline{V}_{\vec{k}%
}^{\vec{i}}\left( \vec{a}\right) \simeq \overline{V}_{\vec{k}}^{\vec{i}}(%
\vec{0})$ and%
\begin{equation*}
\left( \overline{V}_{\vec{k}}^{\vec{i}}\left( \vec{a}\right) \right) _{\mid
I_{K_{f}}}\simeq \left( \overline{V}_{\vec{\ell},\vec{\ell}^{\prime }}\left(
\eta \right) \right) _{\mid I_{K_{f}}}\simeq \omega _{f,\bar{\tau}%
_{0}}^{\alpha }\oplus \omega _{f,\bar{\tau}_{0}}^{\alpha ^{\prime }},
\end{equation*}%
where $\alpha =-\tsum\limits_{i=i}^{f-1}\ell _{i}p^{i}$ and $\alpha ^{\prime
}=-\tsum\limits_{i=0}^{f-1}\ell _{i}^{\prime }p^{i}.$
\end{enumerate}
\end{proposition}

\begin{proof}
For simplicity assume that $\eta =1.$ The general case follows by choosing
the unit $c_{0}$ in the definition of $\Pi _{0}$ appropriately. We restrict $%
V_{\vec{k}}^{\vec{i}}(\vec{0})$ to $G_{K_{2f}}.$ By the construction of the
representation $V_{\vec{k}}^{\vec{i}}(\vec{0})$ in \S \ref{the 4 type
representations}, there exists some $G_{K_{f}}$-stable lattice $\left( 
\mathrm{T}_{\vec{k}}^{\vec{i}}(\vec{0})\right) _{G_{K_{f}}}$ inside $V_{\vec{%
k}}^{\vec{i}}(\vec{0})$ whose Wach module has $\varphi $-action defined by $%
\left( \varphi \left( \eta _{1}\right) ,\varphi \left( \eta _{2}\right)
\right) =\left( \eta _{1},\eta _{2}\right) \cdot \Pi (\vec{0}).$ By
Proposition \ref{comment}, the Wach module of the $G_{K_{2f}}$-stable
lattice $\left( \mathrm{T}_{\vec{k}}^{\vec{i}}(\vec{0})\right) _{\mid
G_{K_{2f}}}$ inside $\left( V_{\vec{k}}^{\vec{i}}(\vec{0})\right) _{\mid
G_{K_{2f}}}$ is defined by $\left( \varphi \left( \eta _{1}\right) ,\varphi
\left( \eta _{2}\right) \right) =\left( \eta _{1},\eta _{2}\right) \cdot \Pi
\left( 0\right) ^{\otimes 2},$ therefore the filtered $\varphi $-module
corresponding to $\left( V_{\vec{k}}^{\vec{i}}(\vec{0})\right) _{\mid
G_{K_{2f}}}$ has Frobenius endomorphism $\left( \varphi \left( \eta
_{1}\right) ,\varphi \left( \eta _{2}\right) \right) =\left( \eta _{1},\eta
_{2}\right) \cdot P\left( 0\right) ^{\otimes 2}.$ The restricted
representation $\left( V_{\vec{k}}^{\vec{i}}(\vec{0})\right) _{\mid
G_{K_{2f}}}$ has labeled weights $\left( \{0,-k_{i}\}\right) _{\tau _{i}},$ $%
i=0,1,...,2f-1,$ with $k_{i+f}=k_{i}$ for $i=0,1,...,f-1,$ and filtration as
in formula \ref{Filitations} for some vectors $\vec{x},\vec{y},$ with the
sets $I_{j}$ being defined for the $2f$ weights above. We prove that $\left(
V_{\vec{k}}^{\vec{i}}(\vec{0})\right) _{\mid G_{K_{2f}}}$ is reducible and
determine its irreducible constituents. First, we change the basis to
diagonalize the matrix of Frobenius. We define matrices $Q_{i}$ as in the
proof of Proposition \ref{prop1}, and we let $Q=\left(
Q_{0},Q_{1},...,Q_{2f-1}\right) ,$ then by the definition of the matrices $%
Q_{i},\ $the\ matrix$\ Q\cdot P\left( 0\right) ^{\otimes 2}\cdot \varphi
\left( Q^{-1}\right) $ is diagonal. By the proof of Proposition \ref{prop1}, 
$Q_{0}=$ $Id\ $and for $i=1,2,...,2f-1,$ $Q_{i}$ is as in formula \ref{Qi}.
We claim that for each $i=0,1,...,f-1,$ $Q_{i}=Q_{i+f}.$ Indeed, from the
definition of the matrices $Q_{i}$ we see that $q_{11}^{i}$ and $%
q_{11}^{i+f} $ are as in formulas \ref{q11i} and \ref{qiii+ff} respectively
in the proof of Proposition \ref{prop1}. Since an even number of coordinates
of $\left( P_{1},P_{2},...,P_{f}\right) $ are of even type, $%
q_{11}^{i+f}=q_{11}^{i}.$ Similarly $q_{ij}^{r+f}=q_{ij}^{r}$ for any entry $%
\left( i,j\right) .$ Consider the ordered basis $\underline{\zeta }=\left(
\zeta _{1},\zeta _{2}\right) $ defined by $\left( \zeta _{1},\zeta
_{2}\right) :=\left( \eta _{1},\eta _{2}\right) \cdot Q^{-1}.$ Let $\vec{q}%
_{ij}$ be th $(i,j)$-entry of $Q.$ In the new basis $\underline{\zeta }$ the
filtration is as in formula \ref{Filitations} with the vector $\vec{x}\eta
_{1}+\vec{y}\eta _{2}$ replaced by $\vec{x}\cdot \left( \vec{q}_{11}\cdot
\zeta _{1}+\vec{q}_{12}\cdot \zeta _{2}\right) +\vec{y}\cdot \left( \vec{q}%
_{12}\cdot \zeta _{1}+\vec{q}_{22}\cdot \zeta _{2}\right) .$ Let $\vec{z}=%
\vec{x}\cdot \vec{q}_{11}+\vec{y}\cdot \vec{q}_{12}$ and $\vec{w}=\vec{x}%
\cdot \vec{q}_{12}+\vec{y}\cdot \vec{q}_{22}.$ The matrix of Frobenius in
this new basis is the diagonal matrix $\mathrm{diag}\left( \vec{\lambda},%
\vec{\mu}\right) .$ Arguing as in Proposition \ref{prop1}, and taking into
account that $q_{ij}^{r+f}=q_{ij}^{r}$ for all $r=0,1,...,f-1$ and all
entries $\left( i,j\right) $ we see that $z_{r+f}=z_{r}$ for all $r.$ From
the proof of the same proposition, $z_{i}=0$ if and only if $q_{11}^{i}=1$
and $x_{i}=0$ or $q_{11}^{i}=0$ and $x_{i}=1.$ From formula \ref%
{Filitations1} it follows that $x_{i}=0$ if and only if $P_{i}\in
\{t_{4},t_{3}\}$ and $x_{i}=1\ $if and only if $P_{i}\in \{t_{2},t_{1}\}.$
Since $z_{i}=z_{i+f}\ $and $k_{i}=k_{i+f} $ for all $i=0,1,...,f-1,$%
\begin{equation*}
\ \ \tsum\limits_{\substack{ i=0  \\ z_{i}=0}}^{2f-1}k_{i}=2\tsum\limits 
_{\substack{ i=0  \\ z_{i}=0}}^{f-1}k_{i}=2\tsum\limits_{\substack{ i=0  \\ %
Q_{i}=R  \\ P_{i}=t_{1}}}^{f-1}k_{i}+2\tsum\limits_{\substack{ i=0  \\ %
Q_{i}=R  \\ P_{i}=t_{2}}}^{f-1}k_{i}+2\tsum\limits_{\substack{ i=0  \\ %
Q_{i}=Id  \\ P_{i}=t_{3}}}^{f-1}k_{i}+2\tsum\limits_{\substack{ i=0  \\ %
Q_{i}=Id  \\ P_{i}=t_{4}}}^{f-1}k_{i}.
\end{equation*}%
We now show that the $\left( 2,2\right) $ entry of $\tprod%
\limits_{i=0}^{2f-1}\left( Q_{i}P_{i+1}Q_{i+1}^{-1}\right) $ is the $p^{n},$
where 
\begin{equation}
n=2\tsum\limits_{\substack{ i=0  \\ Q_{i}=R  \\ P_{i}=t_{1}}}%
^{f-1}k_{i}+2\tsum\limits_{\substack{ i=0  \\ Q_{i}=R  \\ P_{i}=t_{2}}}%
^{f-1}k_{i}+2\tsum\limits_{\substack{ i=0  \\ Q_{i}=Id  \\ P_{i}=t_{3}}}%
^{f-1}k_{i}+2\tsum\limits_{\substack{ i=0  \\ Q_{i}=Id  \\ P_{i}=t_{4}}}%
^{f-1}k_{i}.  \label{n}
\end{equation}%
Since the matrices $Q_{i}P_{i+1}Q_{i+1}^{-1}$ are diagonal, and since $%
Q_{i+f}=Q_{i}$ and $P_{i+f}=P_{i}$ for all $i,$%
\begin{eqnarray*}
\tprod\limits_{i=0}^{2f-1}\left( Q_{i}P_{i+1}Q_{i+1}^{-1}\right)
=\tprod\limits_{\substack{ i=0  \\ Q_{i}=Id  \\ P_{i+1}=t_{4}}}^{f-1}\left(
Q_{i}P_{i+1}Q_{i+1}^{-1}\right) ^{2} &&\cdot \tprod\limits_{\substack{ i=0 
\\ Q_{i}=Id  \\ P_{i+1}=t_{3}}}^{f-1}\left( Q_{i}P_{i+1}Q_{i+1}^{-1}\right)
^{2}\cdot \\
\tprod\limits_{\substack{ i=0  \\ Q_{i}=Id  \\ P_{i+1}=t_{1}}}^{f-1}\left(
Q_{i}P_{i+1}Q_{i+1}^{-1}\right) ^{2}\cdot \tprod\limits_{\substack{ i=0  \\ %
Q_{i}=Id  \\ P_{i+1}=t_{2}}}^{f-1}\left( Q_{i}P_{i+1}Q_{i+1}^{-1}\right)
^{2} &&\cdot \tprod\limits_{\substack{ i=0  \\ Q_{i}=R  \\ P_{i+1}=t_{4}}}%
^{f-1}\left( Q_{i}P_{i+1}Q_{i+1}^{-1}\right) ^{2}\cdot \\
\tprod\limits_{\substack{ i=0  \\ Q_{i}=R  \\ P_{i+1}=t_{3}}}^{f-1}\left(
Q_{i}P_{i+1}Q_{i+1}^{-1}\right) ^{2}\cdot \tprod\limits_{\substack{ i=0  \\ %
Q_{i}=R  \\ P_{i+1}=t_{1}}}^{f-1}\left( Q_{i}P_{i+1}Q_{i+1}^{-1}\right) ^{2}
&&\cdot \tprod\limits_{\substack{ i=0  \\ Q_{i}=R  \\ P_{i+1}=t_{2}}}%
^{f-1}\left( Q_{i}P_{i+1}Q_{i+1}^{-1}\right) ^{2}.
\end{eqnarray*}%
We notice that when $Q_{i}=$ $Id$ and $P_{i+1}=t_{4},$ then by formula \ref%
{Qiiii}, $Q_{i+1}=R$ and $Q_{i}P_{i+1}Q_{i+1}^{-1}=\mathrm{diag}\left(
p^{k_{i+1}},1\right) .$ Therefore the product $\tprod\limits_{\substack{ i=0 
\\ Q_{i}=Id  \\ P_{i+1}=t_{4}}}^{f-1}\left( Q_{i}P_{i+1}Q_{i+1}^{-1}\right) $
has no contribution to the$\ \left( 2,2\right) $ entry of $%
\tprod\limits_{i=0}^{2f-1}\left( Q_{i}P_{i+1}Q_{i+1}^{-1}\right) .$
Similarly, the products%
\begin{equation*}
\tprod\limits_{\substack{ i=0  \\ Q_{i}=Id  \\ P_{i+1}=t_{1}}}^{f-1}\left(
Q_{i}P_{i+1}Q_{i+1}^{-1}\right) ,\tprod\limits_{\substack{ i=0  \\ Q_{i}=R 
\\ P_{i+1}=t_{3}}}^{f-1}\left( Q_{i}P_{i+1}Q_{i+1}^{-1}\right) \ \text{and}%
\tprod\limits_{\substack{ i=0  \\ Q_{i}=R  \\ P_{i+1}=t_{2}}}^{f-1}\left(
Q_{i}P_{i+1}Q_{i+1}^{-1}\right)
\end{equation*}%
have no contribution to the$\ \left( 2,2\right) $ entry of $%
\tprod\limits_{i=0}^{2f-1}\left( Q_{i}P_{i+1}Q_{i+1}^{-1}\right) .$ We now
compute the product $\tprod\limits_{\substack{ i=0  \\ Q_{i}=R  \\ %
P_{i+1}=t_{1}}}^{f-1}\left( Q_{i}P_{i+1}Q_{i+1}^{-1}\right) .$ Formula \ref%
{Qiiii} implies that if $Q_{i}=R\ $and $P_{i+1}=t_{1}\ $then $Q_{i+1}=R,$
therefore $Q_{i}P_{i+1}Q_{i+1}^{-1}=\mathrm{diag}\left( 1,p^{k_{i+1}}\right)
.$ Again, by formula \ref{Qiiii}, $Q_{i}=R$ and $P_{i+1}=t_{1}\ $is
equivalent to $Q_{i+1}=R$ and $P_{i+1}=t_{1}.$ Hence

\begin{equation*}
\tprod\limits_{\substack{ i=0  \\ Q_{i}=R  \\ P_{i+1}=t_{1}}}^{f-1}\left(
Q_{i}P_{i+1}Q_{i+1}^{-1}\right) =\tprod\limits_{\substack{ i=0  \\ Q_{i+1}=R 
\\ P_{i+1}=t_{1}}}^{f-1}\left( Q_{i}P_{i+1}Q_{i+1}^{-1}\right)
=\tprod\limits _{\substack{ i=0  \\ Q_{i}=R  \\ P_{i}=t_{1}}}^{f-1}\mathrm{%
diag}\left( 1,p^{k_{i+1}}\right)
\end{equation*}%
which contributes the fourth summand of the right hand side of equation \ref%
{n}. The claim made before formula \ref{n} follows arguing similarly for the
remaining cases. Hence$\ \mathrm{v}_{\mathrm{p}}\left( \mathrm{Nm}_{\varphi
}\left( \vec{\mu}\right) \right) =\tsum\limits_{\substack{ i=0  \\ z_{i}=0}}%
^{2f-1}k_{i}.$ By Proposition \ref{class thm} $\left( V_{\vec{k},\vec{0}}^{%
\vec{i}}\right) _{\mid G_{K_{2f}}}$ is reducible and $\left( \mathbb{D}%
_{2},\varphi \right) $ is a weakly admissible submodule, where $\mathbb{D}%
_{2}=\left( E^{2f}\right) \cdot \zeta _{2}.$ By \cite[proof of Prop. 4.3]%
{DO10} (or by a direct computation), 
\begin{equation}
\mathrm{Fil}^{\mathrm{j}}\mathbb{D}_{2}=\left\{ 
\begin{array}{l}
\ \ \ \ \mathbb{D}_{2}\text{\ \ \ \ \ \ \ \ \ \ \ \ \ \ \ \ if\ }~\text{\ }%
j\leq 0, \\ 
\left( E^{\mid \tau _{K_{2f}}\mid }\right) f_{I_{i,\vec{z}}}\zeta _{2}\text{
\ if\ \ }1+w_{i-1}\leq j\leq w_{i}\ \text{for all }i=0,1,...,t-1, \\ 
\ \ \ \ 0\ \ \ \ \text{\ \ \ \ \ \ \ \ \ \ \ \ \ \ if \ }j\geq 1+w_{t-1},%
\end{array}%
\right.  \label{whta}
\end{equation}%
where $I_{i.\vec{z}}=I_{i}\cap \{j\in \{0,1,...,2f-1\}:z_{j}=0\}.$ As in the
proof of Proposition \ref{prop1}, the labeled weight for the embedding $\tau
_{i}$ is $0$ if $z_{i}=1\ $and $-k_{i}\ $if $z_{i}=0.$ Next, we prove that
for $i=0,1,...,f-1,$%
\begin{equation}
z_{i}=z_{i+f}=\left\{ 
\begin{array}{l}
0\ \text{if }\ell _{i}=0, \\ 
1\ \text{if }\ell _{i}=k_{i}>0.%
\end{array}%
\right.  \label{zi,f}
\end{equation}%
This is done exactly as in Proposition \ref{prop1}, taking into account that
an even number of the coordinates of $\left( P_{1},P_{2},...,P_{f}\right) $
is of even type. We have $z_{i}=0$ for all $i$ if and only if $\ell _{i}=0$
for all $i$ if and only if $\left( \Pi _{1},\Pi _{2},...,\Pi _{f}\right) \in
C_{1}$ and $z_{i}=1$ for all $i$ if and only if $\ell _{i}=k_{i}>0$ for all $%
i$ if and only if $\left( \Pi _{1},\Pi _{2},...,\Pi _{f}\right) \in C_{2},$
cases excluded. Therefore neither of the summands of $V_{\vec{k}}^{\vec{i}}(%
\vec{0})$ is unramified. By the discussion above the labeled weights of $%
\mathbb{D}_{2}$ are $\left( -\ell _{0}^{\prime },-\ell _{1}^{\prime
},...,-\ell _{f-1}^{\prime },-\ell _{0}^{\prime },-\ell _{1}^{\prime
},...,-\ell _{f-1}^{\prime }\right) .$ By formula \ref{zi,f}, $\mathrm{v}_{%
\mathrm{p}}\left( \mathrm{Nm}_{\varphi }\left( \vec{\mu}\right) \right)
=\tsum\limits_{\substack{ i=0  \\ z_{i}=0}}^{2f-1}k_{i}=\tsum%
\limits_{i=0}^{2f-1}\ell _{i}^{\prime }.$ By Proposition \ref{rank 1 isom}\
and Lemma \ref{chic prop}, the corresponding crystalline character is%
\begin{equation*}
\psi =\chi _{e_{0}}^{\ell _{1}^{\prime }}\cdot \chi _{e_{1}}^{\ell
_{2}^{\prime }}\cdot \cdots \cdot \chi _{e_{f-2}}^{\ell _{f-1}^{\prime
}}\cdot \chi _{e_{f-1}}^{\ell _{0}^{\prime }}\cdot \chi _{e_{0}}^{\ell
_{1}^{\prime }}\cdot \chi _{e_{1}}^{\ell _{2}^{\prime }}\cdot \cdots \cdot
\chi _{e_{f-2}}^{\ell _{f-1}^{\prime }}\cdot \chi _{e_{f-1}}^{\ell
_{0}^{\prime }},
\end{equation*}%
If $V_{\vec{k}}^{\vec{i}}(\vec{0})$ is irreducible, then by Frobenius
reciprocity $V_{\vec{k}}^{\vec{i}}(\vec{0})=\mathrm{Ind}_{K_{2f}}^{K_{f}}%
\left( \psi \right) ,$ which is absurd by Corollary \ref{ti na to pw}. Hence 
$V_{\vec{k}}^{\vec{i}}(\vec{0})$ is reducible and contains an irreducible
constituent which restricts to $\psi .$ By Lemma \ref{chic prop}(iv), the
only choices are $\eta _{\pm 1}\cdot \chi _{e_{0}}^{\ell _{1}}\cdot \chi
_{e_{1}}^{\ell _{2}}\cdot \cdots \cdot \chi _{e_{f-2}}^{\ell _{f-1}}\cdot
\chi _{e_{f-1}}^{\ell _{0}},$ and we are done after twisting by $\eta _{\mp
1}.$ The rest of the proposition follows as in the proof of Proposition \ref%
{prop1}.
\end{proof}

\begin{theorem}
Theorem $\ref{reducibles0}$ holds.
\end{theorem}

\begin{proof}
Follows from Proposition \ref{se xreiaaaaaaaaaazomai!!!}, taking into
account Remark \ref{remark about types}.
\end{proof}

\begin{example}
Let $f=2,$ $\ell _{0}=0$ and $\ell _{1}=k_{1.}$ Let $\left( \Pi _{1},\Pi
_{0}\right) =\left( t_{1},t_{3}\right) $ with $c_{0}=c_{1}=1.$ Then after
possibly twisting by $\eta _{\pm 1},$ 
\begin{equation*}
V_{\vec{k}}^{\left( 1,3\right) }(\vec{0})\simeq \chi _{e_{0}}^{k_{1}}\oplus
\chi _{e_{1}}^{k_{0}}
\end{equation*}
\end{example}

\noindent In the next proposition we study closer the F-semisimple members
of this family, assuming that $c=1.$

\begin{proposition}
\label{an example}Assume that $V_{\vec{k}}^{\left( 1,3\right) }\left( \vec{%
\alpha}\right) $ is F-semisimple.

\begin{enumerate}
\item[(i)] $V_{\vec{k}}^{\left( 1,3\right) }\left( \ \vec{\alpha}\right) $
is irreducible if and only if $\alpha _{0}\alpha _{1}\neq 0,$ and is
non-induced for all but finitely many such $\vec{\alpha};$

\item[(ii)] $V_{\vec{k}}^{\left( 1,3\right) }\left( \vec{\alpha}\right) $ is
non-split reducible if and only if precisely one of the coordinates $\alpha
_{i}$ of $\vec{\alpha}$ is zero;

\item[(iii)] The families $\left\{ V_{\vec{k}}^{\left( 1,3\right) }\left(
\left( \alpha _{0},0\right) \right) ,~\alpha _{0}\in p^{m}\mathfrak{m}%
_{E}\setminus \{{0\}}\right\} $ and $\left\{ V_{\vec{k}}^{\left( 1,3\right)
}\left( \left( 0,\alpha _{1}\right) \right) ,~\alpha _{1}\in p^{m}\mathfrak{m%
}_{E}\setminus \{{0\}}\right\} $ are disjoint;

\item[(iv)] $V_{\vec{k}}^{\left( 1,3\right) }(\vec{0})$ is split-reducible.
\end{enumerate}
\end{proposition}

\begin{proof}
The weakly admissible filtered $\varphi $-module corresponding to $V_{\vec{k}%
}^{\left( 1,3\right) }\left( \vec{\alpha}\right) $ has Frobenius
endomorphism 
\begin{equation*}
\left( \varphi \left( \eta _{1}\right) ,\varphi \left( \eta _{2}\right)
\right) =\left( \eta _{1},\eta _{2}\right) \left( 
\begin{array}{cc}
\left( p^{k_{1}},1\right) & \left( 0,\alpha _{0}\right) \\ 
\left( \alpha _{1},0\right) & \left( 1,p^{k_{0}}\right)%
\end{array}%
\right)
\end{equation*}%
and filtration 
\begin{equation}
\mathrm{Fil}^{\mathrm{j}}\left( \mathbb{D}\right) =\left\{ 
\begin{array}{l}
(E\times E)\eta _{1}\oplus (E\times E)\eta _{2}\text{\ \ if }j\leq 0, \\ 
(E\times E)f_{I_{0}}(\vec{x}\eta _{1}+\vec{y}\eta _{2})\text{ \ \ \ if }%
1\leq j\leq w_{0}, \\ 
(E\times E)f_{I_{1}}(\vec{x}\eta _{1}+\vec{y}\eta _{2})\text{ \ \ \ if }%
1+w_{0}\leq j\leq w_{1}, \\ 
\ \ \ \ \ \ \ \ \ \ \ 0\text{ \ \ \ \ \ \ \ \ \ \ \ \ \ \ \ \ \ \ \ \ if }%
j\geq 1+w_{1},%
\end{array}%
\right.  \label{example13 fil}
\end{equation}%
with $\vec{x}=\left( -\alpha _{0},1\right) \ $and $\vec{y}=\left( 1,-\alpha
_{1}\right) .$ We diagonalize the matrix of Frobenius, arguing as in the
proof of Proposition $2.2$ in \cite{DO10}. The characteristic polynomial is $%
X^{2}-\left( p^{k_{0}}+p^{k_{1}}+\alpha _{0}\alpha _{1}\right)
X+p^{k_{0}+k_{1}},$ and we assume that $\left( p^{k_{0}}+p^{k_{1}}+\alpha
_{0}\alpha _{1}\right) ^{2}\neq 4p^{k_{0}+k_{1}}$ so that its roots $%
\varepsilon _{0}$ and $\varepsilon _{1}$ are distinct. We have the following
cases.

\noindent Case $\left( 1\right) .$ $\alpha _{0}\alpha _{1}\neq 0.$ We change
the basis to $\underline{\xi }=\left( \xi _{1},\xi _{2}\right) ,$ where%
\begin{eqnarray*}
\xi _{1}=\left( \left( \varepsilon _{0}-p^{k_{1}}-\alpha _{0}\alpha
_{1}\right) \alpha _{1},\ \frac{\alpha _{0}\left( \varepsilon
_{0}-\varepsilon _{1}\right) \left( \varepsilon _{0}-p^{k_{0}}\right) \left(
\varepsilon _{0}-p^{k_{0}}-\alpha _{0}\alpha _{1}\right) \left( \varepsilon
_{1}-p^{k_{1}}-\alpha _{0}\alpha _{1}\right) }{\left( 2\varepsilon
_{0}\varepsilon _{1}-p^{k_{0}}\varepsilon _{1}-p^{k_{1}}\varepsilon
_{0}-\alpha _{0}\alpha _{1}\varepsilon _{1}\right) \left( \varepsilon
_{1}-p^{k_{0}}-\alpha _{0}\alpha _{1}\right) }\right) \eta _{1} && \\
&& \\
+\left( \left( \varepsilon _{0}-p^{k_{1}}-\alpha _{0}\alpha _{1}\right)
\alpha _{1},\ \frac{\alpha _{0}\left( \varepsilon _{0}-\varepsilon
_{1}\right) \left( \varepsilon _{1}-p^{k_{0}}\right) \left( \varepsilon
_{0}-p^{k_{0}}-\alpha _{0}\alpha _{1}\right) \left( \varepsilon
_{1}-p^{k_{1}}-\alpha _{0}\alpha _{1}\right) }{\left( 2\varepsilon
_{0}\varepsilon _{1}-p^{k_{0}}\varepsilon _{1}-p^{k_{1}}\varepsilon
_{0}-\alpha _{0}\alpha _{1}\varepsilon _{1}\right) \left( \varepsilon
_{1}-p^{k_{0}}-\alpha _{0}\alpha _{1}\right) }\right) \eta _{2} &&
\end{eqnarray*}%
and%
\begin{eqnarray*}
\xi _{2}=\left( \left( \varepsilon _{1}-p^{k_{1}}-\alpha _{0}\alpha
_{1}\right) \left( \varepsilon _{0}-p^{k_{1}}\right) ,\ \frac{\alpha
_{0}^{2}\left( \varepsilon _{0}-\varepsilon _{1}\right) \left( \varepsilon
_{1}-p^{k_{1}}-\alpha _{0}\alpha _{1}\right) \left( \varepsilon
_{1}-p^{k_{0}}-\alpha _{0}\alpha _{1}\right) }{\left( 2\varepsilon
_{0}\varepsilon _{1}-p^{k_{0}}\varepsilon _{1}-p^{k_{1}}\varepsilon
_{0}-\alpha _{0}\alpha _{1}\varepsilon _{1}\right) \left( \varepsilon
_{1}-p^{k_{0}}-\alpha _{0}\alpha _{1}\right) }\right) \eta _{1} && \\
&& \\
+\left( \left( \varepsilon _{1}-p^{k_{1}}-\alpha _{0}\alpha _{1}\right)
\left( \varepsilon _{1}-p^{k_{1}}\right) ,\ \frac{\alpha _{0}^{2}\left(
\varepsilon _{0}-\varepsilon _{1}\right) \left( \varepsilon
_{1}-p^{k_{1}}-\alpha _{0}\alpha _{1}\right) \left( \varepsilon
_{1}-p^{k_{0}}-\alpha _{0}\alpha _{1}\right) }{\left( 2\varepsilon
_{0}\varepsilon _{1}-p^{k_{0}}\varepsilon _{1}-p^{k_{1}}\varepsilon
_{0}-\alpha _{0}\alpha _{1}\varepsilon _{1}\right) \left( \varepsilon
_{1}-p^{k_{0}}-\alpha _{0}\alpha _{1}\right) }\right) \eta _{2} &&
\end{eqnarray*}%
In the ordered basis $\underline{\xi },$%
\begin{equation*}
\varphi \left( \xi _{1}\right) =\left( 1,\ \varepsilon _{0}\right) \xi _{1}\ 
\text{and }\varphi \left( \xi _{2}\right) =\left( \lambda ,\ \frac{%
\varepsilon _{1}}{\lambda }\varepsilon _{1}\right) \xi _{2},
\end{equation*}%
where%
\begin{equation*}
\lambda =\allowbreak -\frac{\left( \varepsilon _{0}-p^{k_{1}}-\alpha
_{0}\alpha _{1}\right) }{\left( \varepsilon _{1}-p^{k_{1}}-\alpha _{0}\alpha
_{1}\right) }\cdot \frac{\left( \varepsilon _{1}-p^{k_{0}}-\alpha _{0}\alpha
_{1}\right) }{\left( \varepsilon _{0}-p^{k_{0}}-\alpha _{0}\alpha
_{1}\right) }\cdot \frac{\left( 2\varepsilon _{0}\varepsilon
_{1}-p^{k_{0}}\varepsilon _{1}-p^{k_{1}}\varepsilon _{0}-\alpha _{0}\alpha
_{1}\varepsilon _{1}\right) }{\left( 2\varepsilon _{0}\varepsilon
_{1}-p^{k_{0}}\varepsilon _{0}-p^{k_{1}}\varepsilon _{1}-\alpha _{0}\alpha
_{1}\varepsilon _{0}\right) },
\end{equation*}%
and the filtration is as in formula \ref{example13 fil}, with $\vec{x}\eta
_{1}+\vec{y}\eta _{2}$ replaced by $\xi _{1}+\xi _{2}.$ By Proposition \ref%
{class thm} $V_{\vec{k}}^{\left( 1,3\right) }\left( \ \vec{\alpha}\right) \ $%
is irreducible. Arguing as in the proof of Proposition \ref{prop1}(iv) we
see that the representations $V_{\vec{k}}^{\left( 1,3\right) }\left( \ \vec{%
\alpha}\right) $ are non-induced with the possibility of at most finitely
many exceptions.

\noindent Case $(2).$ $\alpha _{0}=0,$ $\alpha _{1}\neq 0.$ We argue as
above and see that in the ordered basis $\underline{\xi }=\left( \xi
_{1},\xi _{2}\right) ,$ where%
\begin{equation*}
\xi _{1}=\eta _{2}\ \text{and}\ \xi _{2}=\left( 1,\ \frac{p^{k_{0}}-p^{k_{1}}%
}{\alpha _{1}p^{k_{1}}}\right) \eta _{1}-\left( \frac{\alpha _{1}}{%
p^{k_{0}}-p^{k_{1}}},\ p^{k_{0}-k_{1}}\right) \eta _{2}
\end{equation*}%
we have 
\begin{equation*}
\varphi \left( \xi _{1}\right) =\left( 1,\ p^{k_{0}}\right) \xi _{1}\ \text{%
and }\varphi \left( \xi _{2}\right) =\left( \lambda \left( \alpha
_{1}\right) ,\ \frac{p^{k_{1}}}{\lambda \left( \alpha _{1}\right) }\right)
\xi _{2},
\end{equation*}%
with $\lambda \left( \alpha _{1}\right) =\alpha _{1}^{-1}\left(
p^{k_{0}}-p^{k_{1}}\right) .$ The filtration in this basis is given by
formula \ref{example13 fil}, with $\vec{x}\eta _{1}+\vec{y}\eta _{2}$
replaced by $\xi _{1}+\left( 0,1\right) \xi _{2}.$ By Proposition \ref{class
thm}, $V_{\vec{k}}^{\left( 1,3\right) }\left( \left( 0,\alpha _{1}\right)
\right) $ is reducible, non-split.

\noindent Case $(3).$ $\alpha _{1}=0,\alpha _{0}\neq 0.$ In the ordered
basis $\underline{\xi }=\left( \xi _{1},\xi _{2}\right) ,$ where%
\begin{equation*}
\xi _{1}=\eta _{2}-\left( \frac{p^{k_{1}}\alpha _{0}}{p^{k_{1}}-p^{k_{0}}},\ 
\frac{\alpha _{0}}{p^{k_{1}}-p^{k_{0}}}\right) \eta _{1}\ \text{and }\xi
_{2}=\left( \allowbreak \frac{\allowbreak \alpha _{0}p^{k_{0}}}{%
p^{k_{1}}-p^{k_{0}}},\ 1\right) \eta _{1},
\end{equation*}%
we have 
\begin{equation*}
\varphi \left( \xi _{1}\right) =\left( 1,\ p^{k_{0}}\right) \xi _{1}\ \text{%
and }\varphi \left( \xi _{2}\right) =\left( \lambda \left( \alpha
_{0}\right) ,\ \allowbreak \frac{p^{k_{1}}}{\lambda \left( \alpha
_{0}\right) }\right) \xi _{2},
\end{equation*}%
with $\lambda \left( \alpha _{0}\right) =\alpha _{0}^{-1}\left(
p^{k_{1}}-p^{k_{0}}\right) p^{k_{1}-k_{0}}.$ The filtration in the basis $%
\xi $ is given by formula \ref{example13 fil}, with $\vec{x}\eta _{1}+\vec{y}%
\eta _{2}$ replaced by $\left( 1,0\right) \xi _{1}+\xi _{2}.$ By Proposition %
\ref{class thm}, $V_{\vec{k}}^{\left( 1,3\right) }\left( \left( \alpha
_{0},0\right) \right) $ is reducible, non-split. By \cite[Proposition 7.1]%
{DO10} it follows that there are no isomorphisms between members of the
families $\left\{ V_{\vec{k}}^{\left( 1,3\right) }\left( \left( \alpha
_{0},0\right) \right) ,~\alpha _{0}\in p^{m}\mathfrak{m}_{E}\setminus \{{0\}}%
\right\} $ and $\left\{ V_{\vec{k}}^{\left( 1,3\right) }\left( \left(
0,\alpha _{1}\right) \right) ,~\alpha _{1}\in p^{m}\mathfrak{m}_{E}\setminus
\{{0\}}\right\} .$

\noindent Case $\left( 4\right) .$ $\alpha _{0}=\alpha _{1}=0.$ Then $%
\varphi \left( \eta _{1}\right) =\left( p^{k_{1}},1\right) \eta _{1}\ $and $%
\varphi \left( \eta _{2}\right) =\left( 1,p^{k_{0}}\right) \eta _{2},$ while
the filtration is as in formula \ref{example13 fil},with $\vec{x}=\left(
0,1\right) \ $and $\vec{y}=\left( 1,0\right) .$ Since $J_{\vec{x}}\cap J_{%
\vec{y}}=\varnothing ,$ Proposition \ref{class thm} implies that $V_{\vec{k}%
}^{\left( 1,3\right) }(\vec{0})$ is split-reducible.
\end{proof}

\begin{proposition}
\label{protash paradeigma}Let $0<\mathrm{v}_{\mathrm{p}}\left( \varepsilon
_{i}\right) <k_{0}+k_{1}$ with $\varepsilon _{0}\neq \varepsilon _{1}\ $such
that $\varepsilon _{0}\varepsilon _{1}=p^{k_{0}+k_{1}}$ and assume that $%
0\leq k_{i}\leq p-1.$ Define the families of filtered $\varphi $-modules $%
\mathbb{D}\left( \lambda \right) \ $with 
\begin{equation*}
\varphi \left( \eta _{1}\right) =\left( 1,\varepsilon _{0}\right) \eta
_{1},\ \varphi \left( \eta _{2}\right) =\left( \lambda ,\frac{\varepsilon
_{1}}{\lambda }\right) \eta _{2},
\end{equation*}%
and filtrations as in formula \ref{example f=2 filtration} with $\vec{x}=%
\vec{y}=\vec{1}.$ These are weakly admissible, irreducible filtered $\varphi 
$-modules, sharing the same characteristic polynomial$\ $and filtration. Let 
$V\left( \lambda \right) $ be the corresponding to $\mathbb{D}\left( \lambda
\right) $ crystalline representations of $G_{%
\mathbb{Q}
_{p^{2}}}.$
\end{proposition}

\begin{enumerate}
\item[(i)] If $\lambda =\frac{\varepsilon _{1}}{\varepsilon _{0}}\left( 
\frac{p^{k_{1}}\alpha -\varepsilon _{0}}{p^{k_{1}}\alpha -\varepsilon _{1}}%
\right) ,$ where $\alpha \in m_{E},$ then $\left( \overline{V\left( \lambda
\right) }_{\mid I_{%
\mathbb{Q}
_{p^{2}}}}\right) ^{ss}=\omega _{4,\bar{\tau}_{0}}^{-\left(
k_{0}+pk_{1}\right) }\oplus \omega _{4,\bar{\tau}_{0}}^{-\left(
k_{0}+pk_{1}\right) p^{2}}\ $and $\overline{V\left( \lambda \right) }$ is
irreducible;

\item[(ii)] If$\ \lambda =\left( \frac{\varepsilon _{1}}{\varepsilon _{0}}%
\right) ^{2}\left( \frac{p^{k_{1}}\alpha -\varepsilon _{1}}{p^{k_{1}}\alpha
-\varepsilon _{0}}\right) ,$ where $\alpha \in m_{E},$ then $\left( 
\overline{V\left( \lambda \right) }_{\mid I_{%
\mathbb{Q}
_{p^{2}}}}\right) ^{ss}=\omega _{4,\bar{\tau}_{0}}^{-\left(
pk_{1}+p^{2}k_{0}\right) }\oplus \omega _{4,\bar{\tau}_{0}}^{-\left(
pk_{1}+p^{2}k_{0}\right) p^{2}}$ and $\overline{V\left( \lambda \right) }$
is irreducible;

\item[(iii)] If $\lambda =1,$ then $\overline{V\left( \lambda \right) }$ is
reducible and $\overline{V\left( \lambda \right) }_{\mid I_{%
\mathbb{Q}
_{p^{2}}}}=\omega _{2,\bar{\tau}_{0}}^{-k_{1}}\oplus \omega _{2,\bar{\tau}%
_{0}}^{-pk_{0}}.\ \ $
\end{enumerate}

\begin{proof}
The common characteristic polynomial is $X^{2}-\left( \varepsilon
_{0}+\varepsilon _{1}\right) X+p^{k_{0}+k_{1}}.\ $Parts (i) and (ii) follow
from Examples \ref{touparadeigmatos} (i) and (iii) using the
\textquotedblleft standard parametrization\textquotedblright\ for the
families $V_{\vec{k},\vec{a}}^{\left( 1,2\right) }\ $and $V_{\vec{k},\vec{a}%
}^{\left( 1,4\right) },$ and taking into account that $m=0\ $and Proposition %
\ref{breuil prop}. Part (iii) follows from Proposition \ref{an example}(i)
and a little computation to prove that if $p^{k_{0}}+p^{k_{1}}+\alpha
_{0}\alpha _{1}=\varepsilon _{0}+\varepsilon _{1}$ and $\varepsilon
_{0}\varepsilon _{1}=p^{k_{0}+k_{1}},$ then $\lambda =1.$
\end{proof}

\end{document}